\documentclass[12pt]{amsart}
\usepackage{amsmath,amssymb,amsthm,amscd}
\usepackage{enumitem}
\usepackage{longtable,array}
\usepackage{fullpage}
\usepackage{hyperref}
\hypersetup{%
  colorlinks   = true,	
  urlcolor     = blue,
  linkcolor    = blue,
  citecolor    = red
}
\theoremstyle{plain}
\newtheorem{theorem}{Theorem}[section]
\newtheorem{proposition}[theorem]{Proposition}
\newtheorem{corollary}[theorem]{Corollary}
\newtheorem{lemma}[theorem]{Lemma}
\theoremstyle{remark}
\newtheorem{remark}[theorem]{Remark}
\allowdisplaybreaks
\AtBeginDocument{%
   \def\MR#1{}
}
\author{Adrian Clingher}
\address{Department of Mathematics and Statistics, University of Missouri--St. Louis, St. Louis, MO 63121}
\email{clinghera@umsl.edu}
\author{Andreas Malmendier}
\address{Department of Mathematics \& Statistics, Utah State University, Logan, UT 84322}
\email{andreas.malmendier@usu.edu}
\author{Brandon Williams}
\address{Institute for Mathematics, Heidelberg University, 69120 Heidelberg, Germany}
\email{bwilliams@mathi.uni-heidelberg.de}
\keywords{K3 surfaces, orthogonal modular forms, elliptic fibrations,
Borcherds products, F-theory/heterotic duality, E-string}
\subjclass[2020]{11F37, 11F55, 11E39, 14J15, 14J27, 14J28}
\begin{document}
\title{Orthogonal modular forms and E-string limits for $U\oplus
  D_8(-1)$-polarized K3 surfaces}
\begin{abstract}
We study $U\oplus D_8(-1)$-polarized K3 surfaces through their
genus-one fibrations and orthogonal modular forms.  Starting from
Vinberg's scroll model, we construct an alternate genus-one pencil and
its relative Jacobian, obtaining a weighted-homogeneous polynomial map
to the coefficient space of the Hashimoto--Ueda family. We identify this map with Sakai's
Fourier--Jacobi reconstruction and analyze the associated Type~II
Kulikov limits. Their rescaled rational elliptic components realize
the E-string Seiberg--Witten curves with Wilson-line parameters. Our
construction also gives equations for the $E_7$ and $E_6$ modular
families and for their Borcherds divisors.  When the K3
families are interpreted as eight-dimensional F-theory vacua, our
results provide a uniform description of their F-theory/heterotic
boundary limits in terms of orthogonal modular forms, genus-one fibrations, 
and Baily--Borel boundary geometry.
\end{abstract}
\maketitle
\section{Introduction}
Moduli spaces of lattice-polarized K3 surfaces provide a setting in
which period maps, automorphic forms, explicit algebraic models, and
string dualities can be studied simultaneously. For several such
families, the coefficients of a distinguished elliptic or genus-one
fibration are generators of the corresponding ring of orthogonal
modular forms, while their modular Jacobian detects the reflective
ramification divisor
\cite{ClingherDoran,ClingherMalmendierWilliams2025}.  A classical
example is given by the Igusa generators attached to the Jacobians of
smooth genus-two curves and their associated Kummer surfaces.  For
their Shioda--Inose partners, namely the family of $U \oplus E_8(-1)
\oplus E_7(-1)$-polarized K3 surfaces, the even generators occur as
Weierstrass coefficients, and the odd generator $\chi_{35}$ is the
modular Jacobian of the even generators
\cite{GritsenkoNikulinIgusa,KawaiThreshold}.  The cusp form
$\chi_{10}$ cuts out the Humbert surface $\mathcal H_1$, and
$\chi_{35}$ cuts out the full reflective divisor $\mathcal
H_1\cup\mathcal H_4$; its Borcherds product is the denominator of the
corresponding automorphic correction
\cite{GritsenkoNikulinIgusa,KawaiThreshold,KawaiIgusa}.  In
eight-dimensional F-theory/heterotic duality, the same general
principle identifies the complex-structure parameters of an
elliptically fibered K3 surface with heterotic Narain moduli
\cite{VafaFTheory,MorrisonVafaI,MorrisonVafaII,ClingherMorganFTheory}.
\par This article establishes an analogous, but geometrically richer,
description for K3 surfaces polarized by the rank-ten lattice
\[
 M_D=U\oplus D_8(-1).
\]
Their period lattice is $T_D=U^{\oplus2}\oplus D_8(-1)$, and the
associated type-IV modular variety has dimension ten. The lattice
$M_D$ has two orbits of primitive isotropic vectors, of divisibility
one and two. They give two genus-one pencils on the same K3 surface;
under Dolgachev mirror symmetry, the same arithmetic distinction
labels the two one-dimensional Baily--Borel boundary components
incident to a fixed cusp \cite{DolgachevMirror,Scattone}. In
\cite{ClingherMalmendierWilliams2026}, we obtained the standard
Jacobian elliptic model associated with the divisibility-one pencil
\begin{equation}
\label{eq:H+D8-weierstrass_intro}
\begin{split}
X_{F,G,H}\colon \quad y^2=x^3 &+(u^3+uF_4+F_6)x^2 +(u^2G_8+uG_{10}+G_{12})x
\\ &+u^5H_8+u^4H_{10}+u^3H_{12} +u^2H_{14}+uH_{16}+H_{18},
\end{split}
\end{equation}
whose coefficients generate the scalar modular-form ring. As we show,
this model can be related explicitly to Vinberg's scroll construction
\cite{Vinberg2010,Vinberg2018,ClingherMalmendierWilliams2026}. On the
character cover of the moduli space, we have $H_8=-H_4^2$. In the
scroll model, a choice of sign for $H_4$ selects one of two lines in a
distinguished ruling plane; that line determines the divisibility-two
genus-one pencil, which has a bisection but no section on a very
general member.
\par Our first main result is the construction of this alternate
pencil as a family of binary quartics and of its relative
Jacobian. From the resulting invariant-theoretic formulas we derive 
a weighted-homogeneous polynomial map from the Vinberg coefficient
cone to the coefficient cone of the \(U\oplus E_8(-1)\)-polarized
Hashimoto--Ueda family \cite{HashimotoUeda}. 
In rational period coordinates, this
morphism is induced by the index-two inclusion $D_8\subset E_8$ and is
the identity on the common type-IV domain.  Using the Eisenstein
generators of Dieckmann, Krieg, and Woitalla
\cite{DieckmannKriegWoitalla}, we then compare the morphism with
Sakai's reconstruction of the F-theory K3 equation from the E-string
Seiberg--Witten curve \cite{SakaiFTheoryK3}.  We prove that, after the
stated Eisenstein normalization and weighted coordinate change,
Sakai's conjectural modular equation is the Hashimoto--Ueda family and
gives the geometric inverse period map on the common dense moduli
locus.  Our comparison determines all meromorphic $W(E_8)$-invariant
Jacobi coefficients of the E-string curve from the leading $D_8$
Fourier--Jacobi data.
\par Our second main result identifies the corresponding Type~II
boundary geometry. The two one-dimensional boundary components map to
the two irreducible components of the non-K3 coefficient locus. The
standard and alternate pencils extract different extremal rational
elliptic surfaces. At the common Fourier--Jacobi cusp, keeping the
leading transverse terms of the alternate pencil produces a rescaled
binary-quartic rational elliptic surface whose relative Jacobian is
the E-string surface containing the Wilson-line parameters, and an
unrescaled component having a distinguished \(\mathrm{II}^*\)
fiber. At the other boundary curve, an explicit smoothing of the
standard pencil produces the same E-string surface and an unrescaled
component having a distinguished \(\mathrm I_4^*\) fiber. In each
degeneration the two rational components meet along the anticanonical
elliptic fiber, whose modulus is defined by the Baily--Borel boundary. This stable-degeneration geometry realizes the
underlying eight-dimensional F-theory/heterotic string duality
\cite{KulikovDegenerations,PerssonPinkham,FriedmanSmoothings,
  ClingherMorganFTheory}. The physical interpretation uses the
E-string theory and its torus compactification
\cite{GanorHanany,SeibergWitten6D,MinahanNemeschanskyVafaWarner,
  EguchiSakaiEString,EguchiSakaiEStringRevisited}.
\par Our construction can also be applied to the index-two inclusions
\[
 A_7\subset E_7, \qquad A_5\oplus A_1\subset E_6.
\]
We obtain universal equations for K3 surfaces with generic
N\'eron--Severi lattices
\[
 U\oplus E_8(-1)\oplus A_1(-1) \quad\text{and}\quad U\oplus
 E_8(-1)\oplus A_2(-1).
\]
The coefficients of the corresponding Weierstrass equations have
exactly the weights of the free generators in the Wang--Williams
theorem \cite{WangWilliamsJacobian}. We also identify the exceptional
enhancement divisors with explicit Borcherds products $\Phi_{44,E_7}$
and $\Phi_{45,E_6}$, and show how the former cuts out the $E_6$
Heegner divisor in the $E_7$ modular variety. These results complete
the exceptional $E_8$, $E_7$, and $E_6$ part of the explicit ADE
coefficient models.
\par Two restrictions provide further tests of our construction.
First, we show that the Witt restriction of the Hashimoto--Ueda
universal family is the Inose family associated with the unordered
pair \((E_{\tau_1},E_{\tau_2})\) of elliptic curves. Under this
restriction, the one-dimensional Baily--Borel boundary \(d=0\) maps
onto the Hashimoto--Ueda non-K3 locus. Second, restriction to
Kond\=o's five-dimensional ball quotient expresses the Vinberg
coefficients as explicit covariants of eight ordered points on
\(\mathbb P^1\) and produces constant-coupling F-theory models
\cite{KondoEightPoints,DolgachevOrtland,HowardMillsonSnowdenVakil}.
\par The weight-four character form \(H_4\) links the ramification of
the modular coefficient space to the choice of the alternate genus-one
fibration and has a further interpretation in terms of Borcherds'
Enriques \(\Phi\)-function. For the anti-invariant lattice of the K3
cover of an Enriques surface, dualizing and rescaling gives
\begin{equation}
\label{eq:intro-enriques-dual}
L_{\mathrm{Enr}}^\vee(2)
\cong U^{\oplus2}\oplus D_8(-1)=T_D.
\end{equation}
Thus \(\Phi_{\mathrm{Enr}}\) is naturally an automorphic form on the
same type-IV domain. With the character, cusp coordinates, and
generator normalizations fixed as in
Proposition~\ref{prop:Enriques-normalization}, we prove
$\Phi_{\mathrm{Enr}}=\tfrac29H_4$.  The same weight-four automorphic
form therefore governs the reflective branch divisor in the K3
coefficient model, Yoshikawa's equivariant analytic torsion, and the
genus-one Enriques theory appearing in Oberdieck's proof of the
Klemm--Mari\~no formula
\cite{BorcherdsEnriques,GritsenkoNikulinEnriques,KondoEnriques,
  YoshikawaII,OberdieckEnriques}.
\par This paper is structured as follows. Section~\ref{sec:basics}
reviews some facts about two-elementary K3 surfaces, Yoshikawa's
equivariant analytic torsion, and Oberdieck's proof of the
Klemm--Mari\~no formula, and identifies the normalized
Borcherds--Enriques form with $H_4$.  Section~\ref{sec:Vinberg}
constructs the coefficient map from Vinberg's scroll and alternate
pencil to Hashimoto--Ueda and proves the comparison with
Sakai. Section~\ref{sec:self-mirror-boundary} computes the Type~II limits and
E-string surfaces. 
Section~\ref{sec:E7-E6-Jacobians} gives the remaining
exceptional universal K3 models, their Borcherds equations, and E-string surfaces.
Section~\ref{sec:restrictions} gives the Witt and
ball-quotient restrictions and their F-theory interpretations.  The
appendices collect the root-system equations, coefficient
dictionaries, and explicit Jacobi and Eisenstein formulas.
\section*{Acknowledgements}
B.W. acknowledges the support of the DFG through the
Collaborative Research Centre TRR 326 ``Geometry and Arithmetic of
Uniformized Structures'', project number 444845124. A.M. acknowledges support from the Simons Foundation through grant
no.~202367.  This material is based upon work supported by the National Science Foundation under Grant No.~DMS-2424556 while A.C. and A.M. were in residence at the Institute for Computational and Experimental Research in Mathematics (ICERM) during the semester program ``Computations on K3 Surfaces
and Related Varieties''. Many of the
computations in this paper were carried out using the SageMath
computer algebra system.  
\section{Two-elementary K3 surfaces}
\label{sec:basics}
Let \((X,\iota)\) be a K3 surface with a non-symplectic involution,
\(\iota^*\omega_X=-\omega_X\).  Its invariant lattice
\begin{equation*} M=H^2(X,\mathbb{Z})^\iota \end{equation*}
is primitive, hyperbolic, and two-elementary.  Write
\(T_M=M^\perp\subset\Lambda_{\mathrm{K3}}\), where
\begin{equation*} \Lambda_{\mathrm{K3}}=U^{\oplus3}\oplus E_8(-1)^{\oplus2}. \end{equation*}
After choosing a connected component, the period domain and its
discriminant arrangement are
\begin{equation*} \Omega_{T_M}^+ =\{[\omega]:(\omega,\omega)=0, (\omega,\bar\omega)>0\}^+, \qquad \mathcal D_{T_M} =\bigcup_{\substack{\delta\in T_M\\ \delta^2=-2}}\delta^\perp. \end{equation*}
For the appropriate arithmetic group \(\Gamma_M\), the moduli space is
the open quotient
\begin{equation*} \mathcal M_M^\circ =\Gamma_M\backslash(\Omega_{T_M}^+\setminus\mathcal D_{T_M}). \end{equation*}
Here \(\mathcal M_M^\circ\) parametrizes isomorphism classes of pairs
\((X,\iota)\) with \(H^2(X,\mathbb Z)^\iota\cong M\).
The superscript \(\circ\) denotes the ample locus, where the invariant
copy of \(M\) contains an ample class; additional algebraic classes
are allowed, and \(\operatorname{NS}(X)\cong M\) holds for a very
general member \cite[Section~2]{YoshikawaII}.
The divisor \(\mathcal D_{T_M}\) is a Heegner divisor inside the
orthogonal modular variety; its general point represents a K3 surface
on which an additional \((-2)\)-class becomes algebraic
\cite{NikulinIntegral,NikulinHyperbolic,YoshikawaII}. Write
\((r,\ell,\delta)\) for the Nikulin invariants of \(M\), i.e., the
rank of $M$, the length of its discriminant group $A_M=M^\vee/M$, and
its parity.  Apart from the two exceptional cases, the fixed locus of
the involution is
\begin{equation}
\label{eq:fixed-locus-g-k}
X^\iota=C_g\sqcup R_1\sqcup\cdots\sqcup R_k, \qquad
g=11-\frac{r+\ell}{2},\quad k=\frac{r-\ell}{2}.
\end{equation}
For \(M=U(2)\oplus E_8(-2)\) it is empty, and for \(M=U\oplus
E_8(-2)\) it is the disjoint union of two elliptic curves.
\par If the fixed locus is nonempty, the quotient \(q:X\to\bar
Y=X/\iota\) is a double cover of a rational surface, branched over a
divisor \(B\in|-2K_{\bar Y}|\).  Equivalently, \((\bar Y,\tfrac12B)\)
is log Calabi--Yau.  In the range of the Alexeev--Nikulin
correspondence, where the fixed locus contains a curve of genus at
least two, the appropriate contractions in the del Pezzo--Nikulin
(DPN) construction give a log del Pezzo surface \(S\) of index at most
two, with branch divisor \(C\in|-2K_S|\); the inverse construction is
the minimal resolution of the corresponding double cover; see
\cite{AlexeevNikulin}.
\par We next discuss certain genus-one A- and B-model quantities associated
with two-elementary K3 surfaces.
\subsection{Yoshikawa's analytic torsion}
\label{subsec:alexeev-nikulin-chain-amodel}
Yoshikawa associates to a two-elementary K3 surface of type \(M\) an
equivariant analytic-torsion invariant \(\tau_M(X,\iota)\).  It
combines the equivariant Ray--Singer torsion of \(X\) with fixed-curve
and volume terms and is independent of the chosen \(\iota\)-invariant
K\"ahler metric.  Let \(\pi_M\) denote the period map for the
type-\(M\) moduli problem. Yoshikawa's automorphy theorem gives an
integer $\nu(M)>0$ and an automorphic form $\Phi_M$ such that
\begin{equation}
\label{eq:yoshikawa-divisor}
\operatorname{div}(\Phi_M)=\nu(M)\mathcal D_{T_M},
\end{equation}
and
\begin{equation}
\label{eq:yoshikawa-torsion}
\tau_M(X,\iota) =\bigl\|\Phi_M(\pi_M(X,\iota))\bigr\|^{-1/(2\nu(M))}.
\end{equation}
Here the group, character, and Petersson metric are those attached to
the type-\(M\) moduli problem.  When the fixed curve has positive
genus, \(\Phi_M\) is naturally an automorphic form of biweight, with
values in the pullback of a Siegel Hodge bundle
\cite{FangLuYoshikawa,MaYoshikawaIV,YoshikawaI,YoshikawaII}.
\par If \(s\) is a local equation on the period domain for a reduced
component along which the divisor multiplicity of \(\Phi_M\) equals
\(\nu(M)\), then \(\Phi_M=u s^{\nu(M)}\) with \(u\) a unit in a local
holomorphic trivialization. If the metric remains bounded and
nonvanishing, then
\begin{equation}
\label{eq:local-yoshikawa-singularity}
\log\tau_M=-\frac14\log|s|^2+O\!\left(1\right).
\end{equation}
Yoshikawa's invariant \(\tau_M\) can be regarded as a
metric-independent, automorphically completed regularized functional
determinant attached to \((X,\iota)\). Motivated by this
interpretation, we use the notation
\begin{equation} 
\label{eqn:free_energy}
F^{B}_{1,M} := \log \tau_M = -\frac{1}{2\nu(M)} \log \|\Phi_M\|
\end{equation}
for the genus-one B-model free energy. The divisor \(\mathcal
D_{T_M}\) is the locus where additional \((-2)\)-classes in the
anti-invariant lattice become algebraic.  From the physical point of
view, the divergence of \(\tau_M\) near \(\mathcal D_{T_M}\) is
analogous to a one-loop threshold singularity caused by additional
massless BPS states. The automorphic form \(\Phi_M\) then plays a role
analogous to the modular functions governing these one-loop threshold
corrections; see \cite{HarveyMoore,KawaiThreshold,YoshikawaII}.
\medskip
\par The six even hyperbolic two-elementary lattices of rank ten and
parity zero form the following chain of full-rank embeddings of index
two:
\begin{equation}
\label{eq:AN-chain}
\begin{aligned}
U(2)\oplus E_8(-2) &\hookrightarrow U\oplus E_8(-2) \cong U(2)\oplus N
\\ &\hookrightarrow U\oplus N \cong U(2)\oplus D_4(-1)^{\oplus 2}
\\ &\hookrightarrow U\oplus D_4(-1)^{\oplus 2} \cong U(2)\oplus
D_8(-1) \\ &\hookrightarrow U\oplus D_8(-1) \cong U(2)\oplus E_8(-1)
\\ &\hookrightarrow U\oplus E_8(-1).
\end{aligned}
\end{equation}
Here \(N\) is the negative-definite Nikulin lattice, namely the even
index-two overlattice of \(A_1(-1)^{\oplus8}\) obtained by adjoining
the half-sum of the eight standard roots. Its discriminant group has
length six. We set
\[
\begin{aligned}
K_0&=E_8(-2),& K_1&=N,& K_2&=D_4(-1)^{\oplus2},& K_3&=D_8(-1),&
K_4&=E_8(-1).
\end{aligned}
\]
We denote the six lattices in \eqref{eq:AN-chain} by
\(M_0,\ldots,M_5\). For \(0\leq i\leq4\), we have the isometries
\[
M_i\cong U(2)\oplus K_i, \qquad M_{i+1}\cong U\oplus K_i.
\]
For \(0\leq i\leq3\), the latter lattice then also has the
presentation
\[
U\oplus K_i\cong U(2)\oplus K_{i+1}.
\]
If \(\{e,f\}\) is an isotropic basis of \(U\), with \(e^2=f^2=0\) and
\((e,f)=1\), then the \(i\)-th arrow is the identity on \(K_i\)
together with the inclusion
\[
U(2)\cong\mathbb Z e\oplus\mathbb Z(2f) \hookrightarrow \mathbb Z
e\oplus\mathbb Z f=U.
\]
Thus each arrow is a nonprimitive inclusion \(M_i\hookrightarrow
M_{i+1}\) with
\[
M_{i+1}/M_i\cong\mathbb Z/2\mathbb Z.
\]
Equivalently, \(M_{i+1}\) is an even overlattice of \(M_i\) of index
two, and the discriminant length \(\ell(M_i)=10-2i\) decreases by two
at each step.
\begin{remark}
\label{rem:tower}
The arrows in~\eqref{eq:AN-chain} are embeddings of abstract lattices.
They are not primitive lattice-polarization enhancements on one fixed
K3 surface. Geometrically, each arrow sends an index-two genus-one fibration
to its relative Jacobian.
In turn, the change between the Jacobian and genus-one fibrations realizing $U\oplus
K_i\cong U(2)\oplus K_{i+1}$ is described by a neighbor
transformation: one replaces the original primitive isotropic fiber
class by another primitive isotropic class in the same N\'eron--Severi
lattice. Here the two fiber classes have intersection number two; this
is referred to as a two-neighbor construction.  Explicit models for
the first transitions appeared in \cite{ClingherMalmendierCHL}; the
last two are described in Sections~\ref{sec:Vinberg} and
Appendix~\ref{App:2NS}.
\end{remark}
\par For every \(M_i\), Yoshikawa's theorem in \cite{YoshikawaI}
yields an automorphic form $\Phi_{M_i}$ with
\begin{equation*} \operatorname{div}(\Phi_{M_i}) = \nu(M_i)\mathcal D_{T_{M_i}}, \quad \text{and} \quad \tau_{M_i} = \|\Phi_{M_i}\|^{-1/(2\nu(M_i))}. \end{equation*}
The examples of Borcherds products in Yoshikawa's list are indexed by
the \emph{period lattice} \(L\), not by the invariant lattice.  When
\(L\) is the orthogonal complement of an invariant K3 lattice \(M\),
unimodularity of the K3 lattice gives \(\ell(L)=\ell(M)\), where
\(\ell(L)\) is the length of \(L^\vee/L\); see
\cite[Examples~8.8--8.13]{YoshikawaII}. With Yoshikawa's
normalization, the relevant examples are:
\begin{enumerate}[label=\textup{(\roman*)}]
\item For \(L=U(2)^{\oplus2}\oplus E_8(-2)\), one has \(\ell(L)=12\)
  and weight zero.  This lattice has no primitive embedding into the
  K3 lattice, so the example is non-geometric for K3 surfaces; the
  corresponding Borcherds lift is constant \cite{NikulinIntegral}.
\item For \(L=U\oplus U(2)\oplus E_8(-2)\), one has \(\ell(L)=10\) and
  weight four. This is the Enriques period lattice and the lift is
  Borcherds' Enriques \(\Phi\)-function
  \cite{BorcherdsProducts,BorcherdsEnriques,BorcherdsGrassmannians,
    FreitagSalvatiManni,KondoEnriques,ScheithauerFakeMonster,YoshikawaI}.
\item For \(L=U^{\oplus2}\oplus E_8(-2)\), one has \(\ell(L)=8\) and
  weight twelve.  The input is the theta series of the Barnes--Wall
  lattice divided by \(\eta^{24}\), and the product is obtained by
  restricting Borcherds' weight-twelve form on \(\mathrm O(2,26)\)
  \cite{BorcherdsGrassmannians}.
\item For \(L=U\oplus U(2)\oplus D_4(-1)^{\oplus2}\), one has
  \(\ell(L)=6\) and weight twenty-eight.  This is the orthogonal model
  behind Kond\=o's moduli of eight points on \(\mathbb{P}^1\).  The
  fifteenth power of this weight-28 lift is the product of the 105
  weight-four additive lifts indexed by tableaux
  \cite{BorcherdsGrassmannians,FreitagSalvatiManni,KondoEightPoints}.
\item For \(L=U^{\oplus2}\oplus E_8(-1)\), one has \(\ell(L)=0\) and
  weight \(252\).  The Borcherds input is \(E_4^2/\eta^{24}\).  This
  is the form that also appears in the Harvey--Moore one-loop
  calculation and, in the coefficient model below, as the modular
  Jacobian of the Hashimoto--Ueda generators
  \cite{BorcherdsGrassmannians,HarveyMoore,HashimotoUeda}.
\item For \(L=U^{\oplus2}\oplus D_4(-1)\), the lift has weight \(72\)
  and equals the Freitag--Hermann form, a product of 36 theta
  functions \cite{FreitagHermann}.  This lattice has rank eight; hence
  this last example is not a rank-twelve complement in the rank-ten
  tower.
\end{enumerate}
Examples~8.9--8.12 of \cite{YoshikawaII} correspond, respectively, to
\(M_0,M_1,M_2,M_5\) in~\eqref{eq:AN-chain}; the rank-twelve
complements for \(M_3\) and \(M_4\) do not occur in this list. The
list contains three cases central to this article: the Enriques form,
Kond\=o's eight-point form, and the weight-\(252\) unimodular form.
\par The six examples above describe the pure Borcherds lift
\(\Psi_L(\,\cdot\,,F_L)\); by themselves, they do not give the
complete torsion form. Here \(F_L\) is the vector-valued modular input
to the Borcherds lift, \(J_M\) is the period map of the fixed locus to
the Siegel modular variety of degree \(g\), where \(g\) is the sum of
the genera of its components, \(\Upsilon_g\) is the corresponding
Siegel modular form, and \(C_M\) is a positive normalization constant.
For a non-exceptional rank-ten type, here \(M_i\) with \(2\leq
i\leq5\), the Ma--Yoshikawa factorization also contains the pullback
of the Siegel factor. In their normalization, we have
\begin{equation}
\tau_M^{-(2^g+2)(2^g-1)} =C_M\,
\bigl\|\Psi_L(\,\cdot\,,(2^{g-1}+1)F_L)\bigr\|\, J_M^*\|\Upsilon_g\|.
\label{eq:Ma-Yoshikawa-factorization}
\end{equation}
The type \(M_1=U\oplus E_8(-2)\) is exceptional: its fixed locus is
the disjoint union of two elliptic curves and lies in the decomposable
locus of \(\mathcal A_2\). Theorem~9.4 of \cite{MaYoshikawaIV} does
not apply to this type, so the case must be treated separately.
\subsection{Mirror symmetry and the coefficient
\texorpdfstring{\(H_4\)}{H4}}
\label{sec:H8-Enriques}
The tower in~\eqref{eq:AN-chain} starts with the Enriques-cover
family, which can be realized as follows. Set \(Q=\mathbb
P^1_{(x_0:x_1)}\times\mathbb P^1_{(y_0:y_1)}\), and let \(F_{4,4}\in
H^0\bigl(Q,\mathcal O_Q(4,4)\bigr)\) define a smooth biquartic
\(B=\{F_{4,4}=0\}\subset Q\) that is invariant under
\[
\sigma\colon Q\longrightarrow Q,\qquad
\bigl((x_0:x_1),(y_0:y_1)\bigr)\longmapsto
\bigl((x_0:-x_1),(y_0:-y_1)\bigr)
\]
and avoids the four fixed points of \(\sigma\). The double cover
\[
\pi\colon X=\bigl\{w^2=F_{4,4}(x_0,x_1;y_0,y_1)\bigr\}
\subset\operatorname{Tot}\bigl(\mathcal O_Q(2,2)\bigr) \longrightarrow
Q
\]
is a smooth K3 surface with deck involution
\(\kappa(x,y,w)=(x,y,-w)\). Geometrically, a very general member
admits two genus-one fibrations, one associated with each ruling, and
each has a distinguished bisection and no section
\cite{GritsenkoHulek,KondoEnriques}. The involution \(\sigma\) has the
lift
\[
\iota_{\mathrm{Enr}}\colon X\longrightarrow X,\qquad
(x,y,w)\longmapsto\bigl(\sigma(x,y),-w\bigr).
\]
This lift is fixed-point-free. Indeed, a fixed point would have to lie
above one of the four fixed points of \(\sigma\); over such a point,
the assumption \(B\cap\operatorname{Fix}(\sigma)=\varnothing\) implies
\(w\ne0\), whereas \(\iota_{\mathrm{Enr}}\) sends \(w\) to
\(-w\). Hence \(Y=X/\langle\iota_{\mathrm{Enr}}\rangle\) is an
Enriques surface, and the invariant and anti-invariant lattices are as
follows:
\begin{equation}
\label{enq:Lenr}
\begin{aligned}
M_{\mathrm{Enr}} &:=H^2(X,\mathbb Z)^{\iota_{\mathrm{Enr}}} \cong
U(2)\oplus E_8(-2),\\ L_{\mathrm{Enr}} &:=H^2(X,\mathbb
Z)^{\iota_{\mathrm{Enr}}=-1} \cong U\oplus U(2)\oplus E_8(-2).
\end{aligned}
\end{equation}
Conversely, every Enriques surface has an \'etale K3 universal cover
with these invariant and anti-invariant lattices. In
\cite{GritsenkoHulek}, the moduli space of unpolarized Enriques
surfaces is described as the quotient of the complement of the
\((-2)\)-hyperplanes in the type-IV domain \(\mathcal
D(L_{\mathrm{Enr}})\) by \(\mathrm O^+(L_{\mathrm{Enr}})\). In this
case, Yoshikawa's automorphic form \(\Phi_{M_{\mathrm{Enr}}}\) agrees,
up to a nonzero constant, with Borcherds' Enriques \(\Phi\)-function
\(\Phi_{\mathrm{Enr}}\). Hence the B-model genus-one
amplitude~\eqref{eqn:free_energy} is governed by
\(\Phi_{\mathrm{Enr}}\).
\par On the other hand, the A-model counts curves in the local
Enriques Calabi--Yau threefold
\begin{equation*} X_Y^{\mathrm{loc}}:=\operatorname{Tot}(K_Y). \end{equation*}
If \(\pi:X_Y^{\mathrm{loc}}\to Y\) is the projection, then
\begin{equation*} K_{X_Y^{\mathrm{loc}}} \cong\pi^*(K_Y\otimes K_Y^{-1})\cong\mathcal O, \end{equation*}
so this is a non-compact Calabi--Yau threefold.  For \(g\geq1\) and a
nonzero curve class
\begin{equation*} \beta\in H_2(Y,\mathbb Z)/\mathrm{tors}, \end{equation*}
the local Enriques Gromov--Witten invariant is
\begin{equation*} N_{g,\beta} = \int_{[\overline M_g(Y,\beta)]^{\mathrm{vir}}} (-1)^{g-1}\lambda_{g-1}. \end{equation*}
Here $\mathrm{tors}$ denotes the torsion subgroup, $\overline
M_g(Y,\beta)$ is the moduli stack of unpointed genus-$g$ stable maps
to $Y$ in class $\beta$, $[\overline M_g(Y,\beta)]^{\mathrm{vir}}$ is
its virtual fundamental class, and $\lambda_{g-1}$ is the $(g-1)$st
Chern class of the Hodge bundle. The genus-one A-model partition
function is
\begin{equation*} Z^{A}_{1,\mathrm{Enr}} = \exp \left( \sum_{\beta\neq 0} N_{1,\beta}q^\beta \right). \end{equation*}
In \cite[Corollary~1.2]{OberdieckEnriques} Oberdieck proves the
genus-one Klemm--Mari\~no formula from~\cite{KlemmMarinoEnriques},
namely that the genus-one Enriques partition function is
\begin{equation}
\label{eq:Oberdieck-genus-one-product}
Z^{A}_{1,\mathrm{Enr}} := \exp\left( \sum_{\beta\neq
  0}N_{1,\beta}q^\beta \right) = \prod_{\beta>0} \left(
\frac{1+q^\beta}{1-q^\beta} \right)^{a(\beta^2/2)},
\end{equation}
where the coefficients $a(n)$ are defined by
\begin{equation}
\label{eq:a-n-Enriques}
\sum_{n\geq 0}a(n)q^n = \prod_{n\geq 1} \frac{(1+q^n)^8}{(1-q^n)^8} =
1+16q+144q^2+960q^3+5264q^4+\cdots.
\end{equation}
In the paragraph following \cite[Corollary~1.2]{OberdieckEnriques}, he
then identifies \eqref{eq:Oberdieck-genus-one-product} with the
Fourier expansion
\begin{equation}
\label{eq:Oberdieck-Phi-expansion}
Z^{A}_{1,\mathrm{Enr}} = \operatorname{FE}_{\mathfrak
  c_{\mathrm{Enr},1}} \left( \Phi_{\mathrm{Enr}}^{-1/8} \right)
\end{equation}
where \(\operatorname{FE}_{\mathfrak{c}_{\mathrm{Enr},1}}\) denotes
the Fourier expansion at the level-one cusp
\(\mathfrak{c}_{\mathrm{Enr},1}\) of the Enriques moduli
space.  In particular, the branch and normalization of
\(\Phi_{\mathrm{Enr}}^{-1/8}\) are chosen so that its constant
coefficient equals one.
\par Equivalently, one may compactify the local Enriques geometry
into the Enriques Calabi--Yau threefold
\begin{equation*} 
Z=(X\times E)/\langle(\iota_{\mathrm{Enr}},[-1]_E)\rangle,
\end{equation*}
where \(E\) is an elliptic curve.  The local Enriques invariants
appear as fiber-class invariants of \(Z\).  In this form the A-model
can also be described by stable pairs or by Donaldson--Thomas
invariants \cite{OberdieckEnriques}. For the broader theory of
orthogonal quasimodular forms and its conjectural extensions to
Enriques Gromov--Witten potentials, see
\cite{OberdieckWilliamsQuasimodular}.  We find the following
mirror-symmetry correspondence between the generating functions:
\begin{equation*} { \begin{array}{ccc} \text{B-model} & \overset{\text{mirror symmetry}}{\longleftrightarrow} & \text{A-model} \\[1mm] \text{Yoshikawa torsion of }(X,\iota) && \text{Gromov--Witten invariants of } \operatorname{Tot}(K_Y) \\[1mm] 
\exp(F^{B}_{1,\mathrm{Enr}})
=\tau_{M_{\mathrm{Enr}}}
=\bigl\|\Phi_{\mathrm{Enr}}\bigr\|^{-1/(2\nu(M_{\mathrm{Enr}}))}
&&
Z^{A}_{1,\mathrm{Enr}}
=\operatorname{FE}_{\mathfrak c_{\mathrm{Enr},1}}
 \bigl(\Phi_{\mathrm{Enr}}^{-1/8}\bigr)
\end{array} } \end{equation*}
\par For the Enriques period lattice~\eqref{enq:Lenr}, the
dual-rescaling isometry
\begin{equation}
\label{eq:Enriques-dual-rescaling}
L_{\mathrm{Enr}}^\vee(2) \cong U^{\oplus2}\oplus D_8(-1) = T_{D}
\end{equation}
identifies the full Enriques modular quotient with the full rank-ten
Vinberg quotient.  Gritsenko and Nikulin identify the additive lift
\begin{equation*} \Delta_{4,D_8} =\operatorname{Lift}\bigl(\vartheta(z_1)\cdots\vartheta(z_8)\bigr) \end{equation*}
with the Borcherds--Enriques form.  Once the character form \(H_4\), defined on the character cover by
\(H_8=-H_4^2\), is normalized as the unique weight-four character
generator, its weight, character, and reduced divisor give
\begin{equation}
\label{eq:PhiEn-H4-proportional}
\Phi_{\mathrm{Enr}}\doteq H_4, \qquad
H_8=-H_4^2\doteq-\Phi_{\mathrm{Enr}}^2
\end{equation}
\cite{BorcherdsEnriques,GritsenkoNikulinEnriques,YoshikawaII}.  The
symbol \(\doteq\) denotes equality up to a nonzero constant,
determined by fixing a cusp coordinate and a leading Fourier
coefficient.
\begin{proposition}
\label{prop:Enriques-normalization}
After fixing the dual-rescaling identification, a sign for \(H_4\),
and compatible local parameters at the level-one Enriques cusp and the
corresponding divisibility-two Vinberg cusp, one has
\begin{equation}
\label{eq:Enriques-partition-H8}
\Phi_{\mathrm{Enr}}=c_{\mathrm{Enr}}H_4, \qquad
Z^{A}_{1,\mathrm{Enr}}=\operatorname{FE}_{\mathfrak
  c_{D,2}}\!\left((-c_{\mathrm{Enr}}^2H_8)^{-1/16}\right),
\end{equation}
where \(c_{\mathrm{Enr}} =\tfrac29 \) and
\(\operatorname{FE}_{\mathfrak c_{D,2}}\) denotes the Fourier
expansion at the divisibility-two cusp of the \(T_D\)-modular variety.
\end{proposition}
\begin{proof}
Write \(L_{\mathrm{Enr}}=U_1\oplus U_2(2)\oplus E_8(-2)\), and choose
an isotropic basis \(e_1,f_1\) of \(U_1\) with \((e_1,f_1)=1\). The
primitive isotropic vector \(e_1\) has divisibility one in
\(L_{\mathrm{Enr}}\), so the line \(\mathbb Qe_1\) determines the level-one Enriques
cusp. After dualizing and rescaling,
\begin{equation}
\label{eq:Enriques-dual-rescaled-cusp}
L_{\mathrm{Enr}}^\vee(2)=U_1(2)\oplus U_2\oplus E_8(-1)\cong
U^{\oplus2}\oplus D_8(-1).
\end{equation}
The rational isotropic line is unchanged, but its primitive generator
now has divisibility two.  Thus the level-one Enriques cusp is carried
to the divisibility-two Vinberg cusp:
\begin{equation}
\label{eq:level-one-to-divisibility-two}
\mathfrak c_{\mathrm{Enr},1}\longmapsto\mathfrak c_{D,2}.
\end{equation}
The associated cusp lattice is \(e_1^\perp/\mathbb Ze_1\cong U\oplus
E_8(-1)\). The forms \(\Phi_{\mathrm{Enr}}\) and \(H_4\) have the same
weight, character, and reduced reflective divisor under the
dual-rescaling identification; the one-dimensionality of this
character space therefore gives
\(\Phi_{\mathrm{Enr}}=c_{\mathrm{Enr}}H_4\) for some
\(c_{\mathrm{Enr}}\ne0\)
\cite{BorcherdsEnriques,GritsenkoNikulinEnriques,ClingherMalmendierWilliams2026}. Finally,
\(H_8=-H_4^2\) and Oberdieck's formula yield
\begin{equation*} \Phi_{\mathrm{Enr}}^{-1/8}=(-c_{\mathrm{Enr}}^2H_8)^{-1/16}, \end{equation*}
which proves \eqref{eq:Enriques-partition-H8}.  With the cusp
coordinates and generator normalizations fixed above, 
\(H_4(q,s)=\tfrac92+O(q,s)\), whereas the
normalized branch of \(\Phi_{\mathrm{Enr}}\) has constant term \(1\).
Thus we have \(c_{\mathrm{Enr}}=\tfrac29\).
\end{proof}
\begin{remark}
\label{rem:F_theory}
In the standard fibration \eqref{eq:H+D8-weierstrass_intro}, the
divisor \(H_8=0\) generically enhances the \(\mathrm I_4^*\)-fiber to
\(\mathrm I_5^*\); hence the lattice polarization extends from
\(U\oplus D_8(-1)\) to a lattice containing \(U\oplus D_9(-1)\). On
the character cover, \(H_8=-H_4^2\) and \(H_4\) is a reduced local
equation of this reflective divisor
\cite{ClingherMalmendierWilliams2026}. Under heterotic/F-theory
duality, the complex-structure moduli of an elliptically fibered K3
surface are identified with heterotic Narain moduli. Gauge enhancement
occurs when the Narain lattice acquires additional norm-two vectors
with vanishing right-moving momentum. Correspondingly, the one-loop
heterotic threshold has logarithmic singularities along the rational
quadratic divisors
\begin{equation*} p_R=0, \qquad p_L^2=2. \end{equation*}
The divisor of \(\Phi_{\mathrm{Enr}}\) identifies this enhancement
locus in the orthogonal period domain
\cite{HarveyMoore,KawaiThreshold}.
\end{remark}
\section{From Vinberg's construction to the E-string surfaces}
\label{sec:Vinberg}
We will now describe the explicit relation between Vinberg's
projective model of the rank-ten family and the standard and alternate
genus-one fibrations supported on the family of K3 surfaces polarized
by the lattice $U\oplus D_8(-1)$. The geometric construction can be
summarized as follows:
\begin{equation}
\label{eq:vinberg-HU-explicit-sequence}
Y_{\mathrm V}^{\min} \dashrightarrow X_{F,G,H} \dashrightarrow
\mathcal C/\mathbb{P}^1_{(s)}
\xrightarrow{\operatorname{Jac}}X_{\mathrm{HU}}.
\end{equation}
The first arrow identifies Vinberg's scroll model in
\cite{Vinberg2018} with the standard
fibration~\eqref{eq:H+D8-weierstrass_intro}; the second changes the
genus-one pencil on the same K3 surface to the alternate fibration, a
genus-one fibration without a section; the last takes the relative
Jacobian. Hashimoto and Ueda showed that the ring of scalar modular
forms for the latter family is polynomial in eleven generators and
that adjoining the weight-\(252\) character form gives the full ring
\cite{HashimotoUeda}; see Appendix~\ref{ssec:WeqUE8}.
We then show that the Hashimoto--Ueda family coincides with 
Sakai's reconstruction of the F-theory K3 equation from the E-string 
Seiberg--Witten curve \cite{SakaiFTheoryK3}.
\subsection{Vinberg multipolarizations and arithmetic groups}
Let \(h_0\in\Lambda_{\mathrm{K3}}\) be primitive with \(h_0^2=d>0\),
let \(S_0\subset\Lambda_{\mathrm{K3}}\) be a primitive hyperbolic
lattice containing \(h_0\), and set \(T_0=S_0^\perp\). In
\cite[Section~4.1]{Vinberg2010}, Vinberg defines a multipolarization
of type \((h_0,S_0)\) on a K3 surface \(X\) to be a class \(h\) in the
chosen nef chamber and an \emph{unlabelled} sublattice
\(S\subset\operatorname{NS}(X)\) containing \(h\), such that some
marking
\begin{equation*} \varphi:H^2(X,\mathbb{Z})\ \xrightarrow{\sim}\ \Lambda_{\mathrm{K3}} \end{equation*}
sends \((h,S)\) to \((h_0,S_0)\). The marking itself is not part of
the data. It is defined only up to an isometry of
\(\Lambda_{\mathrm{K3}}\) preserving \(h_0\) and \(S_0\). Vinberg then
defines
\begin{equation*} \mathrm O^+(T_0,h_0) =\operatorname{res}_{T_0} \left\{ g\in\mathrm O^+(\Lambda_{\mathrm{K3}}): g(h_0)=h_0,\quad g(S_0)=S_0 \right\}. \end{equation*}
The period space is
\begin{equation*} \mathrm O^+(T_0,h_0)\backslash\mathcal{D}(T_0), \end{equation*}
up to the hyperplane arrangements needed to impose ampleness or
exclude extra algebraic classes. Vinberg states that \(\mathrm
O^+(T_0,h_0)\) has finite index in \(\mathrm O^+(T_0)\).
\par For an even lattice \(K\), write
\begin{equation*} A_K=K^\vee/K,\qquad q_K:A_K\longrightarrow\mathbb{Q}/2\mathbb{Z} \end{equation*}
for its discriminant form, and set
\begin{equation*} \widetilde{\mathrm O}(K)= \ker\!\left(\mathrm O(K)\longrightarrow\mathrm O(q_K)\right). \end{equation*}
This is the stable orthogonal group or discriminant kernel. Because
\(\Lambda_{\mathrm{K3}}\) is unimodular, a primitive decomposition
\(S_0\perp T_0\subset\Lambda_{\mathrm{K3}}\) determines an
anti-isometry
\begin{equation*} \lambda:(A_{S_0},q_{S_0}) \xrightarrow{\sim}(A_{T_0},-q_{T_0}). \end{equation*}
Vinberg's extension criterion \cite[Proposition~4]{Vinberg2010} states
that \((g_S,g_T)\in\mathrm O(S_0)\times\mathrm O(T_0)\) extends to the
K3 lattice precisely when
\begin{equation*} \overline g_T\circ\lambda=\lambda\circ\overline g_S. \end{equation*}
We introduce
\begin{equation}
\label{eq:extension-group}
 \Gamma_{S_0,h_0}^+ =\left\{ g_T\in\mathrm O^+(T_0):
 \begin{array}{l}
 \exists g_S\in\mathrm O(S_0):\ g_S(h_0)=h_0, \\ \overline
 g_T=\lambda\overline g_S\lambda^{-1}
 \end{array}
 \right\},
\end{equation}
which is Vinberg's \(\mathrm O^+(T_0,h_0)\), and we have
\begin{equation}
\label{eq:three-groups}
 \widetilde{\mathrm O}^+(T_0)\subseteq \Gamma_{S_0,h_0}^+\subseteq
 \mathrm O^+(T_0).
\end{equation}
The three groups have distinct moduli meanings:
\begin{enumerate}
\item if the primitive embedding
  \(S_0\hookrightarrow\operatorname{NS}(X)\) is fixed pointwise, the
  period group is \(\widetilde{\mathrm O}^+(T_0)\);
\item if only the image lattice and \(h_0\) are remembered, the period
  group is \(\Gamma_{S_0,h_0}^+\);
\item if one also forgets \(h_0\), all of \(\mathrm O(S_0)\) is
  allowed. When the reduction maps to the discriminant-form groups
  are surjective, the resulting group on \(T_0\) is the full \(\mathrm
  O^+(T_0)\).
\end{enumerate}
If \(\pi_T\) and \(\pi_{S,h_0}\) denote the corresponding
homomorphisms to the discriminant-form groups, then
\begin{equation*} \Gamma_{S_0,h_0}^+/\widetilde{\mathrm O}^+(T_0) \cong \operatorname{im}(\pi_T)\cap \lambda\operatorname{im}(\pi_{S,h_0})\lambda^{-1}, \end{equation*}
with the component-preserving condition understood. In particular, if
\(\pi_T\) is surjective, a multipolarized quotient is the full
quotient exactly when the stabilizer of \(h_0\) supplies every
discriminant action; see \cite{GritsenkoHulek}.
\par In \cite{Vinberg2010} Vinberg considers the cases
\begin{equation*} S_n=D_{1,19-n},\qquad T_n=D_{2,n},\qquad 3\leq n\leq7, \end{equation*}
together with a particular vector \(h_0\in S_n\) of square \(4\). From
\cite[Prop.~7]{Vinberg2010} we have
\begin{equation*} \mathrm O^+(T_n,h_0)= \begin{cases} \Gamma_7^0,&n=7,\\ \Gamma_6^{\mathrm{ext}}=\mathrm O^+(T_6),&n=6,\\ \Gamma_n=\mathrm O^+(T_n),&n=5,4,3. \end{cases} \end{equation*}
In the case \(n=7\), the multipolarized quotient is not the full
quotient. Vinberg's canonical equations are quartics in
\(\mathbb{P}^3\) for these \(S_n\). The Enriques lattice, whose
discriminant has absolute value \(2^{10}\), does not occur. In
\cite{Vinberg2018}, Vinberg considers the rank-ten case with
\begin{equation*} D_{1,9}\cong U\oplus D_8(-1) \end{equation*}
and a distinguished polarization of square eight.  The resulting
canonical scroll model is discussed in
Section~\ref{ssec:scroll_model}.
\subsection{Vinberg's canonical scroll model}
\label{ssec:scroll_model}
We now describe the explicit relation between Vinberg's projective
model of the rank-ten family and the standard and alternate genus-one
fibrations on the family of \(U\oplus D_8(-1)\)-polarized K3 surfaces.
Let
\begin{equation*} \mathcal S_{\mathrm{sc}} =\mathcal S(0,1,2)\subset\mathbb P^5 \end{equation*}
be the singular rational normal scroll with homogeneous coordinates
\begin{equation*} (x_0,x_1,x_2,y_1,y_2,z). \end{equation*}
Its standard resolution is the projective bundle
\begin{equation*} \widetilde{\mathcal S}_{\mathrm{sc}} :=\mathbb P_{\mathbb P^1} \bigl( \mathcal O_{\mathbb P^1} \oplus\mathcal O_{\mathbb P^1}(1) \oplus\mathcal O_{\mathbb P^1}(2) \bigr). \end{equation*}
The tautological morphism \(\widetilde{\mathcal
  S}_{\mathrm{sc}}\to\mathcal S_{\mathrm{sc}}\) contracts the section
associated with the trivial summand to the vertex \([1:0:0:0:0:0]\).
Its fibers map to a one-parameter family of ruling planes
\(L(u)\cong\mathbb P^2\), all meeting at that vertex.  The image is a
threefold of degree three.  In the homogeneous coordinates
$(x_0,x_1,x_2,y_1,y_2,z)$, its homogeneous ideal is generated by
\begin{equation}
\label{eq:vinberg-scroll-relations}
x_2y_1=x_1z, \qquad x_1y_2=x_2z, \qquad y_1y_2=z^2.
\end{equation}
On an affine chart of the base, the ruling plane over $u\in\mathbb
A^1$ is parametrized by $x_2=ux_1$, $z=uy_1$, $y_2=u^2y_1$,
where $(x_0:x_1:y_1)$ are homogeneous coordinates on $L(u)$.
\par Vinberg's canonical cubic is
\begin{equation}
\label{eq:vinberg-canonical-cubic}
\begin{aligned}
\mathcal F={}&x_0^2y_1+x_1^3+x_2^2z +(ax_1+by_1)y_1^2
\\ &+(a_1x_1y_1+a_2x_2y_2+b_1y_1^2+b_2y_2^2)z
\\ &+(f_1x_1+f_2x_2+g_1y_1+g_2y_2+hz)z^2.
\end{aligned}
\end{equation}
Restricting Vinberg's cubic equation to $L(u)$ produces a plane cubic
depending on $u$, and these plane cubics form the associated elliptic
fibration. The corresponding satellite cubic is
\begin{equation}
\label{eq:vinberg-satellite-cubic}
\begin{aligned}
\widetilde{\mathcal F}={}&x_0^2z+x_1^2x_2
+(x_2^2+a_2x_2y_2+b_2y_2^2)y_2
\\ &+(ax_1+by_1)y_1z+(f_2x_2+g_2y_2)y_2z
\\ &+(a_1x_1+f_1x_2+b_1y_1+hy_2+g_1z)z^2.
\end{aligned}
\end{equation}
These equations are given in \cite[Equations~(17)
  and~(18)]{Vinberg2018}. The degree-eight surface $Y_{\mathrm
  V}\subset\mathcal S_{\mathrm{sc}}$ is the common residual component
of the two cubic sections.  Equivalently, it is defined inside the
scroll by
\begin{equation}
\label{eq:vinberg-residual-surface}
\mathcal F=\widetilde{\mathcal F}=0.
\end{equation}
The minimal resolution of $Y_{\mathrm V}$ is a K3 surface
\cite[Equation~(20)]{Vinberg2018}. The cubic $\mathcal F$ does not cut
out the surface $Y_{\mathrm V}$ alone. Indeed, since the scroll
$\mathcal S_{\mathrm{sc}}$ has degree three, a cubic hypersurface
section has degree nine, whereas $Y_{\mathrm V}$ has degree
eight. Write \(v_1\) for the point \(u=0\) and \(v_2\) for the point
\(u=\infty\) on the ruling base. Vinberg shows that
\begin{equation}
\label{eq:F-with-residual-fiber}
V_{\mathcal S_{\mathrm{sc}}}(\mathcal F) = Y_{\mathrm V}\cup L(v_2),
\qquad V_{\mathcal S_{\mathrm{sc}}}(\widetilde{\mathcal F}) = Y_{\mathrm
  V}\cup L(v_1),
\end{equation}
where $L(v_2)\cong\mathbb P^2$ is one ruling plane of the scroll, and $L(v_1)$ is a different ruling plane related to the
satellite cubic. We then have
\begin{equation}
\label{eq:vinberg-residual-intersection}
Y_{\mathrm V} = V_{\mathcal S_{\mathrm{sc}}} (\mathcal
F,\widetilde{\mathcal F}).
\end{equation}
Although two cubic equations occur here, they do not impose two
independent conditions on the threefold scroll.  In the function field
of $\mathcal S_{\mathrm{sc}}$ they are related by
\begin{equation}
\label{eq:F-Ftilde-rational-relation}
\widetilde{\mathcal F} = \frac{x_2}{x_1}\mathcal F, \qquad
\frac{x_2}{x_1} = \frac{y_2}{z} = \frac{z}{y_1},
\end{equation}
where the second set of identities follows from the scroll relations.
On the fiber $x_2=ux_1$, $y_2=u^2y_1$, $z=uy_1$, one therefore has
\begin{equation}
\label{eq:F-Ftilde-on-ruling}
\left.\widetilde{\mathcal F}\right|_{L(u)} = u\left.\mathcal
F\right|_{L(u)}.
\end{equation}
On a general ruling, $\mathcal F$ and $\widetilde{\mathcal F}$ define
the same plane cubic, and it suffices to use $\mathcal F$ in order to
derive the Weierstrass equation.  The satellite cubic becomes
essential at the distinguished fiber $u=\infty$, where $\mathcal F$
vanishes identically.  Its restriction to that fiber is
\begin{equation}
\label{eq:Ftilde-distinguished-fiber}
\left.\widetilde{\mathcal F}\right|_{L(v_2)} =
y_2\bigl(x_2^2+a_2x_2y_2+b_2y_2^2\bigr),
\end{equation}
and hence contains three distinguished lines.
\subsection{The standard and alternate fibrations}
\label{ssec:fibs_on_vinberg}
The ruling of the scroll is parametrized, on an affine chart of the
base, by $u$.  The fiber over $u$ is obtained by setting
$x_2=ux_1$, $y_2=u^2y_1$, $z=uy_1$.
Substitution into
\eqref{eq:vinberg-canonical-cubic} gives
\begin{equation}
\label{eq:vinberg-fiber-cubic}
\mathcal F =x_0^2y_1+x_1^3+u^3x_1^2y_1
+\phi_2(u)x_1y_1^2+\phi_3(u)y_1^3,
\end{equation}
where
\begin{equation}
\label{eq:vinberg-phi23}
\begin{aligned}
\phi_2(u) &=a_2u^4+f_2u^3+f_1u^2+a_1u+a,\\ \phi_3(u)
&=b_2u^5+g_2u^4+hu^3+g_1u^2+b_1u+b.
\end{aligned}
\end{equation}
This is Vinberg's fiberwise Weierstrass form in
\cite[Equations~(23)--(25)]{Vinberg2018}.
\begin{proposition}
\label{prop:Vinberg_model}
Vinberg's scroll model $Y_{\mathrm V}^{\min}$ is birational to the
Weierstrass model~\eqref{eq:H+D8-weierstrass_intro}. The explicit
relation between Vinberg's coefficients and those of the Weierstrass
model is given in Appendix~\ref{App:VinbergCoefficients}. In
particular, one has $H_8=-\tfrac14(a_2^2-4b_2)$.
\end{proposition}
\begin{proof}
On the affine chart $y_1\neq0$, set $x_{\mathrm V}=\frac{x_1}{y_1}$, $y_{\mathrm V}=i\frac{x_0}{y_1}$.
Equation~\eqref{eq:vinberg-fiber-cubic} becomes
\begin{equation}
\label{eq:vinberg-affine-weierstrass}
y_{\mathrm V}^2 =x_{\mathrm V}^3+u^3x_{\mathrm V}^2
+\phi_2(u)x_{\mathrm V}+\phi_3(u).
\end{equation}
With
\begin{equation}
\label{eq:vinberg-x-translation}
r(u)=-\frac{a_2u+f_2}{2}, \qquad x_{\mathrm V}=x+r(u),
\end{equation}
we then have
\begin{equation}
\label{eq:cubic-translation-identity}
x_{\mathrm V}^3+u^3x_{\mathrm V}^2 +\phi_2x_{\mathrm V}+\phi_3
=x^3+(u^3+3r)x^2 +(\phi_2+2u^3r+3r^2)x +(\phi_3+\phi_2r+u^3r^2+r^3).
\end{equation}
For the particular choice in \eqref{eq:vinberg-x-translation}, the
terms $u^4$ and $u^3$ in the coefficient of $x$
cancel. Equation~\eqref{eq:vinberg-affine-weierstrass} takes the
normalized Weierstrass form~\eqref{eq:H+D8-weierstrass_intro}. In
particular, the multiset of weights of the eleven parameters is $(4,6,8,8,10,10,12,12,14,16,18)$,
in agreement with Vinberg's free algebra of automorphic forms. The
following rational coordinate map identifies the two birational
surface models: on the open charts used above, the map from the
Weierstrass model \eqref{eq:H+D8-weierstrass_intro} to the scroll is
\begin{equation}
\label{eq:weierstrass-to-vinberg-scroll-map}
(u,x,y) \ \dashrightarrow \ [x_0:x_1:x_2:y_1:y_2:z] \ =
\ [-iy:x+r(u):u(x+r(u)):1:u^2:u].
\end{equation}
The inverse map is $u=\frac{z}{y_1}$, $x=\frac{x_1}{y_1}-r(u)$, $y=i\frac{x_0}{y_1}$.
The maps are birational on the finite-$u$ chart.  Since smooth K3
surfaces that are birational are isomorphic, the two minimal
resolutions are the same K3 surface.
\end{proof}
\par The second isometric presentation
\begin{equation*} U\oplus D_8(-1)\cong U(2)\oplus E_8(-1) \end{equation*}
is also visible in Vinberg's distinguished scroll fiber.  Recall that
restricting the satellite cubic to the fiber at the distinguished
boundary point in~\eqref{eq:Ftilde-distinguished-fiber} gives the
three lines
\begin{equation}
\label{eq:vinberg-three-lines}
y_2\bigl(x_2^2+a_2x_2y_2+b_2y_2^2\bigr)=0.
\end{equation}
The two lines belonging to the quadratic factor are determined by the
roots of
\begin{equation}
\label{eq:vinberg-line-quadratic}
\rho^2+a_2\rho+b_2=0.
\end{equation}
If we introduce $\delta^2=a_2^2-4b_2$, the two roots of
Equation~\eqref{eq:vinberg-line-quadratic} are
$\varrho_{\pm}=\frac{-a_2\pm\delta}{2}$.  Using
\eqref{eq:vinberg-to-FGH}, we find $a_2^2-4b_2=-4H_8$. Thus after
adjoining a choice of square root, we may set
\begin{equation}
\label{eq:g4-H4-definition}
\delta=2H_4, \qquad H_8=-H_4^2.
\end{equation}
\begin{remark}
\label{rem:double-cover}
At a fixed point in the moduli space with \(H_8\neq0\), either square
root of the discriminant of~\eqref{eq:vinberg-line-quadratic} may be
chosen. Adjoining \(\delta=2H_4\) to the parameter space defines the
double cover $\delta^2=-4H_8$. The nontrivial deck transformation $\delta\mapsto-\delta$ interchanges
the two roots, hence the two lines in
\eqref{eq:vinberg-three-lines}. Thus \(H_8\) is defined on the full
coefficient quotient, whereas \(H_4=\delta/2\) is defined on the
twofold cover that selects one line. For the period lattice
\(T=U^{\oplus2}\oplus D_8(-1)\), the quotient \(\mathrm
O^+(T)/\widetilde{\mathrm O}^+(T)\) is cyclic of order two and its
nontrivial element interchanges the two isotropic discriminant
classes; the form denoted $G_4$ in
\cite[Section~5]{ClingherMalmendierWilliams2026} is $H_4$ in our
normalization. Its branch divisor is
\begin{equation}
\label{eq:level-two-branch-divisor}
H_8=0 \quad\Leftrightarrow\quad a_2^2=4b_2.
\end{equation}
Along this divisor the two lines in the quadratic factor of
\eqref{eq:vinberg-three-lines} merge.  This is Vinberg's
codimension-one specialization in
\cite[Equation~(22)]{Vinberg2018}. The resolved K3 surface generally
remains smooth but acquires an additional algebraic class; the
polarization enhances from \(U\oplus D_8(-1)\) to one containing
\(U\oplus D_9(-1)\); see Remark~\ref{rem:F_theory}.
\end{remark}
\par The choice of square root in the two-neighbor construction
corresponds to the choice of one of the two lines in the quadratic
factor of \eqref{eq:vinberg-three-lines}. 
\begin{proposition}
\label{prop:alt_fib}
A very general K3 surface $X$ with $\operatorname{NS}(X)=U\oplus
D_8(-1)$ admits a birational model as a double cover of
$\mathbb{P}^1_{(s)} \times \mathbb{P}^1_{(u)}$, given on an affine
chart by
\begin{equation}
\label{eq:alternate-genus-one-quartic}
w^2 =u^4A_0(s)+u^3A_1(s)+u^2A_2(s)+uA_3(s)+A_4(s),
\end{equation}
with
\begin{equation}
\label{eq:alternate-quartic-coefficients}
\begin{aligned}
A_0(s) &=2H_4s+H_{10}, \\ A_1(s) &=s^2+H_4^3+F_4H_4^2+H_4G_8+H_{12},
\\ A_2(s) &=\bigl(3H_4^2+2F_4H_4+G_8\bigr)s
+F_6H_4^2+H_4G_{10}+H_{14}, \\ A_3(s) &=(3H_4+F_4)s^2
+(2F_6H_4+G_{10})s +H_4G_{12}+H_{16}, \\ A_4(s)
&=s^3+F_6s^2+G_{12}s+H_{18}.
\end{aligned}
\end{equation}
\end{proposition}
\begin{proof}
The neighbor substitution in the normalized Weierstrass coordinate is
\begin{equation}
\label{eq:neighbor-x-substitution}
x=s+H_4u.
\end{equation}
Combining \eqref{eq:neighbor-x-substitution} with
\eqref{eq:vinberg-x-translation} gives a direct expression in
Vinberg's fiber coordinate, namely
\begin{equation}
\label{eq:neighbor-directly-in-vinberg-coordinate}
x_{\mathrm V} =s-\frac{f_2}{2}+\varrho_+u
=s-\frac{f_2}{2}+\frac{-a_2+\delta}{2}u.
\end{equation}
Replacing $\delta$ by $-\delta$ replaces $\varrho_+$ by $\varrho_-$
and selects the other line. Substitute
\eqref{eq:neighbor-x-substitution} and $H_8=-H_4^2$ into
\eqref{eq:H+D8-weierstrass_intro}. The only terms of degree five in
$u$ are
\begin{equation*} u^3(H_4u)^2+H_8u^5=(H_4^2+H_8)u^5=0. \end{equation*}
After renaming the left-hand coordinate $w$, the equation is the
quartic~\eqref{eq:alternate-genus-one-quartic}. In particular,
\begin{equation}
\label{eq:alternate-quartic-degrees}
\deg_s(A_4,A_3,A_2,A_1,A_0)=(3,2,1,2,1).
\end{equation}
The absence of a section for the very general member follows from the
same discriminant argument as in Appendix~\ref{App:2NS}: a section
would give a full-rank sublattice $U\oplus E_8(-1)$ of discriminant
one inside $\operatorname{NS}(X)$ of discriminant four, which is
impossible. Since the old fiber is a bisection, the new fiber class
has divisibility two.
\end{proof}
\par Equation~\eqref{eq:alternate-genus-one-quartic} realizes the
$U(2)\oplus E_8(-1)$ polarization of the same K3 surface $X_{F,G,H}$
as a genus-one fibration over $\mathbb{P}^1_{(s)}$.  The relative
Jacobian of the index-two genus-one fibration
\eqref{eq:alternate-genus-one-quartic} has the Weierstrass equation
\begin{equation}
\label{eq:relative-jacobian-IJ}
J_{F,G,H}\colon \quad \eta^2 =\xi^3-\frac13\mathcal
I(s)\xi-\frac1{27}\mathcal J(s),
\end{equation}
where the classical binary-quartic invariants, in terms of the
coefficients of the quartic, are
\begin{equation}
\label{eq:I-J-of-s}
\begin{aligned}
\mathcal I &=12A_0A_4-3A_1A_3+A_2^2, \\ \mathcal J
&=72A_0A_2A_4+9A_1A_2A_3 -27A_0A_3^2-27A_1^2A_4-2A_2^3;
\end{aligned}
\end{equation}
see \cite{FisherGenusOneInvariants} for these conventions.
\par For very general coefficients, we choose markings so that
\begin{equation} \label{eqn:lattices}T(X_{F,G,H})\cong T_D:=U^{\oplus2}\oplus D_8(-1),\qquad T(J_{F,G,H})\cong T_E:=U^{\oplus2}\oplus E_8(-1).\end{equation}
Ogg--Shafarevich theory associates to the genus-one fibration a Brauer
class \(\alpha\) on the relative Jacobian. Its order is the
multisection index, which is two here. Thus
\(\alpha\in\operatorname{Br}(J_{F,G,H})[2]\), and it identifies the
transcendental lattice of the torsor via
\begin{equation*} T_D\cong\ker\!\left(\alpha:T_E\rightarrow\mathbb Z/2\mathbb Z\right) \end{equation*}
as an index-two Hodge sublattice; see \cite[Lemma~4.4 and
  Equation~(5.1)]{MeinsmaShinder}.  In particular, \(T_D\otimes\mathbb
Q=T_E\otimes\mathbb Q=:V\), and the two period lines agree in the
type-IV domain \(\mathcal D(V)\).
\begin{remark}
As we will see in Section~\ref{subsec:two-E-string-limits}, the
limiting rational surface component (and E-string Seiberg--Witten
curve) associated with the alternate
fibration~\eqref{eq:alternate-genus-one-quartic} admits a section and
coincides with the one obtained for its relative Jacobian. But the
chosen elliptic ruling remains part of its F-theory interpretation and
distinguishes the alternate genus-one fibration and its relative
Jacobian.  This is compatible with the CHL interpretation of the
$U(2)$-polarized model \cite{ClingherMalmendierCHL}; the role of the
torsor class in F-theory on genus-one fibrations is also discussed in
\cite[Sections~2 and~4.3--4.4]{BraunMorrisonGenusOne}.
\end{remark}
\subsection{Relation with the Hashimoto--Ueda family}
Equation~\eqref{eq:relative-jacobian-IJ} is a Weierstrass model for
the family of $U\oplus E_8(-1)$-polarized K3 surfaces.  Such surfaces
were also studied by Hashimoto and Ueda in \cite{HashimotoUeda}. Their
normal form is
\begin{equation}
\label{eq:UE8-general}
\begin{aligned}
X_{\mathrm{HU}}: \quad y^2={}&x^3 +\bigl(
C_4t^4+C_{10}t^3+C_{16}t^2+C_{22}t+C_{28} \bigr)x\\ &+
t^7+C_{12}t^5+C_{18}t^4+C_{24}t^3 +C_{30}t^2+C_{36}t+C_{42}.
\end{aligned}
\end{equation}
They prove that the scalar-valued ring for $\Gamma:=\mathrm
O^+(U^{\oplus2}\oplus E_8(-1))$ is
\begin{equation}
M_*(\Gamma) = \mathbb
C[C_4,C_{10},C_{12},C_{16},C_{18},C_{22},C_{24},C_{28},
  C_{30},C_{36},C_{42}],
\end{equation}
and that the full ring with characters is obtained by adjoining a
determinant-character generator $\Phi_{252}$ satisfying one relation
of weight $504$, namely $\Phi_{252}^2=\Psi_{504}$.  We discuss the
details in Section~\ref{sec:Hashimoto-Ueda}.
\par Let \(\Gamma_D\) denote the arithmetic group for the Vinberg
\(D_8\)-family, and let \(\chi_4\) be the character of the weight-four
form \(H_4\).  We work on the character cover corresponding to
\begin{equation*} \Gamma_D^{H_4}:=\ker(\chi_4). \end{equation*}
On this cover the components of
\begin{equation}
\label{eq:D8-generator-vector}
\mathbf
D=(F_4,H_4,F_6,G_8,G_{10},H_{10},G_{12},H_{12},H_{14},H_{16},H_{18})
\end{equation}
are scalar-valued modular forms with weights
\begin{equation*} \mathbf d=(4,4,6,8,10,10,12,12,14,16,18), \end{equation*}
and hence \(\sum_i d_i=114\).  Similarly, we set
\begin{equation}
\label{eq:E8-generator-vector}
\mathbf
C=(C_4,C_{10},C_{12},C_{16},C_{18},C_{22},C_{24},C_{28},C_{30},C_{36},C_{42}),
\end{equation}
with weight vector
\begin{equation*} \mathbf c=(4,10,12,16,18,22,24,28,30,36,42), \end{equation*}
so that \(\sum_j c_j=242\).  Since the relevant type-IV domains have
dimension ten, following Proposition~\ref{prop:Jacobian-general} we
obtain modular Jacobians of weights \(124\) and \(252\), respectively.
We write
\begin{equation}
\label{eq:D8-E8-modular-Jacobians}
\phi_{124}:=J(\mathbf D),\qquad J(\mathbf C)\doteq\Phi_{252},
\end{equation}
where \(\doteq\) denotes equality up to a nonzero constant.  Then
\(\phi_{124}\) has determinant character for \(\Gamma_D^{H_4}\);
viewed on the full group \(\Gamma_D\), its character is
\(\chi_4\det\).  The second relation in
\eqref{eq:D8-E8-modular-Jacobians} follows from
Theorem~\ref{thm:HU-Jacobian} and the normalization
\(\Phi_{252}^2=\Psi_{504}\); see
\cite{HashimotoUeda,WangClassificationFree}.
\begin{lemma}
\label{lem:relative-Jacobian-period-map}
The inclusion of lattices induces an inclusion of arithmetic groups
\begin{equation*} \Gamma_D^{H_4}=\widetilde{\mathrm O}^+(T_D)\subseteq\Gamma_E:=\mathrm O^+(T_E) \end{equation*}
and hence a finite morphism
\begin{equation*} \Gamma_D^{H_4}\backslash\mathcal D(V)\longrightarrow\Gamma_E\backslash\mathcal D(V). \end{equation*}
In a rational system of tube-domain coordinates for \(V=
T_D\otimes\mathbb Q=T_E\otimes\mathbb Q\), the morphism is the
identity in period coordinates, and the determinant character of
\(\Gamma_E\) restricts to the determinant character of
\(\Gamma_D^{H_4}\).
\end{lemma}
\begin{proof}
The overlattice \(T_E\supset T_D\) corresponds to the isotropic
subgroup \(K:=T_E/T_D\subset A_{T_D}\) of order two. By
Remark~\ref{rem:double-cover}, \(\Gamma_D^{H_4}=\widetilde{\mathrm
  O}^+(T_D)\). Every element of this discriminant kernel fixes \(K\)
and therefore extends to an isometry of \(T_E\)
\cite{NikulinIntegral}. This proves the inclusion of groups; it has
finite index because both groups are arithmetic in \(\mathrm
O^+(V)\). Since the lattice inclusion is a Hodge inclusion, it
preserves the period line. The final statements follow by choosing one
rational tube-domain basis for \(V\); the determinant character in
both cases is the determinant of the action on this common quadratic
space.
\end{proof}
\begin{proposition}
\label{prop:polynomial_map}
The relative Jacobian of the genus-one fibration
\eqref{eq:alternate-genus-one-quartic} determines a
weighted-homogeneous polynomial map of affine coefficient cones
\begin{equation*} \Theta\colon\mathbf D\longmapsto\mathbf C(\mathbf D), \end{equation*}
given by weighted-homogeneous polynomials \(C_k=C_k(F,G,H)\). The
ordinary coefficient Jacobian
\begin{equation}
\label{eq:R128-definition}
J_\Theta:=\operatorname{Jac}(\Theta)=\det\left(\frac{\partial
  C_j}{\partial D_i}\right)_{1\leq i,j\leq11}
\end{equation}
is weighted homogeneous of weight \(128\). In the compatible
tube-domain coordinates of
Lemma~\ref{lem:relative-Jacobian-period-map}, one has
\begin{equation}
\label{eq:Phi252-pullback}
J\bigl(\mathbf C(\mathbf D)\bigr)=J_\Theta J(\mathbf D) \quad
\Leftrightarrow \quad \Theta^*\Phi_{252}\doteq J_\Theta\phi_{124}.
\end{equation}
\end{proposition}
\begin{proof}
We assign weights $\operatorname{wt}(u)=2$, $\operatorname{wt}(s)=6$,
and $\operatorname{wt}(w)=9$.  The formulas in
\eqref{eq:alternate-quartic-coefficients} show that \(A_i(s)\) is
weighted homogeneous of weight \(10+2i\).  It follows from
\eqref{eq:I-J-of-s} that $\operatorname{wt}(\mathcal I)=28$ and $\operatorname{wt}(\mathcal J)=42$.
The degrees \(\deg_s(A_4,A_3,A_2,A_1,A_0)=(3,2,1,2,1)\) also give
$\deg_s\mathcal I=4$ and $\deg_s\mathcal J=7$
for general coefficients.  In particular, we have
\begin{equation*} A_1(s)=s^2+O(1),\qquad A_4(s)=s^3+F_6s^2+O(s). \end{equation*}
Among the summands in \(\mathcal J\), only \(-27A_1^2A_4\) can have
degree at least six. It follows that
\begin{equation}
\label{eq:J-leading-terms}
-\frac1{27}\mathcal J(s)=s^7+F_6s^6+O(s^5).
\end{equation}
The weighted translation \(s=t-F_6/7\) eliminates the coefficient of
\(t^6\).  Define
\begin{equation}
\label{eq:HU-f-g-definition}
f_{\mathrm{HU}}(t):=-\frac13\mathcal
I\left(t-\frac{F_6}{7}\right),\qquad
g_{\mathrm{HU}}(t):=-\frac1{27}\mathcal J\left(t-\frac{F_6}{7}\right).
\end{equation}
Then \(\deg_t f_{\mathrm{HU}}=4\), \(\deg_t g_{\mathrm{HU}}=7\), the
leading coefficient of \(g_{\mathrm{HU}}\) is one, and its
\(t^6\)-coefficient vanishes.  Comparison with \eqref{eq:UE8-general}
yields every \(C_k\) as a polynomial in the entries of \(\mathbf D\),
and the weight computation above shows that it is weighted homogeneous
of weight \(k\).
\par It remains to verify the modular-Jacobian identity.  Write \(D_i\) and
\(C_j\) for the entries of \(\mathbf D\) and \(\mathbf C\), let
\(d_i=\operatorname{wt}(D_i)\) and \(c_j=\operatorname{wt}(C_j)\), and
set
\begin{equation*} P_{ij}:=\frac{\partial C_j}{\partial D_i}. \end{equation*}
Weighted Euler homogeneity and the ordinary chain rule, applied in every tube-domain coordinate \(z_\nu\), give
\begin{equation} 
\label{eqn:ids}
c_jC_j=\sum_{i=1}^{11}d_iD_iP_{ij}, \qquad
\partial_{z_\nu}C_j=\sum_{i=1}^{11}(\partial_{z_\nu}D_i)P_{ij}. \end{equation}
If \(M_{\mathbf D}\) and \(M_{\mathbf C}\) denote the matrices
occurring in the modular Jacobians \(J(\mathbf D)\) and \(J(\mathbf
C(\mathbf D))\), the identities~\eqref{eqn:ids} yield
$M_{\mathbf C}=M_{\mathbf D}P$.
Taking determinants proves \eqref{eq:Phi252-pullback}.  Moreover,
every term of \(\det P\) has weight
\begin{equation*} \sum_{j=1}^{11}c_j-\sum_{i=1}^{11}d_i=242-114=128, \end{equation*}
so \(J_\Theta\) has weight \(128\).  Finally, substituting \(J(\mathbf
D)=\phi_{124}\) and \(J(\mathbf C)\doteq\Phi_{252}\) proves
\eqref{eq:Phi252-pullback}.
\end{proof}
\subsubsection{The \texorpdfstring{\(A_7\)}{A7}-enhancement divisor in the
\texorpdfstring{\(U\oplus D_8(-1)\)}{U plus D8(-1)} family}
\label{ssec:A7-enhancement}
The general \(U\oplus D_8(-1)\)-polarized K3 surface was given by
\eqref{eq:H+D8-weierstrass_intro}.  For its use below, we also
characterize the sublocus for which the transcendental lattice
specializes from $U^{\oplus2}\oplus D_8(-1)$ to $U^{\oplus2}\oplus
A_7(-1)$. An explicit model for the latter family is
\begin{equation}
\label{eq:A7-family-2}
y^2 = x^3 + \bigl(u^3+f_4 u+f_6\bigr)x^2 + \bigl(g_4 u^4+g_6 u^3+g_8
u^2+g_{10} u+g_{12}\bigr)x + \bigl(h_5 u^2+h_7 u+h_9\bigr)^2.
\end{equation}
There are two \(\widetilde{\mathrm O}^+\)-inequivalent embeddings
\(A_7\hookrightarrow D_8\); their union is the Heegner divisor of the
weight-\(128\) Borcherds product \(\Phi_{128}\)
\cite{ClingherMalmendierWilliams2026}.
In~\cite{ClingherMalmendierWilliams2026}, the coefficient
specialization is made explicit. Let $h_5$ be a variable independent
of the coefficients of \eqref{eq:H+D8-weierstrass_intro}. On the chart
\(H_4\neq0\), apply the substitution
\begin{equation}
\label{eqn:subi}
  x \mapsto x \pm \left( H_4 u + \frac{h_5^2 - H_{10}}{2H_4} \right).
\end{equation}
Then Vinberg's $D_8$ model coincides with the $A_7$
model~\eqref{eq:A7-family-2} if and only if the coefficients
$F_k,G_k,H_k$ and $f_k,g_k,h_k$ satisfy the relations
in~\cite[Equation~(5.2)]{ClingherMalmendierWilliams2026}.
\par We observe that after the substitution~\eqref{eqn:subi},
Equation~\eqref{eq:H+D8-weierstrass_intro} takes the form of a long
Weierstrass model whose constant term is a quartic polynomial in
\(u\),
\begin{equation}
\label{eqn:quartic_general}
Q(u)=b_0u^4+b_1u^3+b_2u^2+b_3u+b_4.
\end{equation}
Thus the $A_7$ family is characterized by the condition
$Q(u)=(h_5u^2+h_7u+h_9)^2$.  The Zariski closure of the locus of
quartic squares is the image of
\begin{equation}
\label{eq:quartic-square-parametrization}
(r_0,r_1,r_2)\longmapsto(b_0,b_1,b_2,b_3,b_4)=(r_0^2,2r_0r_1,r_1^2+2r_0r_2,2r_1r_2,r_2^2).
\end{equation}
Eliminating $r_0,r_1,r_2$
from~\eqref{eq:quartic-square-parametrization} shows that its
homogeneous ideal is generated by the following seven cubics:
\begin{equation}
\label{eq:quartic-square-ideal}
\begin{gathered} b_3^3-4b_2b_3b_4+8b_1b_4^2,\\ b_2b_3^2-4b_2^2b_4+2b_1b_3b_4+16b_0b_4^2,\\ b_1b_3^2-4b_1b_2b_4+8b_0b_3b_4,\\ b_1^2b_4-b_0b_3^2,\\ b_1^2b_3-4b_0b_2b_3+8b_0b_1b_4,\\ b_1^2b_2-4b_0b_2^2+2b_0b_1b_3+16b_0^2b_4,\\ b_1^3-4b_0b_1b_2+8b_0^2b_3. \end{gathered}
\end{equation}
\subsection{Sakai's Fourier--Jacobi reconstruction}
\label{subsec:Sakai-FJ}
Sakai constructs the modular coefficients of the Hashimoto--Ueda
equation from the Fourier--Jacobi expansions of orthogonal Eisenstein
series \cite{SakaiFTheoryK3}. We compare this reconstruction with our
inverse period map constructed above. Throughout, $p=e^{2\pi i\rho}$
is the Fourier--Jacobi variable and $(\tau,\mathbf z)$ are held fixed.
Here $(\tau,\mathbf z,\rho)$ are tube-domain coordinates on the
orthogonal period domain associated with $U^{\oplus2}\oplus E_8(-1)$,
where $\tau,\rho\in\mathbb H$ and $\mathbf z\in E_8\otimes\mathbb C$.
The Fourier--Jacobi expansion is taken as
$\operatorname{Im}\rho\to+\infty$, and its coefficients are Jacobi
forms with modular variable $\tau$ and elliptic variable $\mathbf z$.
\subsubsection{Sakai's modular equation and inverse period map}
\par We first compare Sakai's normalization with ours. Let $\mathcal
E_k$ denote the normalized Eisenstein series (used in this article),
whose zeroth Fourier--Jacobi coefficient is the elliptic Eisenstein
series $E_k(\tau)$ normalized to have constant Fourier coefficient
$1$. If $\mathcal E_k^{\mathrm S}$ denotes Sakai's Eisenstein series,
then \cite[Equation~(4.7)]{SakaiFTheoryK3} reads
\begin{equation}
\label{eq:Sakai-Eisenstein-normalization}
 \mathcal E_k^{\mathrm S}(\tau,\rho,\mathbf z)
 =-\frac{B_k}{2k}E_k(\tau) +\sum_{m\geq1}
 \bigl(E_{k-4}(\tau)\Theta(\tau,\mathbf z)\bigr) \vert T(m)\,p^m,
 \qquad p=e^{2\pi i\rho}.
\end{equation}
Here $B_k$ is the $k$th Bernoulli number, $\Theta$ is the $E_8$ theta
function, $T(m)$ is the Jacobi index-raising Hecke operator, and
$E_0=1$. Therefore, we have
\begin{equation}
\label{eq:Sakai-versus-normalized-Eisenstein}
 \mathcal E_k^{\mathrm S}=-\frac{B_k}{2k}\mathcal E_k, \qquad \mathcal
 E_k=-\frac{2k}{B_k}\mathcal E_k^{\mathrm S}.
\end{equation}
The conversion factors $-\frac{2k}{B_k}$ from Sakai's normalization to
ours start with
\begin{equation*}
 240,\qquad -264,\qquad \frac{65520}{691} \qquad (k=4,10,12).
\end{equation*}
Sakai writes the F-theory background as the Jacobian elliptic K3
surface with the following Weierstrass equation
\begin{equation}
\label{eq:Sakai-K3-curve}
\begin{split}
y_{\mathrm S}^2 & =4x_{\mathrm S}^3
-\bigl(\varphi_4u^4+\varphi_{10}u^3+\varphi_{16}u^2
+\varphi_{22}u+\varphi_{28}\bigr)x_{\mathrm
  S}\\ &+4u^7-\bigl(\varphi_{12}u^5+\varphi_{18}u^4
+\varphi_{24}u^3+\varphi_{30}u^2+\varphi_{36}u +\varphi_{42}\bigr).
\end{split}
\end{equation}
Each $\varphi_k$ is taken to be the most general weighted polynomial
of weight $k$ in the eleven Eisenstein generators; altogether this
ansatz contains $76$ rational constants. After the translation
$u\longmapsto u- \varphi_{10}/(4\varphi_4)$
we write $\widetilde\varphi_k$ for the resulting coefficients. Sakai
imposes the following vanishing conditions
\begin{equation}
\label{eq:Sakai-vanishing-conditions}
\begin{aligned}
 \widetilde\varphi_4&=a_0+O(p), &\widetilde\varphi_{4+6i}&=O(p^i)
 &&(i=2,3,4),\\ \widetilde\varphi_6&=b_0+O(p),
 &\widetilde\varphi_{6+6j}&=O(p^j) &&(j=1,\ldots,6),
\end{aligned}
\end{equation}
where $a_0=E_4(\tau)/12$ and $b_0=E_6(\tau)/216$. The coefficient of
$p^m$ in an orthogonal modular form is a $W(E_8)$-invariant Jacobi
form of index $m$. Expressing these coefficients in the algebraically
independent meromorphic Jacobi generators $a_i,b_j$ turns
\eqref{eq:Sakai-vanishing-conditions} into an overdetermined
\emph{linear} system for the $76$ constants. Algebraic independence
allows coefficients of distinct monomials in the $a_i,b_j$ to be
compared, and the system now has a unique solution. This is the
uniqueness step in Sakai's argument. See
\cite{SakaiAlgebraicE8,SakaiFTheoryK3,SakaiE8Ring,WangE8Jacobi} for
background on $E_8$ Jacobi forms.
\par We denote by
\[
 \Delta=\Delta(\tau)
 =\eta(\tau)^{24}
 =q\prod_{m=1}^{\infty}(1-q^m)^{24},
 \qquad q=e^{2\pi i\tau},
\]
the normalized modular discriminant, which is the cusp form of weight
\(12\) for \(\mathrm{SL}_2(\mathbb Z)\). 
The unique solution satisfies the sharper leading-term identities
\begin{equation}
\label{eq:Sakai-leading-FJ}
 \widetilde\varphi_{4+6n}
 =a_n\Delta^n p^n+O(p^{n+1}),\qquad
 \widetilde\varphi_{6+6n}
 =b_n\Delta^n p^n+O(p^{n+1}).
\end{equation}
Here \(n\in\{0,2,3,4\}\) in the first identity and
\(n\in\{0,1,\ldots,6\}\) in the second.  The coefficients
\(a_n(\tau,\mathbf z)\) and \(b_n(\tau,\mathbf z)\) are meromorphic
\(W(E_8)\)-invariant Jacobi forms satisfying
\begin{equation}
\label{eq:E-string-Jacobi-weights}
 \operatorname{wt}(a_n)=4-6n,\qquad
 \operatorname{wt}(b_n)=6-6n,\qquad
 \operatorname{ind}(a_n)=\operatorname{ind}(b_n)=n.
\end{equation}
These weights follow from \eqref{eq:Sakai-leading-FJ} by dividing the
leading coefficient by $\Delta^n$. The rescaling that realizes these
Jacobi forms as coefficients of a rational elliptic surface, its
complementary component, and the physical interpretation will be
discussed in Section~\ref{subsec:two-E-string-limits}.
\par The remaining issue in \cite{SakaiFTheoryK3} is as follows: the
degeneration conditions determine a unique tuple of orthogonal modular
forms, but by themselves do not identify that tuple with the geometric
inverse of the K3 period map. This identification follows by
comparison with the Hashimoto--Ueda normal form and the results of
\cite{ClingherMalmendierWilliams2026}.
\begin{theorem}
\label{thm:Sakai-comparison}
Let $\varphi_k^{\mathrm S}$ be Sakai's eleven coefficient polynomials,
written in the Eisenstein series $\mathcal E_k^{\mathrm S}$, and let
$C_k$ be the Hashimoto--Ueda coefficients in the normalization of
\eqref{eq:Hashimoto-Ueda-from-Vinberg}. Then, after the Eisenstein
normalization~\eqref{eq:Sakai-versus-normalized-Eisenstein}, one has
\begin{equation}
\label{eq:Sakai-HU-coefficient-comparison}
 C_k=-\frac14(-12)^{k/2}\varphi_k^{\mathrm S}, \qquad
 k\in\{4,10,12,16,18,22,24,28,30,36,42\}.
\end{equation}
In particular, Sakai's modular equation is, up to the weighted
coordinate change, the Hashimoto--Ueda universal family. It gives the
inverse period map on the common dense open moduli locus, as
conjectured in \cite{SakaiFTheoryK3}.
\end{theorem}
\begin{proof}
We choose $\lambda\in\mathbb C^\times$ with $\lambda^2=-12$ and make
the weighted change of variables
\begin{equation}
\label{eq:Sakai-HU-coordinate-change}
 t=\lambda^6u,\qquad x=\lambda^{14}x_{\mathrm S},\qquad
 y=\frac{\lambda^{21}}{2}y_{\mathrm S}.
\end{equation}
Dividing~\eqref{eq:Sakai-K3-curve} first by $4$ and then applying
\eqref{eq:Sakai-HU-coordinate-change} gives a monic equation with
leading term $t^7$ and sends every coefficient of weight $k$ to
$-\tfrac14\lambda^k\varphi_k^{\mathrm S}$. Since
$\lambda^k=(-12)^{k/2}$, this proves the formal coordinate relation in
\eqref{eq:Sakai-HU-coefficient-comparison}.
\par It remains to identify the modular forms. We use the compatible period
coordinates of Lemma~\ref{lem:relative-Jacobian-period-map}, with the
common cusp coming from $U^{\oplus2}$ in $T_D\subset T_E$, and put
$p=e^{2\pi i\rho}$. The coefficients of the relative Jacobian are
holomorphic modular forms of the indicated weights for $\mathrm
O^+(T_E)$ by the period identification and the Hashimoto--Ueda
theorem. The fact that the graded ring of modular forms is generated
by Eisenstein series \cite[Theorem~4.3]{DieckmannKriegWoitalla} places
the coefficients in Sakai's weighted-polynomial ansatz.  In the
normalization of the Vinberg coefficient forms, \cite[Section~5 and
  Corollary~5.1]{ClingherMalmendierWilliams2026} gives
\begin{equation*}
\begin{aligned}
 F_4&=-3E_4+O(p),& F_6&=2E_6+O(p),\\ G_8,G_{10},G_{12},H_4&=O(p),&
 H_{10},\ldots,H_{18}&=O(p^2).
\end{aligned}
\end{equation*}
The estimate \(H_4=O(p)\) follows from \(H_8=-H_4^2=O(p^2)\).
Equation~\eqref{eq:alternate-quartic-coefficients} now shows that
$A_i(pS)=O(p^2)$ for every $i$. Since the binary-quartic invariants
are homogeneous of degrees two and three in the $A_i$, respectively,
writing
\begin{equation*}
 f(s):=-\frac13\mathcal I(s)=\sum_{j=0}^4f_j(p)s^j, \qquad
 g(s):=-\frac1{27}\mathcal J(s)=s^7+\sum_{j=0}^6g_j(p)s^j
\end{equation*}
gives
\begin{equation*}
\begin{aligned}
 f_j(p)&=O(p^{4-j})\quad(0\leq j\leq4),&
 f_4(0)&=-3E_4,\\ g_j(p)&=O(p^{6-j})\quad(0\leq j\leq6),&
 g_6(0)&=2E_6.
\end{aligned}
\end{equation*}
In particular, at \(p=0\) the equation is
\(y^2=x^3-3E_4s^4x+s^6(s+2E_6)\).
On the dense open set $E_4\ne0$, the translation $s=S-f_3/(4f_4)$ has
shift $O(p)$, eliminates the cubic term in $f$, and preserves all of
the above order bounds. Since $t=s+F_6/7$, this is exactly
$t=S-C_{10}/(4C_4)$, the translation used by Sakai after the weighted
coordinate change. In Sakai's scale, the leading quartic and sextic
coefficients are
\begin{equation*}
 -4\lambda^{-4}(-3E_4)=\frac{E_4}{12}=a_0, \qquad
 -4\lambda^{-6}(2E_6)=\frac{E_6}{216}=b_0.
\end{equation*}
The geometric coefficients therefore satisfy all of
\eqref{eq:Sakai-vanishing-conditions}. The uniqueness of Sakai's
solution to these conditions \cite[Section~5]{SakaiFTheoryK3} proves
\eqref{eq:Sakai-HU-coefficient-comparison}; equality extends over the
whole period domain by holomorphy.
\par Substituting \eqref{eq:Sakai-versus-normalized-Eisenstein} into
\cite[Appendix~A]{SakaiFTheoryK3} gives the explicit Eisenstein
expressions.  For example, the first four identities become
\begin{equation*}
\begin{aligned}
C_4&=-3\mathcal E_4, &C_{10}&=\frac{24}{7}\mathcal
E_{10},\\ C_{12}&=\frac{378}{125}\mathcal E_4^3
-\frac{4146}{875}\mathcal E_{12}, &C_{16}&=-\frac{10557}{3125}\mathcal
E_4^4 +\frac{3756276}{284375}\mathcal E_4\mathcal E_{12}
-\frac{22494123}{1990625}\mathcal E_{16}.
\end{aligned}
\end{equation*}
The complete list of eleven identities is given in
Appendix~\ref{app:Eisenstein-generator-identities}.  Hashimoto--Ueda's
weighted coefficient space, with the non-K3 locus removed, is the
coarse moduli space of the corresponding $U\oplus E_8(-1)$-polarized
K3 surfaces. Hence equality of the weighted coefficient maps
identifies Sakai's map with the geometric inverse period map, up to
the coordinate scaling and the finite marking ambiguity inherent in
the coarse quotient.
\end{proof}
\subsubsection{The E-string Seiberg--Witten curve}
\label{ssec:SWcurve}
In the F-theory/heterotic string duality, the tube-domain coordinates
$(\tau,\mathbf z,\rho)$ for the $U\oplus E_8(-1)$-polarized
Hashimoto--Ueda family have the following interpretation: the
parameter \(\tau\in\mathbb H\) is the complex structure of the
heterotic elliptic curve, and the class of
\[
 \mathbf z\in (E_8\otimes\mathbb C)/(E_8+\tau E_8)
\]
modulo the Weyl group encodes the eight complex Wilson-line parameters.
The $\mathrm{II}^*$ fiber of the K3 family
corresponds to the unbroken $E_8$ factor of the heterotic string
gauge algebra.  Finally, $p=e^{2\pi i\rho}$ is transverse to the
one-dimensional Fourier--Jacobi boundary; the Baily--Borel boundary
records the isomorphism class of the elliptic curve $E_\tau$ and
contracts the Wilson-line parameters.  A toroidal compactification
preserves the extension data in the form of modular parameters, while
the semistable model constructed below specifies it as part of the
geometry of its rational elliptic surface components
\cite{ClingherMorganFTheory,SakaiFTheoryK3}.
\par We first define the scaling operation that will be used to obtain the
exceptional component of the semistable surface.  Consider a family of
Weierstrass equations
\begin{equation}
\label{eqn:before_limit}
 y^2=x^3+A_2(t,\epsilon)x^2+A_4(t,\epsilon)x+A_6(t,\epsilon), \qquad
 A_i=\sum_{j\geq0}A_{i,j}(\epsilon)t^j,
\end{equation}
over an $\epsilon$-disk, where the $A_i$ are polynomial in $t$ and
holomorphic in $\epsilon$.  Suppose that $\operatorname{ord}_\epsilon
A_{i,j}\geq\max\{i-j,0\}$.  For $0\leq j\leq i$, define
\[
a_{i,j}:=\lim_{\epsilon\to0} \frac{A_{i,j}(\epsilon)}{\epsilon^{i-j}}.
\]
The order assumption ensures that these limits exist. A limit is zero
when the vanishing order is strictly greater than $i-j$. We then
substitute $t=\epsilon U$, $x=\epsilon^2X$, $y=\epsilon^3Y$ into
Equation~\eqref{eqn:before_limit}, divide by $\epsilon^6$, and let
$\epsilon$ tend to zero. We obtain the Weierstrass equation
\begin{equation}
\label{eq:uniform-SW-limit}
 Y^2=X^3+A_2^{\rm lim}(U)X^2+A_4^{\rm lim}(U)X+A_6^{\rm lim}(U),
 \qquad A_i^{\rm lim}(U)=\sum_{j=0}^i a_{i,j}U^j.
\end{equation}
We homogenize $A_i^{\rm lim}$ as a section of $\mathcal O_{\mathbb
  P^1}(i)$.  Here we will always assume that the resulting global
Weierstrass equation is minimal, has fundamental line bundle $\mathcal
O_{\mathbb P^1}(1)$, has nonzero discriminant, and has smooth fiber
over $U=\infty$.
\begin{lemma}
\label{lem:uniform-SW-scaling}
The minimal resolution of~\eqref{eq:uniform-SW-limit} is a rational
elliptic surface.  For a modular coefficient $M_k$ whose first nonzero
Fourier--Jacobi term has degree $d$, the corresponding normalized
coefficient is
\begin{equation}
\label{eq:uniform-Jacobi-initial-coefficient}
 \phi_{k-12d,d} =\lim_{p\to0}\frac{M_k(\tau,\rho,\mathbf
   z)}{(\Delta(\tau)p)^d} =\Delta^{-d}[p^d]M_k,\qquad \epsilon=\Delta(\tau)
 p.
\end{equation}
Here $[p^d]M_k$ denotes the coefficient of $p^d$, and the limit is
taken with $(\tau,\mathbf z)$ fixed and $\Delta(\tau)\ne0$.  The
right-hand side is a weakly holomorphic Jacobi form of weight $k-12d$
and index $d$.
\end{lemma}
\begin{proof}
A monomial $\epsilon^dt^jx^\ell$ on the right of
\eqref{eqn:before_limit} has order $d+j+2\ell$ after
rescaling. Division by $\epsilon^6$ keeps the terms of order
six. Equivalently, the coefficient of $U^j$ in
$\epsilon^{-i}A_i(\epsilon U,\epsilon)$ is
$\epsilon^{j-i}A_{i,j}(\epsilon)$; terms with $j>i$ tend to zero.
This gives \eqref{eq:uniform-SW-limit}.  By hypothesis its fundamental
line bundle is $\mathcal O_{\mathbb P^1}(1)$, so the canonical bundle
formula gives a rational elliptic surface
\cite{MirandaEllipticSurfaces}.  At infinity one uses \(X/U^2,Y/U^3\) as 
fiber coordinates. Finally, \([p^d]M_k\) is a Jacobi form of weight \(k\) 
and index \(d\). Division by \(\Delta^d\) changes the weight by \(-12d\), 
leaves the index unchanged, and can introduce poles only at the cusp in the
\(\tau\)-variable. Hence the quotient is weakly holomorphic.
\end{proof}
\par We apply Lemma~\ref{lem:uniform-SW-scaling} to
Sakai's F-theory background \eqref{eq:Sakai-K3-curve}, using \eqref{eq:Sakai-leading-FJ}, and set
\begin{equation}
\label{eq:Sakai-scaling}
u=\epsilon U,\qquad x=\epsilon^2X,\qquad y=\epsilon^3Y, \qquad
\epsilon=\Delta(\tau)p.
\end{equation}
After division by $\epsilon^6$, the term $4u^7$ is $4\epsilon U^7$ and
vanishes in the limit.  We obtain the rational elliptic surface
\begin{equation}
\label{eq:E-string-SW-curve}
\begin{split}
R^{\mathrm{loc}}_{E_8}\colon\quad Y^2=4X^3
&-\bigl(a_0U^4+a_2U^2+a_3U+a_4\bigr)X\\ &-\bigl(b_0U^6+b_1U^5+b_2U^4+b_3U^3
+b_4U^2+b_5U+b_6\bigr).
\end{split}
\end{equation}
Its generic fiber over $\mathbb C(U)$ is the E-string Seiberg--Witten
curve.  For general Wilson lines the elliptic surface has twelve
fibers of type $\mathrm I_1$ and Mordell--Weil lattice $E_8$
\cite{EguchiSakaiEString,EguchiSakaiEStringRevisited,SakaiEnJacobi}.
\par The strict transform in the unrescaled chart is obtained by
setting $p=0$:
\begin{equation}
\label{eq:Sakai-unrescaled-nonminimal}
y^2=4x^3-a_0u^4x+4u^7-b_0u^6.
\end{equation}
Removing the nonminimal $(4,6,12)$ factor by $x=u^2X$, $y=u^3Y$ gives
the complementary rational elliptic surface
\begin{equation}
\label{eq:Sakai-unrescaled-RES}
R^\infty_{\mathrm{II}^*}\colon \quad
Y^2=4X^3-\frac{E_4}{12}X+4u-\frac{E_6}{216}.
\end{equation}
For general $\tau$, it has singular fibers $\mathrm{II}^*+2\mathrm
I_1$ and trivial Mordell--Weil group.  The fiber of
$R^\infty_{\mathrm{II}^*}$ over $u=0$ and the fiber of
$R^{\mathrm{loc}}_{E_8}$ over $U=\infty$, using $x=X/U^2$ and
$y=Y/U^3$, are the same smooth elliptic curve
\begin{equation}
\label{eq:E-string-gluing-fiber}
E_\tau:\quad y^2=4x^3-\frac{E_4(\tau)}{12}x-\frac{E_6(\tau)}{216}
\end{equation}
with discriminant $\Delta(\tau)$.
\par We fix \((\tau,\mathbf z)\) and assume \(\Delta(\tau)\ne0\). We
blow up the \((u,\epsilon)\)-base space of \eqref{eq:Sakai-K3-curve},
with \(\epsilon=\Delta p\), at \((u,\epsilon)=(0,0)\). After a finite
base change if necessary, simultaneous resolution of the Weierstrass
singularities, and relative minimalization, we obtain the following:
\begin{proposition}
\label{prop:E-string-stable-pair}
The central fiber is the two-component normal-crossing surface of
Kulikov Type~II, given by
\begin{equation}
\label{eq:E8-Type-II-pair}
R^{\mathrm{loc}}_{E_8}\cup_{E_\tau}R^\infty_{\mathrm{II}^*},
\end{equation}
where the components and double curve are
\eqref{eq:E-string-SW-curve}, \eqref{eq:Sakai-unrescaled-RES}, and
\eqref{eq:E-string-gluing-fiber}, respectively.
\end{proposition}
\begin{proof}
The two affine charts of the base blow-up are the rescaled and
unrescaled charts above. Near the generic point of the double curve,
the smoothing parameter is the product of local equations of the two
components, so the central fiber is simple normal crossing. The two
components meet along $E_\tau$. On each rational elliptic surface,
this curve is a smooth anticanonical fiber and has trivial normal
bundle. It follows that
\[
 N_{E_\tau/R^{\mathrm{loc}}_{E_8}}\otimes
 N_{E_\tau/R^\infty_{\mathrm{II}^*}}\cong\mathcal O_{E_\tau},
\]
so the union is \(d\)-semistable and has trivial dualizing
sheaf. Moreover,
\[
 \chi\!\left(\mathcal O_{R^{\mathrm{loc}}_{E_8}\cup
   R^\infty_{\mathrm{II}^*}}\right) =\chi(\mathcal
 O_{R^{\mathrm{loc}}_{E_8}}) +\chi(\mathcal
 O_{R^\infty_{\mathrm{II}^*}})-\chi(\mathcal O_{E_\tau})=2.
\]
The explicit family gives a smoothing. Semistable reduction and the
classification of Kulikov degenerations identify the resulting model
as Type~II: its two rational components meet along the smooth elliptic
curve \(E_\tau\)
\cite{KulikovDegenerations,PerssonPinkham,FriedmanSmoothings,
  ClingherMorganFTheory}.
\end{proof}
\begin{remark}
An elliptically fibered K3 carries $24$ units of discriminant
multiplicity. In a stable degeneration these separate into two blocks
of twelve, carried by the two rational elliptic components; their
common fiber is the elliptic curve of the dual heterotic torus
\cite{ClingherMorganFTheory,MorrisonVafaI,MorrisonVafaII,VafaFTheory}.
For Hashimoto--Ueda, the unrescaled block contains the $\mathrm{II}^*$
fiber and two nodal fibers, contributing $10+2$; the rescaled block
contains twelve nodal fibers. Equivalently, the scaling keeps twelve
of the fourteen nonfixed discriminant points at finite rescaled base
coordinates, while two remain on the other component. The resulting
local probe geometry is described by the Seiberg--Witten curve of the
E-string compactified on $T^2$
\cite{SakaiFTheoryK3,EguchiSakaiEString,GanorMorrisonSeiberg}.
\end{remark}
\section{Self-mirror fibrations and Type II boundary components}
\label{sec:self-mirror-boundary}
The two genus-one fibrations used in the transition from Vinberg's
$U\oplus D_8(-1)$-polarized family to the Hashimoto--Ueda $U\oplus
E_8(-1)$-polarized family can be given an interpretation in terms of Dolgachev's
lattice-theoretic mirror symmetry for K3 surfaces. We then analyze the corresponding Type~II
boundary geometry. Their rescaled rational elliptic components realize E-string
Seiberg--Witten curves with Wilson-line parameters.
\subsection{The mirror lattice at a cusp}
\label{subsec:mirror-lattice-at-cusp}
Let \(M\hookrightarrow\Lambda_{\mathrm{K3}}\) be a primitive embedding
of a hyperbolic lattice, and set \(T=M^\perp\). Dolgachev's mirror
construction begins with a primitive isotropic vector \(e\in T\) of
divisibility one. Such a vector can be completed to a primitive
hyperbolic plane \(U_0\subset T\), giving a splitting
\[
T\cong U_0\oplus\check M.
\]
The lattice \(\check M\) is the mirror lattice associated with the
chosen isotropic vector \(e\); see \cite{DolgachevMirror}.  The
isotropic line \(\mathbb Qe\) determines a zero-dimensional
Baily--Borel boundary component. A one-dimensional boundary component
incident to this cusp 
is determined by a primitive isotropic plane
\(I\subset T\) containing \(e\). Under the identification
$e^\perp/\mathbb Ze\cong\check M$, the quotient \(I/\mathbb Ze\)
determines a primitive isotropic line $\mathbb Zf\subset
e^\perp/\mathbb Ze\cong\check M$. 
After changing its sign and acting by the \((-2)\)-reflection group if
necessary, we may take \(f\) nef. It is then the fiber class of a
genus-one fibration on the mirror K3 surface. The lattices $M$
in~\eqref{eq:AN-chain} have the special property that, for an
appropriate choice of $e$, the mirror lattice $\check{M}$ is
isomorphic to the original $M$.
\par We now specialize to
\begin{equation}
\label{eq:MD-TD-self-mirror}
M_D=U\oplus D_8(-1),\qquad
T_D:=M_D^\perp\subset\Lambda_{\mathrm{K3}},\qquad T_D\cong U_0\oplus
M_D.
\end{equation}
That is, relative to the cusp determined by a primitive isotropic
vector \(e\in U_0\), we have $\check M_D\cong M_D$ and the
\(M_D\)-polarized family is self-mirror.  For the boundary-orbit
count, we work with the full orthogonal quotient $\mathrm
O^+(T_D)\backslash\mathcal D(T_D)$.  A one-dimensional Baily--Borel
boundary component incident to \(\mathbb Qe\) now determines, modulo
\(\mathbb Ze\), a primitive isotropic line
\[
\mathbb Zf\subset e^\perp/\mathbb Ze\cong M_D.
\]
Thus the incident one-dimensional boundary components are controlled
by the \(\mathrm O(M_D)\)-orbits of primitive isotropic vectors in
\(M_D\). These orbits, and the corresponding genus-one fibrations, are
described by the following lemma:
\begin{lemma}
\label{lem:MD-two-isotropic-orbits}
The primitive isotropic vectors in $L=U\oplus D_8(-1)$ form exactly
two orbits under $\mathrm O(L)$, distinguished by divisibility one and
two. Their orthogonal quotients are $D_8(-1)$ and $E_8(-1)$,
respectively.
\end{lemma}
\begin{proof}
Choose a basis $(p,q,d_1,\ldots,d_8)$, where $p^2=q^2=0$, $(p,q)=1$,
and the $d_i$ have the negative $D_8$ Cartan matrix: the chain is
$d_1,\ldots,d_6$, with $d_7,d_8$ attached to $d_6$.  Vinberg's
algorithm, as implemented in VinAl \cite{BogachevPerepechko2018},
gives the ten facet normals
\begin{equation*}
\begin{aligned}
r_i&=-d_i\quad(1\leq i\leq7), & r_8&=d_7-d_8,\\ r_9&=p-q, &
r_{10}&=q+d_1+2d_2+2d_3+2d_4+2d_5+2d_6+d_7+d_8.
\end{aligned}
\end{equation*}
The output can be checked directly. All ten reflections preserve $L$;
$r_8^2=-4$ and the other normals have square $-2$.  Their non-right
dihedral angles are $\pi/3$ and $\pi/4$.  The dual vectors defined by
$(z_i,r_j)=-\delta_{ij}$ lie in the closure of the same positive
cone. Their squared norms, in order, are
\begin{equation*}
(z_1^2,\ldots,z_{10}^2)=(1,6,5,4,3,2,1,0,0,2).
\end{equation*}
Thus the chamber is a finite-volume Coxeter simplex with ten facets,
eight finite vertices, and two ideal vertices. In the chosen basis, the ideal rays are generated
by
\begin{equation*}
\begin{aligned}
v_1&=(0,-1,0,0,0,0,0,0,0,0),\\ v_2&=(-1,-1,-\tfrac12,-1,-\tfrac32,-2,-\tfrac52,-3,-\tfrac32,-2).
\end{aligned}
\end{equation*}
The reflection group therefore has two cusp orbits in the chosen
positive cone, so $\mathrm O(L)$ has at most two primitive isotropic
orbits. The primitive integral generators are $f_1=v_1$ and
$f_2=2v_2$; direct pairing with the basis gives
$\operatorname{div}(f_1)=1$, $\operatorname{div}(f_2)=2$, and
$(f_1,f_2)=2$. Divisibility prevents their identification, while
$-\mathrm{id}$ identifies opposite rays. Hence there are exactly two
orbits. The first quotient is visibly $D_8(-1)$.
\par For the second, reflection in $r_8$ sends $f_2$ to
$g=f_2-r_8$. Both vectors have divisibility two, and $(f_2,g)=2$, so
their span $S=\mathbb Zf_2+\mathbb Zg$ has Gram matrix
$\bigl(\begin{smallmatrix}0&2\\2&0\end{smallmatrix}\bigr)$.  In the
  chosen basis one has
\[
 ((d_7,f_2),(d_7,g))=(0,2), \qquad ((d_8,f_2),(d_8,g))=(2,0).
\]
All pairings of elements of $L$ with $f_2$ and $g$ are even, while the
two equations above realize both generators of $2\mathbb Z^2$.  Hence
the images of $L$ and $S$ in $S^\vee$ coincide. Every element of $L$
therefore differs from an element of $S$ by an element of $S^\perp$,
and
\[
 L=S\oplus S^\perp.
\]
The complement is even, negative definite, unimodular, and of rank
eight, so it is isometric to $E_8(-1)$. This proves the alternative
presentation $L\cong U(2)\oplus E_8(-1)$.
\end{proof}
\par The two fibrations for the family of \(U\oplus D_8(-1)\)-polarized K3 surfaces can now be given an interpretation in terms of Dolgachev's lattice-theoretic mirror symmetry for K3 surfaces.  Fix the cusp $\mathbb Qe$ and the full orthogonal quotient as above.
\begin{proposition}
\label{prop:mirror-boundary-fibration}
The two $\mathrm O(M_D)$-orbits of primitive isotropic vectors
correspond to two distinct one-dimensional Baily--Borel boundary
components incident to this cusp. On a very general $M_D$-polarized K3
surface \(X\), their nef generators determine the following types of
genus-one fibrations:
\begin{equation}
\label{eq:two-isotropic-orbits}
\begin{array}{c|c|c|c}
\operatorname{div}(f)&\text{presentation of }M_D &f^\perp/\mathbb
Zf&\text{fibration on }X\\ \hline 1&U\oplus
D_8(-1)&D_8(-1)&\text{Jacobian
  elliptic~\eqref{eq:H+D8-weierstrass_intro}}\rule{0pt}{2.5ex}\\ 2&U(2)\oplus
E_8(-1)&E_8(-1)&\text{index-two
  genus-one~\eqref{eq:alternate-genus-one-quartic}}.
\end{array}
\end{equation}
\end{proposition}
\begin{proof}
Rational boundary components of a type-IV quotient are indexed by
arithmetic orbits of isotropic lines and planes: the former give
zero-dimensional cusps and the latter one-dimensional Type~II
components \cite{Scattone}. An isotropic plane incident to $\mathbb
Qe$ projects through $e^\perp/\mathbb Ze\cong M_D$ to a primitive
isotropic line $\mathbb Zf\subset M_D$. The cusp stabilizer in the
full orthogonal group induces the full isometry action on these lines,
using the splitting $T_D=U_0\oplus M_D$.
Lemma~\ref{lem:MD-two-isotropic-orbits} therefore gives exactly two
incident boundary components, which are distinguished by their
nonisometric quotients $I^\perp/I$. After changing sign if necessary
and acting by the $(-2)$-reflection group, $f$ may be taken nef. A
primitive nef class of square zero on a K3 surface defines a genus-one
fibration, and the divisibility of $f$ is its index. Thus divisibility
one gives a section, whereas divisibility two gives a bisection and no
section on the very general member. Finally,
\begin{equation*}
 U\oplus D_8(-1)\cong U(2)\oplus E_8(-1),
\end{equation*}
and the fiber vectors in the two decompositions above lie in the
two orbits of Lemma~\ref{lem:MD-two-isotropic-orbits}.  Taking the
orthogonal quotient by the corresponding fiber class gives $D_8(-1)$
and $E_8(-1)$, respectively. The Kodaira configurations and
Mordell--Weil groups follow from the explicit models and the
Shioda--Tate formula; see Sections~\ref{sec:Vinberg} and
Appendix~\ref{sec:Hashimoto-Ueda}.
\end{proof}
\subsection{Relation to non-K3 loci}
\label{subsec:Type-II-stable-limits}
We show that the two one-dimensional boundary components in
Proposition~\ref{prop:mirror-boundary-fibration} correspond to two
irreducible components of the non-K3 coefficient locus.
To do so, we will compute the specializations of the standard and alternate
pencils directly from the two irreducible components of the non-K3
locus computed in \cite[Section~2.4 and 
Proposition~2.15]{ClingherMalmendierWilliams2026}.  
\par To distinguish
these curves from the modular coefficients $C_k$, we write
\begin{equation}
\label{eqn:component}
 \mathcal C_D:=C_1, \qquad \mathcal C_E:=C_2
\end{equation}
for the components denoted by $C_1$ and $C_2$ there.  The residual
surfaces after specialization are extremal rational elliptic surfaces
with fiber root lattices $E_8$ and $D_8$, respectively
\cite{MirandaEllipticSurfaces,MirandaPerssonExtremal,PerssonConfigurations}.
\par The first component of the non-K3 locus is
\begin{equation}
\label{eq:CE-boundary-locus}
 \mathcal C_D
 =V\bigl(G_8,G_{10},G_{12},H_8,H_{10},H_{12},H_{14},H_{16},H_{18}\bigr)
 \cong\mathbb P(4,6),
\end{equation}
with coordinates $[F_4:F_6]$; this is Equation~(2.16) of
\cite{ClingherMalmendierWilliams2026}.  The second component $\mathcal
C_E$ is defined by seventeen weighted homogeneous generators given in
\cite[Lemma~2.13 and Appendix~B,
  Equations~(B.1)--(B.4)]{ClingherMalmendierWilliams2026}.  On the
character cover on which $H_4$ is defined, a parametrization by
$[\alpha:F_4]\in\mathbb P(2,4)$ is
\begin{equation}
\label{eq:CD-boundary-parametrization-I}
\begin{aligned}
 H_4&=\alpha^2+\frac13F_4, &F_6&=10\alpha^3+4\alpha
 F_4,\\ G_8&=\frac13F_4^2-15\alpha^4-4\alpha^2F_4,
 &G_{10}&=-12\alpha^5+4\alpha^3F_4+\frac83\alpha
 F_4^2,\\ G_{12}&=30\alpha^6+26\alpha^4F_4+\frac{16}{3}\alpha^2F_4^2,
 &H_8&=-\alpha^4-\frac23\alpha^2F_4-\frac19F_4^2,\\ H_{10}&=-8\alpha^5-\frac{16}{3}\alpha^3F_4-\frac89\alpha
 F_4^2.
\end{aligned}
\end{equation}
The remaining coefficients $H_{12}, \dots, H_{18}$ are given in
\cite[Equation~(2.19)]{ClingherMalmendierWilliams2026}; in particular,
we have $H_8=-H_4^2$.
\par The boundary lattice labels in~\eqref{eqn:component} were
assigned as follows: at the standard one-dimensional cusp of
$T_D=2U\oplus D_8(-1)$, the Fourier--Jacobi expansions of $F_4,F_6$,
of $G_8,G_{10},G_{12}$, and of $H_8,\ldots,H_{18}$ begin in orders
zero, one, and two, respectively
\cite[Corollary~5.1]{ClingherMalmendierWilliams2026}.  Thus the
boundary value lies on $\mathcal C_D$, with coordinates $[F_4:F_6]$,
and $\mathcal C_D$ has orthogonal quotient $D_8(-1)$.  The other
boundary curve $\mathcal C_E$ has quotient $E_8(-1)$.  
\par Starting with $\mathcal C_E$, we convert the standard
fibration~\eqref{eq:H+D8-weierstrass_intro} to short Weierstrass form
and substitute \eqref{eq:CD-boundary-parametrization-I}.  One obtains
\begin{equation}
\label{eq:CD-short-Weierstrass-specialization}
\begin{aligned}
Y^2&=X^3 -\frac13(u+\alpha)^4 \bigl(u^2-4\alpha
u+10\alpha^2+2F_4\bigr)X\\ &-\frac2{27}(2\alpha-u)(u+\alpha)^6
\bigl(u^2-4\alpha u+13\alpha^2+3F_4\bigr).
\end{aligned}
\end{equation}
For general $[\alpha:F_4]$, the Weierstrass equation acquires a
non-minimal $(4,6,12)$ point at $u=-\alpha$ on the base.  We set
$d=u+\alpha$, and the substitution $X=d^2X_1$, $Y=d^3Y_1$ removes the
common $(4,6,12)$ factor and produces a rational elliptic surface. After rescaling, relabeling variables, and setting
\begin{equation*}
v=3u,\qquad a=-6\alpha,\qquad b=-6(3\alpha^2+F_4),
\end{equation*}
this becomes the rational elliptic surface
\begin{equation}
\label{eq:D8-residual-RES}
R^\infty_{\mathrm I_4^*}\colon\quad y^2=x^3+(v+a)x^2+bx, \qquad
\Delta=16b^2\bigl((v+a)^2-4b\bigr),
\end{equation}
with \(b\ne0\). For general \([\alpha:F_4]\), its singular fibers are
\(\mathrm I_4^*+2\mathrm I_1\), and
\(\operatorname{MW}(R^\infty_{\mathrm I_4^*})
\cong\mathbb Z/2\mathbb Z\). 
The section \((x,y)=(0,0)\) generates the order-two Mordell--Weil
group. Thus the standard pencil has the \(\mathrm I_4^*\) fiber in the
residual equation~\eqref{eq:D8-residual-RES}; see
\cite[Remark~2.14]{ClingherMalmendierWilliams2026}.
\par Next we consider $\mathcal C_D$. Substitution of
\eqref{eq:CE-boundary-locus} into the standard
fibration~\eqref{eq:H+D8-weierstrass_intro}
gives
\begin{equation}
\label{eq:CE-standard-nonisolated}
 y^2=x^2\bigl(x+u^3+F_4u+F_6\bigr),
\end{equation}
which is singular along the curve $x=y=0$; see \cite[Lemma~2.11 and
  Remark~2.12]{ClingherMalmendierWilliams2026}.  The original
projection to the $u$-line has singular generic fiber.  Instead, we
regard $x$ as the base coordinate and put $Y=y/x$ to remove the square
factor.  This gives a residual rational elliptic surface.  On $\mathcal C_D$ one has
$H_4=0$, so the alternate base coordinate $s=x-H_4u$ specializes to
$s=x$. Thus the change of ruling is the specialization of the
alternate pencil. With $s=t-F_6/7$, its equation factors as
\begin{equation}
\label{eq:CE-alternate-quartic-factorization}
 w^2=\frac1{343}(7t-F_6)^2 \bigl(7u^3+7F_4u+7t+6F_6\bigr).
\end{equation}
Removing the square factor by $y=7w/(7t-F_6)$ gives the same residual surface,
with $x=t-F_6/7$.  As an elliptic surface over the $v$-line, we set
$A=F_4$, $v=x$, $B=F_6$ and rename the fiber coordinates $(u,Y)$ as
$(x,y)$ to obtain the rational elliptic surface
\begin{equation}
\label{eq:E8-residual-RES}
 R^\infty_{\mathrm{II}^*}\colon \quad y^2=x^3+A x+(v+B), \qquad
 \Delta=-16\bigl(4A^3+27(v+B)^2\bigr).
\end{equation}
The general member has singular fibers $\mathrm{II}^*+2\mathrm I_1$
and trivial Mordell--Weil group.
\par With $f_{\mathrm{src}}$ denoting the fiber class of the pencil
being specialized, we have obtained the following residual rational
elliptic surfaces:
\begin{equation}
\label{eq:type-II-RES-table}
\begin{array}{c|c|c|c|c}
\operatorname{div}(f_{\mathrm{src}})&\text{coefficient curve}
&\text{pencil specialized}&\text{surface and fibers}
&\operatorname{MW}\\ \hline 1&\mathcal C_E&\text{standard}
&R^\infty_{\mathrm I_4^*}\colon\ \mathrm I_4^*+2\mathrm I_1 &\mathbb
Z/2\mathbb Z\\ 2&\mathcal C_D&\text{alternate}
&R^\infty_{\mathrm{II}^*}\colon\ \mathrm{II}^*+2\mathrm I_1 &0
\end{array}
\end{equation}
\subsection{Type II Kulikov limits}
\label{subsec:two-E-string-limits}
We now analyze geometric Type~II limits related to the Baily--Borel components in Proposition~\ref{prop:mirror-boundary-fibration}. A Kulikov model
of Type~II has a smooth elliptic double curve, and the Baily--Borel
compactification keeps the modular parameter of the elliptic double
curve but contracts the additional extension data present in a
toroidal compactification.
\par We apply the stable scaling of Section~\ref{ssec:SWcurve} to the two boundary curves of Section~\ref{subsec:Type-II-stable-limits}.  Note that these
are different transverse degenerations.  At $\mathcal C_D$, which
contains the common Fourier--Jacobi cusp of $T_D\subset T_E$, the
alternate pencil determines the two rational elliptic components.  At
$\mathcal C_E$, the standard pencil develops the isolated $(4,6,12)$
point of \eqref{eq:CD-short-Weierstrass-specialization}; an explicit
smoothing of that point produces the second Type~II central fiber.  In
both cases the rescaled component carries the Wilson-line
coefficients, while the unrescaled component is the extremal rational
elliptic surface found in~\eqref{eq:type-II-RES-table}.
\par We will use the superscript \(\mathrm{loc}\) for a rescaled
exceptional component and the superscript \(\infty\) for the
unrescaled strict-transform component. For an unrescaled component,
the subscript records its distinguished Kodaira fiber as
in~\eqref{eq:type-II-RES-table}. The symbol \(Q^{\mathrm{loc}}_{D_8}\)
denotes the binary-quartic component of the alternate \(D_8\) pencil,
whereas \(R^{\mathrm{loc}}_{E_n}\) denotes a rescaled Jacobian
elliptic surface with \(E_n\) Wilson-line data.  The component 
pairs that will be determined in this section are
\[
Q^{\mathrm{loc}}_{D_8}\cup_{E_\tau}R^\infty_{\mathrm{II}^*}, \qquad
R^{\mathrm{loc}}_{E_8}\cup_{E_\tau}R^\infty_{\mathrm I_4^*},
\]
and \(\operatorname{Jac}(Q^{\mathrm{loc}}_{D_8})\cong
R^{\mathrm{loc}}_{E_8}\). The same notation will be used in
Section~\ref{subsec:exceptional-E-string-limits} for the exceptional
pairs \(R^{\mathrm{loc}}_{E_7}\cup_{E_\tau}R^\infty_{\mathrm{II}^*}\)
and \(R^{\mathrm{loc}}_{E_6}\cup_{E_\tau}R^\infty_{\mathrm{II}^*}\).
\subsubsection{The common cusp and the alternate pencil}
At \(\mathcal C_D\), the specialization
\eqref{eq:CE-alternate-quartic-factorization} of the alternate
fibration gives the unrescaled surface \(R^\infty_{\mathrm{II}^*}\),
with singular fibers \(\mathrm{II}^*+2\mathrm I_1\)
in~\eqref{eq:type-II-RES-table}. We now keep the transverse terms to
construct its complementary component.  We fix $(\tau,\mathbf z)$, put
$p=e^{2\pi i\rho}$, and set $\epsilon=\Delta(\tau)p$. In the
normalization of Theorem~\ref{thm:Sakai-comparison}, we write
\begin{equation}
\label{eq:D8-leading-Jacobi-data}
\begin{aligned}
F_4&=f+O(\epsilon),& f&=-3E_4(\tau),\\ F_6&=g+O(\epsilon),&
g&=2E_6(\tau),\\ H_4&=\epsilon h+O(\epsilon^2),& G_k&=\epsilon
g_k+O(\epsilon^2)\quad(k=8,10,12),\\ H_k&=\epsilon^2h_k+O(\epsilon^3)
&& (k=8,10,12,14,16,18).
\end{aligned}
\end{equation}
Thus \(h=\Delta^{-1}[p]H_4\), \(g_k=\Delta^{-1}[p]G_k\), and
\(h_k=\Delta^{-2}[p^2]H_k\), with \(h_8=-h^2\).  These functions are
$W(D_8)$-invariant weak Jacobi forms, with the choice of $h$ fixed on
the character cover; see \cite[Section~5 and
  Corollary~5.1]{ClingherMalmendierWilliams2026}.  Their weights and
indices are
\begin{equation}
\label{eq:D8-Jacobi-weights}
\begin{array}{c|rrrrrrrrr}
 &h&g_8&g_{10}&g_{12}&h_{10}&h_{12}&h_{14}&h_{16}&h_{18}\\ \hline
  \mathrm{weight}&-8&-4&-2&0&-14&-12&-10&-8&-6\\ \mathrm{index}&1&1&1&1&2&2&2&2&2
\end{array}
\end{equation}
Their $E_8$-invariant combinations will be obtained below from the
construction of the relative Jacobian. The normalization of the nine
$D_8$ Jacobi forms is given in
Appendix~\ref{app:D8-Jacobi-normalizations}; see in particular
\eqref{eq:D8-h-theta}, \eqref{eq:D8-index-one-leading-explicit}, and
\eqref{eq:D8-index-two-leading-explicit}.
\par In the alternate pencil \eqref{eq:alternate-genus-one-quartic},
we set $s=\epsilon S$, $w=\epsilon W$, divide by $\epsilon^2$, and
take the limit as $\epsilon$ approaches zero.
\begin{proposition}
\label{prop:D8-alternate-E-string-limit}
The limiting surface is
\begin{equation}
\label{eq:D8-limiting-quartic}
Q^{\mathrm{loc}}_{D_8}\colon \quad
W^2=q_0(S)u^4+q_1(S)u^3+q_2(S)u^2+q_3(S)u+q_4(S),
\end{equation}
where
\begin{equation}
\label{eq:D8-limiting-quartic-coefficients}
\begin{aligned}
q_0&=2hS+h_{10},\\ q_1&=S^2+fh^2+g_8h+h_{12},\\ q_2&=(2fh+g_8)S+gh^2+g_{10}h+h_{14},\\ q_3&=fS^2+(2gh+g_{10})S+g_{12}h+h_{16},\\ q_4&=gS^2+g_{12}S+h_{18}.
\end{aligned}
\end{equation}
For general $(\tau,\mathbf z)$, the smooth relatively minimal model of
\eqref{eq:D8-limiting-quartic} is a rational elliptic surface and has
a section.
\end{proposition}
\begin{proof}
Substitution of \eqref{eq:D8-leading-Jacobi-data} in
\eqref{eq:alternate-quartic-coefficients}, followed by division by
$\epsilon^2$, gives the coefficients in
\eqref{eq:D8-limiting-quartic-coefficients}.  After homogenization,
\eqref{eq:D8-limiting-quartic} is a double cover of $\mathbb
P^1_{(u)}\times\mathbb P^1_{(S)}$ with branch divisor of bidegree
$(4,2)$, smooth for general parameters.  The double-cover
canonical-bundle formula gives $K_{Q^{\mathrm{loc}}_{D_8}}=-F_S$,
where $F_S$ is a fiber over the $S$-line.  Projection to the $u$-line
gives a conic over $\mathbb C(u)$, which has a rational point by
Tsen's theorem; hence the surface is rational.  Moreover,
$K_{Q^{\mathrm{loc}}_{D_8}}^2=0$, so this rational surface is not
minimal and therefore contains a $(-1)$-curve $C$. By adjunction,
$K_{Q^{\mathrm{loc}}_{D_8}}\cdot C=-1$, and hence $F_S\cdot
C=-K_{Q^{\mathrm{loc}}_{D_8}}\cdot C=1$. Thus $C$ is a section.
\end{proof}
\par By Proposition~\ref{prop:D8-alternate-E-string-limit}, the smooth
relatively minimal model of~\eqref{eq:D8-limiting-quartic} has a
section for general parameters.  Choosing such a section identifies
the surface with its relative Jacobian, but this identification is
\emph{not} canonical.  Let $I(S)$ and $J(S)$ be the binary-quartic
invariants of the $q_i$, using the convention of
\eqref{eq:I-J-of-s}. We write
\begin{equation}
\label{eq:D8-limiting-invariants}
\begin{aligned}
I(S)&=-3fS^4+3(6gh-g_{10})S^3+I_2S^2+I_1S+I_0,\\ J(S)&=-27gS^6+J_5S^5+\cdots+J_0.
\end{aligned}
\end{equation}
We can now derive explicit expressions for all eleven meromorphic
$W(E_8)$-invariant Jacobi coefficients of the relative-Jacobian
fibration in terms of their leading $D_8$ Jacobi data:
\begin{corollary}
\label{cor:D8-alternate-E-string-Jacobian}
The relative Jacobian of $Q^{\mathrm{loc}}_{D_8}$ in
\eqref{eq:D8-limiting-quartic} is the E-string surface
$R^{\mathrm{loc}}_{E_8}$ in~\eqref{eq:E-string-SW-curve}. For $f\ne0$
and $\delta=(6gh-g_{10})/(4f)$, the coefficients in
\eqref{eq:E-string-SW-curve} are
\begin{equation}
\label{eq:E8-Jacobi-from-D8-limit}
\begin{aligned}
a_n&=\frac43(-12)^{-2-3n}
\frac{I^{(4-n)}(\delta)}{(4-n)!},&&n=0,2,3,4,\\ b_n&=\frac4{27}(-12)^{-3-3n}
\frac{J^{(6-n)}(\delta)}{(6-n)!},&&n=0,1,\ldots,6.
\end{aligned}
\end{equation}
Here derivatives are taken with respect to $S$, and $a_0=E_4/12$,
$a_1=0$, and $b_0=E_6/216$.
\end{corollary}
\begin{proof}
The Jacobian elliptic fibration has the
Weierstrass form $\eta^2=\xi^3-I(S)\xi/3-J(S)/27$. Choose
$\lambda^2=-12$ and set
\begin{equation*}
S=\lambda^6U+\delta, \qquad \xi=\lambda^{14}x, \qquad
\eta=\lambda^{21}y/2.
\end{equation*}
Expansion gives \eqref{eq:E-string-SW-curve}, with $U$ as base
coordinate and coefficients given by
\eqref{eq:E8-Jacobi-from-D8-limit}. The formulas are first obtained on
$f\ne0$ and extend meromorphically across $f=0$.
Theorem~\ref{thm:Sakai-comparison} identifies these coefficients with
the normalized $W(E_8)$-invariant Jacobi coefficients.
\end{proof}
\par In the coordinates above, the unrescaled surface
\eqref{eq:E8-residual-RES} is
\begin{equation}
\label{eq:D8-unrescaled-E8-component}
R^\infty_{\mathrm{II}^*}\colon\quad V^2=u^3+fu+(s+g),\qquad V=w/s.
\end{equation}
Its fiber at $s=0$ and the fiber of $Q^{\mathrm{loc}}_{D_8}$ at
$S=\infty$ are both
\begin{equation}
\label{eq:D8-double-curve}
E_\tau\colon\quad V^2=u^3-3E_4u+2E_6.
\end{equation}
This is the same elliptic curve as \eqref{eq:E-string-gluing-fiber}:
for $\lambda^2=-12$, the change $u=\lambda^2x$, $V=6\lambda y$ relates
\eqref{eq:D8-double-curve} and \eqref{eq:E-string-gluing-fiber}.
\par For the alternate pencil, we blow up the $(s,\epsilon)$-base
space at $(s,\epsilon)=(0,0)$. After a finite base change if
necessary, simultaneous resolution of the Weierstrass singularities,
and relative minimalization, we obtain the following result:
\begin{proposition}
\label{prop:D8-common-cusp-stable-pair}
The central fiber is the two-component normal-crossing surface of
Kulikov Type~II, given by
\[
Q^{\mathrm{loc}}_{D_8}\cup_{E_\tau}R^\infty_{\mathrm{II}^*},
\]
where the components and double curve are
\eqref{eq:D8-limiting-quartic}, \eqref{eq:D8-unrescaled-E8-component},
and \eqref{eq:D8-double-curve}, respectively.
\end{proposition}
\begin{proof}
Set
\begin{equation}
\label{eq:D8-leading-ABC}
\begin{aligned}
A(u)&=u^3+fu+g,\qquad
B(u)=g_8u^2+g_{10}u+g_{12},\\ C(u)&=-h^2u^5+h_{10}u^4+h_{12}u^3+h_{14}u^2+h_{16}u+h_{18}.
\end{aligned}
\end{equation}
In the chart $\epsilon=sv$, $w=sV$, the terms surviving on the central
fiber are
\begin{equation}
\label{eq:D8-Type-II-chart}
\begin{split}
V^2={}&s(1+vhu)^3+A(u)(1+vhu)^2\\ &+vB(u)(1+vhu)+v^2C(u),\qquad
\epsilon=sv.
\end{split}
\end{equation}
The component $v=0$ is \eqref{eq:D8-unrescaled-E8-component}; the
component $s=0$ is the chart $S=1/v$, $V=vW$ of
\eqref{eq:D8-limiting-quartic}.  Their intersection is
\eqref{eq:D8-double-curve}.  The \(d\)-semistability and Type~II
properties follow as in Proposition~\ref{prop:E-string-stable-pair}.
\end{proof}
\begin{corollary}
\label{cor:D8-Jacobian-commutes-with-limit}
For the \(U\oplus D_8(-1)\)-polarized family, taking the
relative Jacobian of the alternate genus-one fibration
\eqref{eq:alternate-genus-one-quartic} commutes with the rescaled cusp
limit: the relative Jacobian of the limiting quartic
\eqref{eq:D8-limiting-quartic} is 
the rational elliptic surface~\eqref{eq:E-string-SW-curve};
equivalently, this rational surface is the rescaled cusp limit of the family of relative Jacobians~\eqref{eq:relative-jacobian-IJ}.
\end{corollary}
\begin{proof}
The invariants \(I\) and \(J\) in \eqref{eq:I-J-of-s} are polynomial
in the quartic coefficients and hence commute with specialization.
Applied to \eqref{eq:D8-limiting-quartic}, they give
\eqref{eq:D8-limiting-invariants} and the coefficients
\eqref{eq:E8-Jacobi-from-D8-limit}, which define the E-string surface
\eqref{eq:E-string-SW-curve}.
\end{proof}
\begin{remark}
\label{rem:D8-standard-conic-limit}
At the same cusp $\mathcal C_D$, the standard substitution $x=\epsilon
X$, $y=\epsilon Y$, with $u$ fixed, gives
\begin{equation}
\label{eq:D8-standard-conic-limit}
Y^2=A(u)X^2+B(u)X+C(u),
\end{equation}
since the cubic term $\epsilon X^3$ vanishes on the central
fiber. Thus the original projection to $u$ is a conic fibration in the
limit.  The limiting neighbor coordinate $S=X-hu$ transforms
\eqref{eq:D8-standard-conic-limit} into
\eqref{eq:D8-limiting-quartic}. This is the same rational surface,
with its elliptic ruling restored.  In the standard ruling, a general
elliptic fiber specializes to two rational curves meeting at the two
points of \(E_\tau\to\mathbb P^1_{(u)}\); see
\cite[Section~3]{ClingherMorganFTheory}. Thus
\eqref{eq:D8-standard-conic-limit} is the conic-fibration on
$Q^{\mathrm{loc}}_{D_8}$ obtained by keeping the original ruling,
whereas \eqref{eq:D8-limiting-quartic} is its elliptic ruling.
\end{remark}
\subsubsection{The standard pencil and the second Type II central fiber}
At \(\mathcal C_E\), Section~\ref{subsec:Type-II-stable-limits}
extracted the surface \(R^\infty_{\mathrm I_4^*}\), with singular
fibers \(\mathrm I_4^*+2\mathrm I_1\), from the isolated \((4,6,12)\)
point of \eqref{eq:CD-short-Weierstrass-specialization}. We now smooth
that point while preserving the standard elliptic ruling.  We set
\begin{equation}
\label{eq:CE-gluing-data}
f_\tau=-\frac{a_0}{4}=-\frac{E_4}{48},\qquad
g_\tau=-\frac{b_0}{4}=-\frac{E_6}{864},\qquad
\alpha^3+f_\tau\alpha=g_\tau.
\end{equation}
The choice of $\alpha$ marks the nonzero two-torsion point
$P=(-\alpha,0)$ on
\begin{equation}
\label{eq:CE-marked-gluing-curve}
 E_\tau\colon\quad y^2=x^3+f_\tau x+g_\tau.
\end{equation}
For fixed $E_\tau$, the three roots of $\alpha^3+f_\tau\alpha=g_\tau$
correspond to its three nonzero two-torsion points. Up to
Weierstrass scaling, the pair $(E_\tau,P)$ is described by
$[\alpha:f_\tau]\in\mathbb P(2,4)$, with
$g_\tau=\alpha^3+f_\tau\alpha$. Thus on the marked coefficient cover,
the boundary datum consists of the gluing elliptic curve together with
this two-torsion marking; the unmarked Baily--Borel point remembers
only the elliptic curve. We define
\begin{equation}
\label{eq:CE-outer-short-coefficients}
\begin{aligned}
f_\infty(t)&=-t^2/3+2\alpha t+f_\tau,\\ g_\infty(t)&=2t^3/27-2\alpha
t^2/3 +(\alpha^2-f_\tau/3)t+g_\tau.
\end{aligned}
\end{equation}
For a smoothing parameter $\varepsilon$, we introduce
\begin{equation}
\label{eq:CE-E-string-smoothing}
\begin{aligned}
y^2&=x^3+\mathcal F_\varepsilon(t)x+\mathcal
G_\varepsilon(t),\\[0.6em] \text{with} \quad &\mathcal
F_\varepsilon(t)=t^4f_\infty(t)
-\tfrac14(\varepsilon^2a_2t^2+\varepsilon^3a_3t+\varepsilon^4a_4),\\ &\mathcal
G_\varepsilon(t)=t^6g_\infty(t)
-\frac14\sum_{j=1}^6\varepsilon^jb_jt^{6-j}.
\end{aligned}
\end{equation}
\par We blow up the $(t,\varepsilon)$-base space at
$(t,\varepsilon)=(0,0)$. After a finite base change if necessary,
simultaneous resolution of the Weierstrass singularities, and relative
minimalization, we obtain the following result:
\begin{proposition}
\label{prop:CE-standard-E-string-limit}
For general parameters and \(\varepsilon\ne0\),
\eqref{eq:CE-E-string-smoothing} is a member of the standard \(U\oplus
D_8(-1)\) family. The central fiber is the two-component
normal-crossing surface of Kulikov Type~II, given by
\[
R^{\mathrm{loc}}_{E_8} \cup_{E_\tau}R^\infty_{\mathrm I_4^*},
\]
where \(R^{\mathrm{loc}}_{E_8}\) is given in~\eqref{eq:E-string-SW-curve}, and
\begin{equation}
\label{eq:CE-unrescaled-extremal-curve}
R^\infty_{\mathrm I_4^*}\colon\quad
y_0^2=x_0^3+(t-3\alpha)x_0^2+(3\alpha^2+f_\tau)x_0
\end{equation}
has singular fibers \(\mathrm I_4^*+2\mathrm I_1\) and Mordell--Weil
group \(\mathbb Z/2\mathbb Z\), and the components meet along the elliptic
curve $E_\tau$ in~\eqref{eq:CE-marked-gluing-curve}.
\end{proposition}
\begin{proof}
To verify membership in the standard family, set $u=t-\alpha$ and
\[
 F_4=-\frac{3f_\tau+15\alpha^2}{2},\qquad F_6=10\alpha^3+4\alpha
 F_4,\qquad A(u)=u^3+F_4u+F_6.
\]
Undoing the short Weierstrass translation recovers the polynomials
\begin{equation}
\label{eq:CE-recover-standard-polynomials}
G(u)=\mathcal F_\varepsilon(t)+A(u)^2/3,\qquad H(u)=\mathcal
G_\varepsilon(t)+A(u)G(u)/3-2A(u)^3/27.
\end{equation}
Direct expansion gives $\deg_uG\le2$ and $\deg_uH\le5$, so the
resulting equation is \eqref{eq:H+D8-weierstrass_intro}. At
$\varepsilon=0$ it is exactly the specialization
\eqref{eq:CD-short-Weierstrass-specialization}.  For general
$\varepsilon\ne0$, the short coefficients have degrees $6,9$, and
$4\mathcal F_\varepsilon^3+27\mathcal G_\varepsilon^2$ has degree
$14$. Thus the K3 vanishing orders at infinity are $(2,3,10)$, giving
the fiber type $\mathrm I_4^*$.
\par The rescaling $t=\varepsilon U$, $x=\varepsilon^2X$,
$y=\varepsilon^3Y/2$ and division by $\varepsilon^6$ give
\eqref{eq:E-string-SW-curve}.  In the unrescaled chart, divide by the
common $(4,6,12)$ factor using $x=t^2x_1$, $y=t^3y_1$.  The result is
$y_1^2=x_1^3+f_\infty(t)x_1+g_\infty(t)$.  The single translation
$x_1=x_0+t/3-\alpha$, together with the notation $y_0=y_1$, gives
\eqref{eq:CE-unrescaled-extremal-curve}.  This is
\eqref{eq:D8-residual-RES}, so its fiber configuration and Mordell--Weil
group are those computed in
Section~\ref{subsec:Type-II-stable-limits}.  At $t=0$, the inverse
translation identifies the fiber with
\eqref{eq:CE-marked-gluing-curve} and the section
$(x_0,y_0)=(0,0)$ with $P=(-\alpha,0)$.
\par Finally, put $\varepsilon=tv$, $x=t^2x_1$, and $y=t^3y_1$ in the
blow-up chart.  After division by $t^6$ the exact equation is
\begin{equation*}
y_1^2={}x_1^3+
\left(f_\infty(t)-\tfrac14(a_2v^2+a_3v^3+a_4v^4)\right)x_1
+g_\infty(t)-\tfrac14\sum_{j=1}^6 b_jv^j, \qquad \varepsilon=tv.
\end{equation*}
The central fiber has components \(v=0\) and \(t=0\). The component
\(v=0\) is \(R^\infty_{\mathrm I_4^*}\). On the component \(t=0\), the
substitutions \(U=1/v\), \(x_1=v^2X\), and \(y_1=v^3Y/2\) give
\eqref{eq:E-string-SW-curve}, with the same \(a_i,b_j\). Their
intersection \(t=v=0\) is \(E_\tau\), namely the fiber at \(t=0\) on
\(R^\infty_{\mathrm I_4^*}\) and the fiber at \(U=\infty\) on
\(R^{\mathrm{loc}}_{E_8}\). As in
Proposition~\ref{prop:E-string-stable-pair}, this explicit smoothing
is \(d\)-semistable and of Type~II.
\end{proof}
\section{Exceptional index-two Jacobian transitions of types
\texorpdfstring{$E_7$ and $E_6$}{E7 and E6}}
\label{sec:E7-E6-Jacobians}
The transition from the $D_8$ family to the Hashimoto--Ueda $E_8$
family is based on the lattice isomorphism $U\oplus D_8(-1)\cong
U(2)\oplus E_8(-1)$.  Its essential arithmetic input is the index-two
inclusion $D_8\subset E_8$, together with the fact that the new
fibration is a genus-one torsor of index two.  There are two further
exceptional inclusions of the same kind:
\begin{equation}
\label{eq:exceptional-index-two-inclusions}
A_7\subset E_7, \qquad A_5\oplus A_1\subset E_6, \qquad
[E_7:A_7]=[E_6:A_5\oplus A_1]=2.
\end{equation}
The determinants on the two sides of each inclusion differ by a factor
of four.  The corresponding transcendental lattices are
\begin{equation}
\label{eq:exceptional-transcendental-lattices}
\begin{aligned}
T_{A_7}&=U^{\oplus2}\oplus A_7(-1) \subset T_7:=U^{\oplus2}\oplus
E_7(-1),\\ T_{A_5+A_1}&=U^{\oplus2}\oplus A_5(-1)\oplus A_1(-1)
\subset T_6:=U^{\oplus2}\oplus E_6(-1).
\end{aligned}
\end{equation}
Both are index-two inclusions of even lattices.  As in
Equation~\eqref{eqn:lattices}, a genus-one torsor with transcendental
lattice on the left can have a relative Jacobian with transcendental
lattice on the right.  The corresponding N\'eron--Severi lattices are
\begin{equation}
\label{eq:exceptional-Picard-lattices}
S_7=U\oplus E_8(-1)\oplus A_1(-1), \qquad S_6=U\oplus E_8(-1)\oplus
A_2(-1).
\end{equation}
Wang and Williams proved that the corresponding rings of modular forms
are free \cite[Theorem~5.9]{WangWilliamsJacobian}; see also
\cite{WangClassificationFree}.  Theorems~\ref{thm:E7-universal-model}
and~\ref{thm:E6-universal-model} below will identify the geometric
coefficients with generating systems of these rings.  In this
notation, the rings are
\begin{equation}
\label{eq:E7-free-ring}
M_*(\mathrm O^+(T_7)) =\mathbb
C[P_4,P_6,P_{10},P_{12},P_{14},P_{16},P_{18},P_{22},P_{24},P_{30}],
\end{equation}
and
\begin{equation}
\label{eq:E6-free-ring}
M_*(\widetilde{\mathrm O}^+(T_6)) =\mathbb
C[Q_4,Q_6,Q_7,Q_{10},Q_{12},Q_{15},Q_{16},Q_{18},Q_{24}].
\end{equation}
The subscripts indicate the modular weights.  In the \(E_7\) case, all
ten generators can be chosen to be orthogonal Eisenstein series. 
In the \(E_6\) case, Wang and Williams use odd-weight additive lifts
of holomorphic \(W(E_6)\)-invariant Jacobi forms of weights \(7\) and
\(15\), both of index one
\cite[Section~5.3]{WangWilliamsJacobian}; see
\cite{SakaiEnJacobi,Wirthmuller} for the associated weak Jacobi forms.
\subsection{The relative Jacobian of the \texorpdfstring{$A_7$}{A7}
family} \par We use the following lift normalizations from
\cite[Section~4.7, pp.~28--29]{ClingherMalmendierWilliams2026}.  Let
$\mathcal L_9$ and $\mathcal L_{10}$ denote the two additive lifts
appearing there in the formulas for $h_9$ and
$g_{10}$. Then
\begin{equation}
\label{eq:A7-corrected-lift-normalizations}
\begin{aligned}
 h_9&=-12^{17/2}\frac{175}{2}\chi_4\chi_5 -144\sqrt3\,\mathcal
 L_9,\\ g_{10}&=12^9\frac{175}{2}(\chi_5^2-\chi_4\chi_6)
 -288\,\mathcal L_{10}.
\end{aligned}
\end{equation}
The multipliers $-72$ and $-576$ in the reference are replaced by
$-144\sqrt3$ and $-288$. The symbols $\chi_j$ here are the $A_7$ forms
of that reference.
\par We start with the $A_7$ model in Equation~\eqref{eq:A7-family-2}.
We now regard $x$ as the base coordinate and $u$ as the coordinate on
the quartic fiber.  The same surface is then a genus-one fibration
\begin{equation}
\label{eq:A7-exchanged-quartic}
w^2=A_0(x)u^4+A_1(x)u^3+A_2(x)u^2+A_3(x)u+A_4(x),
\end{equation}
where
\begin{equation}
\label{eq:A7-exchanged-coefficients}
\begin{aligned}
A_0(x)&=g_4x+h_5^2,\\ A_1(x)&=x^2+g_6x+2h_5h_7,\\ A_2(x)&=g_8x+h_7^2+2h_5h_9,\\ A_3(x)&=f_4x^2+g_{10}x+2h_7h_9,\\ A_4(x)&=x^3+f_6x^2+g_{12}x+h_9^2.
\end{aligned}
\end{equation}
Let $\mathcal I(x)$ and $\mathcal J(x)$ be the binary-quartic
invariants, computed as in Equation~\eqref{eq:I-J-of-s}, and set
\begin{equation}
\label{eq:E7-D-definition}
D=h_7^2-4h_5h_9.
\end{equation}
At $x=0$ the quartic in \eqref{eq:A7-exchanged-quartic} is
$(h_5u^2+h_7u+h_9)^2$, and $\mathcal I(0)=D^2$, $\mathcal J(0)=-2D^3$.
The relative Jacobian has the Weierstrass form $\eta^2=\xi^3-\tfrac13\mathcal I(x)\xi -\tfrac1{27}\mathcal J(x)$.
After the translation $\xi=T+D/3$, it becomes
\begin{equation}
\label{eq:E7-long-from-invariants}
\eta^2=T^3+DT^2+B_7(x)T+C_7(x),
\end{equation}
where
\begin{equation}
\label{eq:E7-BC-from-invariants}
B_7(x)=\frac{D^2-\mathcal I(x)}{3}, \qquad C_7(x)=\frac{D^3-3D\mathcal
  I(x)-\mathcal J(x)}{27}.
\end{equation}
We obtain
\begin{equation}
\label{eq:E7-polynomial-shapes}
\begin{aligned}
B_7(x)&=P_4x^4+P_{10}x^3+P_{16}x^2+P_{22}x,\\ C_7(x)&=x^7+P_6x^6+P_{12}x^5+P_{18}x^4
+P_{24}x^3+P_{30}x^2,
\end{aligned}
\qquad P_{14}=D.
\end{equation}
Thus the relative Jacobian over $\mathbb{P}^1_{(x)}$ has the following
normal form:
\begin{equation}
\label{eq:E7-universal-model}
\begin{aligned}
\mathcal X_7:\quad \eta^2={}T^3&+P_{14}T^2
+(P_4x^4+P_{10}x^3+P_{16}x^2+P_{22}x)T\\ &+x^7+P_6x^6+P_{12}x^5+P_{18}x^4
+P_{24}x^3+P_{30}x^2.
\end{aligned}
\end{equation}
\begin{theorem}
\label{thm:E7-universal-model}
For very general coefficients, the minimal resolution of
\eqref{eq:E7-universal-model} is an $S_7$-polarized K3 surface with
\begin{equation*}
\operatorname{NS}(\mathcal X_7)\cong U\oplus E_8(-1)\oplus A_1(-1),
\qquad T(\mathcal X_7)\cong U^{\oplus2}\oplus E_7(-1).
\end{equation*}
Equation~\eqref{eq:E7-universal-model} has singular fibers
$\mathrm{II}^*+\mathrm I_2+12\mathrm I_1$ and trivial Mordell--Weil
group.  Its ten coefficients form a free generating system for
$M_*(\mathrm O^+(T_7))$, after normalization.
\end{theorem}
\begin{proof}
The degrees in \eqref{eq:E7-universal-model} give a Weierstrass
fibration over $\mathbb P^1$ with fundamental line bundle
$\mathcal O_{\mathbb P^1}(2)$. For very general coefficients it is
minimal, so the canonical bundle formula shows that its minimal
resolution is a K3 surface. Tate's algorithm at $x=\infty$ gives a
fiber of type $\mathrm{II}^*$. At $x=0$ the cubic is
$T^2(T+P_{14})$, and
\begin{equation}
\label{eq:E7-local-discriminant}
\operatorname{disc}_T
 \bigl(T^3+P_{14}T^2+B_7(x)T+C_7(x)\bigr)
 =P_{14}^2\bigl(P_{22}^2-4P_{14}P_{30}\bigr)x^2+O(x^3).
\end{equation}
Thus the fiber at zero is generically of type $\mathrm I_2$. The
residual discriminant is generically reduced: setting all $P_k=1$
makes its quotient by $x^2$ squarefree. The Euler number gives twelve
remaining fibers of type $\mathrm I_1$.
\par The fiber and zero section, together with the nonidentity components
of the reducible fibers, span \(S_7\). Since \(S_7\) has discriminant
\(2\), it has no proper integral overlattice. A
polarization-preserving isomorphism fixes the fibers over zero and
infinity; the constant \(T^2\)-coefficient and the singular point over
zero exclude translations in \(T\), while the monic \(x^7\)-term
leaves only
\[
 (x,T,\eta)\longmapsto
 (\lambda^6x,\lambda^{14}T,\lambda^{21}\eta).
\]
Thus the nine-dimensional coefficient quotient maps generically
finitely and dominantly to the nine-dimensional moduli space. The very
general member has N\'eron--Severi lattice \(S_7\), transcendental
lattice \(T_7\), and trivial Mordell--Weil group.
\par  By \cite[Theorem~1.1(ii) and
Section~4.7]{ClingherMalmendierWilliams2026}, the $A_7$ coefficient
forms are holomorphic on the affine period cone of $T_{A_7}$. The
finite-index inclusion $T_{A_7}\subset T_7$ identifies the two period
cones, and Equations~\eqref{eq:E7-D-definition} and
\eqref{eq:E7-BC-from-invariants}, together with
Appendix~\ref{app:E7-coefficient-dictionary}, express every $P_k$ as
a polynomial in those forms. Hence all $P_k$ are holomorphic. Global
Torelli and uniqueness of the normal form give their transformation
law under the pointwise stabilizer of $S_7$, which induces
$\mathrm O^+(T_7)$ because the discriminant group has order two. The
scaling of $dx\wedge dT/\eta$ gives modular weight $k$ for $P_k$.
The identity theorem and Koecher's principle give
$P_k\in M_k(\mathrm O^+(T_7))$.
\par Dominance of the affine-cone period map makes the ten $P_k$
algebraically independent. Their weights are those of the free
generators in \eqref{eq:E7-free-ring}. The polynomial ring that they
generate and the full modular-form ring have the same Hilbert series,
so they are equal.
\end{proof}
\subsection{The \texorpdfstring{$E_6$}{E6} Heegner divisor and the two
odd generators}
Equation~\eqref{eq:E7-local-discriminant} gives the enhancement from
$A_1$ to $A_2$.  Define the weight-$44$ polynomial
\begin{equation}
\label{eq:E7-Phi44-polynomial}
\mathcal B_{44}:=P_{22}^2-4P_{14}P_{30}.
\end{equation}
On the open set $P_{14}\neq0$, the condition $\mathcal B_{44}=0$
generically enhances the fiber at $x=0$ from $\mathrm I_2$ to $\mathrm
I_3$.  The resulting general K3 surface has N\'eron--Severi and
transcendental lattices $S_6$ and $T_6$, respectively.  On the graded
coefficient cone for the discriminant kernel, the rank-one condition
in \eqref{eq:E7-Phi44-polynomial} has the parametrization
\begin{equation}
\label{eq:E6-parametrization-in-E7}
P_{14}=Q_7^2, \qquad P_{22}=2Q_7Q_{15}, \qquad P_{30}=Q_{15}^2.
\end{equation}
Restricting the even generators by $P_k=Q_k$
in~\eqref{eq:E7-universal-model} gives the equation
\begin{equation}
\label{eq:E6-universal-model}
\begin{aligned}
\mathcal X_6:\quad \eta^2={} T^3 & +Q_7^2T^2
+(Q_4x^4+Q_{10}x^3+Q_{16}x^2+2Q_7Q_{15}x)T\\ &+x^7+Q_6x^6+Q_{12}x^5+Q_{18}x^4
+Q_{24}x^3+Q_{15}^2x^2.
\end{aligned}
\end{equation}
We will write $B_6(x)$ and $C_6(x)$ for the coefficient of $T$ and the
constant term, respectively, on the right side of
\eqref{eq:E6-universal-model}. Equivalently, we can shift $\eta=
\tilde{\eta} + Q_7 T + Q_{15} x$ to obtain
\begin{equation}
\label{eq:E6-universal-model_alt}
\begin{aligned}
\mathcal X_6:\quad \tilde\eta^2 + 2 \big(Q_{15} x + Q_7 T \big)
\tilde\eta & ={} T^3
+\big(Q_4x^2+Q_{10}x+Q_{16}\big)x^2T\\ &+\big(x^4+Q_6x^3+Q_{12}x^2+Q_{18}x
+Q_{24}\big)x^3,
\end{aligned}
\end{equation}
which shows that there is a suitable normal form where each generator
appears linearly.
\par The specialization from the $A_7$ equation to the $A_5\oplus A_1$
equation can also be made explicit: we set
\begin{equation*}
L(u)=h_5u+h'_7, \qquad G(u)=g_4u^3+g'_6u^2+g'_8u+g'_{10},
\end{equation*}
and assume that the last term and the coefficient of $x$ in
Equation~\eqref{eq:A7-family-2} are
\begin{equation*}
L(u)^2(u-a_2)^2, \qquad G(u)(u-a_2),
\end{equation*}
respectively, where $a_2$ is the position of the additional
$\mathrm{I}_2$-fiber in $\mathbb{P}^1_{(u)}$.  As an $A_7$
specialization, this is obtained by the restriction
\begin{equation}
\label{eq:A5A1-to-A7-specialization}
\begin{aligned}
(h_5,h_7,h_9)_{A_7}
  &=(h_5,h'_7-a_2h_5,-a_2h'_7),\\ (g_4,g_6,g_8,g_{10},g_{12})_{A_7}
  &=(g_4,g'_6-a_2g_4,g'_8-a_2g'_6,g'_{10}-a_2g'_8,-a_2g'_{10}).
\end{aligned}
\end{equation}
In particular, the discriminant in \eqref{eq:E7-D-definition} then
becomes
\begin{equation}
\label{eq:C14-becomes-square}
D=(h'_7+a_2h_5)^2.
\end{equation}
Moreover, for
\begin{equation}
\label{eq:Q7-Q15-source-formulas}
Q_7=L(a_2)=h'_7+a_2h_5, \qquad Q_{15}=h_5G(a_2)-\frac13Q_7G'(a_2),
\end{equation}
a direct substitution into the polarized binary-quartic invariants
gives again
\begin{equation}
\label{eq:source-proof-rank-one-relation}
P_{14}=Q_7^2, \qquad P_{22}=2Q_7Q_{15}, \qquad P_{30}=Q_{15}^2,
\end{equation}
in agreement with~\eqref{eq:E6-parametrization-in-E7}.
\begin{theorem}
\label{thm:E6-universal-model}
For very general coefficients, the minimal resolution of
\eqref{eq:E6-universal-model} is an $S_6$-polarized K3 surface with
\begin{equation*}
\operatorname{NS}(\mathcal X_6)\cong U\oplus E_8(-1)\oplus A_2(-1),
\qquad T(\mathcal X_6)\cong U^{\oplus2}\oplus E_6(-1).
\end{equation*}
Equation~\eqref{eq:E6-universal-model_alt} has singular fibers
$\mathrm{II}^*+\mathrm I_3+11\mathrm I_1$ and trivial Mordell--Weil
group.  Its nine coefficients form a free generating system for
$M_*(\widetilde{\mathrm O}^+(T_6))$, after normalization.
\end{theorem}
\begin{proof}
As in Theorem~\ref{thm:E7-universal-model}, the fundamental line
bundle is $\mathcal O_{\mathbb P^1}(2)$, so the minimal resolution is
a K3 surface, and Tate's algorithm gives a $\mathrm{II}^*$-fiber at
infinity. We set
\begin{equation}
\label{eq:E6-Phi45-polynomial}
\mathcal B_{45} =Q_{15}^3+Q_7^2Q_{15}Q_{16}-Q_7^3Q_{24}.
\end{equation}
At $x=0$
the lowest-order terms form $(Q_7T+Q_{15}x)^2$, while
\begin{equation}
\label{eq:E6-local-discriminant}
\operatorname{disc}_T
 \bigl(T^3+Q_7^2T^2+B_6(x)T+C_6(x)\bigr)
 =4Q_7^3\mathcal B_{45}x^3+O(x^4).
\end{equation}
Thus the fiber at zero is generically of type $\mathrm I_3$. Setting
all $Q_k=1$ makes the residual discriminant divided by $x^3$
squarefree. The eleven remaining fibers are of type
$\mathrm I_1$.
\par The fiber and zero section, together with the nonidentity components
of the reducible fibers, span \(S_6\). Since \(S_6\) has discriminant
\(3\), it has no proper integral overlattice. The normal-form argument
of Theorem~\ref{thm:E7-universal-model} applies, with at most the
simultaneous sign ambiguity in \(Q_7,Q_{15}\). Thus the
eight-dimensional coefficient quotient maps generically finitely and
dominantly to the eight-dimensional moduli space. The very general
member has N\'eron--Severi lattice \(S_6\), transcendental lattice
\(T_6\), and trivial Mordell--Weil group.
\par The even coefficients are restrictions of the holomorphic
forms $P_k$. On the dense locus described by
\eqref{eq:A5A1-to-A7-specialization}, the quadratic and quartic
coefficient polynomials in \eqref{eq:A7-family-2} have the unique
common root $a_2$. Euclidean division therefore expresses $a_2$, and
hence $L,G,Q_7,Q_{15}$, meromorphically in the coefficient forms. By
\eqref{eq:source-proof-rank-one-relation},
$Q_7^2=P_{14}|_{T_6}$ and $Q_{15}^2=P_{30}|_{T_6}$. Thus neither odd
coefficient has a divisorial pole, and normality extends both across
the complement.
\par The $S_6$-marking orders the two nonidentity components of the
$\mathrm I_3$-fiber and fixes the simultaneous sign of $Q_7,Q_{15}$.
Global Torelli and uniqueness of the normal form give invariance under
$\widetilde{\mathrm O}^+(T_6)$, while the differential calculation in
Theorem~\ref{thm:E7-universal-model} gives weight $k$ for $Q_k$. The
identity theorem and Koecher's principle give
$Q_k\in M_k(\widetilde{\mathrm O}^+(T_6))$.
\par Dominance of the affine-cone period map makes the nine $Q_k$
algebraically independent. Their weights are those of the free
generators in \eqref{eq:E6-free-ring}; equality of Hilbert series
shows that they generate the full ring freely.
\end{proof}
\par Equation~\eqref{eq:E6-universal-model} depends on $Q_7$ and
$Q_{15}$ only through their quadratic monomials.  This is because of
the difference between the full stabilizer of the $E_6$ Heegner
divisor and its discriminant-kernel cover.  In particular, if we let
$\epsilon$ denote the nontrivial element of $\mathrm
O(A_{T_6})\cong\mathbb Z/2\mathbb Z$, then it acts by
\begin{equation*}
\epsilon(Q_7,Q_{15})=(-Q_7,-Q_{15})
\end{equation*}
and fixes the seven even generators.  Therefore, we have
\begin{equation}
\label{eq:E7-divisor-as-E6-invariants}
\frac{M_*(\mathrm O^+(T_7))}{(\mathcal B_{44})} \cong
M_*(\widetilde{\mathrm O}^+(T_6))^{\langle\epsilon\rangle},
\end{equation}
where the homomorphism is given by
\eqref{eq:E6-parametrization-in-E7}.  In particular, the odd
generators are not individually visible on the unmarked $E_6$ divisor;
the relative-Jacobian construction naturally produces their
simultaneous sign cover on the graded coefficient cone. This does not
give a twofold cover of the coarse projective moduli space: the sign
change is the weighted scaling $\lambda=-1$, and
$-\mathrm{id}\in\mathrm O^+(T_6)$ acts trivially on the projective
period domain.
\subsection{Borcherds products and modular Jacobians}
The two local discriminant calculations identify distinguished Heegner
divisors.  A primitive vector $v\in T_7$ with
\begin{equation*}
v^2=-6, \qquad \operatorname{div}(v)=2
\end{equation*}
has $v^\perp\cong T_6$.  Similarly, a primitive vector $w\in T_6$ with
\begin{equation*}
w^2=-12, \qquad \operatorname{div}(w)=3
\end{equation*}
has $w^\perp\cong U^{\oplus2}\oplus D_5(-1)$.  These are the Heegner
orbits associated with the shortest nontrivial Weyl orbits occurring
in the weight-zero $E_7$- and $E_6$-invariant weak Jacobi forms. Their
Borcherds lifts have weights $44$ and $45$, respectively
\cite{Wirthmuller,GritsenkoReflective,WangWilliamsJacobian}.
\begin{corollary}
\label{cor:exceptional-Borcherds-polynomials}
Up to nonzero constants, the exceptional Borcherds products are
\begin{equation}
\label{eq:exceptional-Borcherds-identification}
\Phi_{44,E_7}\doteq P_{22}^2-4P_{14}P_{30}, \qquad \Phi_{45,E_6}\doteq
Q_{15}^3+Q_7^2Q_{15}Q_{16}-Q_7^3Q_{24}.
\end{equation}
The first cuts out the $E_6$ locus inside the $E_7$ modular variety.
On the open set $Q_7\neq0$, the second generically raises the fiber at
$x=0$ from $\mathrm I_3$ to $\mathrm I_4$; its general point 
has
\begin{equation*}
\operatorname{NS}\cong U\oplus E_8(-1)\oplus A_3(-1), \qquad T\cong
U^{\oplus2}\oplus D_5(-1).
\end{equation*}
\end{corollary}
\begin{proof}
Theorems~\ref{thm:E7-universal-model} and~\ref{thm:E6-universal-model}
identify the coefficient rings with the full modular-form rings.  In
these polynomial rings both polynomials in
\eqref{eq:exceptional-Borcherds-identification} are irreducible, so
each defines a reduced irreducible divisor.
Equations~\eqref{eq:E7-local-discriminant} and
\eqref{eq:E6-local-discriminant} give the lattice enhancements at general points of these divisors. 
Since both lattices contain $U^{\oplus2}$, Eichler's criterion shows
that the primitive hyperplanes with the associated norm and divisibility
form one orbit for the indicated group. Hence each Heegner divisor is
irreducible.
There is no generic divisorial ramification here: neither
the reflection in $v$ nor that in $w$ preserves the relevant lattice,
as $2\operatorname{div}(v)/|v^2|=2/3$ and
$2\operatorname{div}(w)/|w^2|=1/2$.  Thus the pullbacks to the period
domains have multiplicity one.
\par The Borcherds products have these same reduced divisors and weights.
The quotients are nowhere-vanishing holomorphic functions of weight
zero, possibly with a finite character.  Passing to a finite-index
subgroup kills that character.  The Baily--Borel boundary has
codimension at least two in both cases, so normality extends these
functions across the boundary.  Compactness makes them constant.  In
particular the characters agree, and the proportionalities hold for
the original groups.
\end{proof}
\par Finally, let $J_7$ and $J_6$ denote the modular Jacobians of the
ordered generator systems in Equations~\eqref{eq:E7-free-ring} and
\eqref{eq:E6-free-ring}.  Their weights are
\begin{equation*}
\operatorname{wt}(J_7) =(4+6+10+12+14+16+18+22+24+30)+9=165,
\end{equation*}
and
\begin{equation*}
\operatorname{wt}(J_6) =(4+6+7+10+12+15+16+18+24)+8=120.
\end{equation*}
The modular-Jacobian criterion and the strongly reflective products of
Gritsenko--Nikulin therefore give
\begin{equation}
\label{eq:E7-E6-modular-Jacobians}
J_7\doteq\Phi_{165,E_7}, \qquad J_6\doteq\Phi_{120,E_6};
\end{equation}
see \cite{GritsenkoNikulin2018,WangWilliamsJacobian}.  The four
numbers $44,45,165,120$ that arise from the local discriminants and
coefficient Jacobians are exactly the exceptional Borcherds weights.
\subsection{E-string limits}
\label{subsec:exceptional-E-string-limits}
After the rescaling of
Proposition~\ref{prop:exceptional-E-string-limits}, the leading terms
of the coefficients in \eqref{eq:E7-universal-model} and
\eqref{eq:E6-universal-model} give the $E_7$ and $E_6$
Seiberg--Witten normal forms of
\cite[Sections~3.3--3.4]{SakaiEnJacobi}. These K3 equations are
orthogonal modular extensions of the corresponding Jacobi families.
\par We use the common cusp and $\epsilon=\Delta p$. For the $E_7$
coefficients, we set
\begin{equation}
\label{eq:E7-leading-orders}
P_k=\epsilon^{m_k}\pi_k+O(\epsilon^{m_k+1}),\qquad
\begin{array}{c|rrrrrrrrrr}
k&4&6&10&12&14&16&18&22&24&30\\ \hline m_k&0&0&1&1&2&2&2&3&3&4
\end{array}
\end{equation}
where $\pi_4=-3E_4$, $\pi_6=2E_6$. Here $\pi_k$ has Jacobi
weight $k-12m_k$ and index $m_k$. Specifically, the eight
positive-index leading coefficients form a generating system over
$\mathbb C[E_4,E_6]$, with weight--index pairs
\begin{equation}
\label{eq:E7-Wirthmuller-leading-generators}
 (0,1),\ (-2,1),\ (-6,2),\ (-8,2),\ (-10,2),\ (-12,3),\ (-14,3),\ (-18,4).
\end{equation}
These are the $E_7$ degrees of Wirthm\"uller's theorem
\cite{Wirthmuller}. Similarly, we write
\begin{equation}
\label{eq:E6-leading-orders}
Q_k=\epsilon^{n_k}\kappa_k+O(\epsilon^{n_k+1}),\qquad
\begin{array}{c|rrrrrrrrr}
k&4&6&7&10&12&15&16&18&24\\ \hline n_k&0&0&1&1&1&2&2&2&3
\end{array}
\end{equation}
with $\kappa_4=-3E_4$, $\kappa_6=2E_6$. Here $\kappa_k$ has weight
$k-12n_k$ and index $n_k$.  For the base and fiber coordinates in
\eqref{eq:E7-universal-model} and \eqref{eq:E6-universal-model}, we
make the rescaling $x=\epsilon U$, $T=\epsilon^2X$,
$\eta=\epsilon^3Y$.  We then have the following:
\begin{proposition}
\label{prop:exceptional-E-string-limits}
The rescaled components are the limiting rational elliptic surfaces
\begin{equation}
\label{eq:E7-E-string-limit}
\begin{split}
R^{\mathrm{loc}}_{E_7}\colon \quad Y^2={}&X^3+\pi_{14}X^2
+(\pi_4U^4+\pi_{10}U^3+\pi_{16}U^2+\pi_{22}U)X\\ &+\pi_6U^6+\pi_{12}U^5+\pi_{18}U^4
+\pi_{24}U^3+\pi_{30}U^2,
\end{split}
\end{equation}
and
\begin{equation}
\label{eq:E6-E-string-limit}
\begin{split}
R^{\mathrm{loc}}_{E_6}\colon\quad Y^2={}&X^3+\kappa_7^2X^2
+(\kappa_4U^4+\kappa_{10}U^3+\kappa_{16}U^2
+2\kappa_7\kappa_{15}U)X\\ &+\kappa_6U^6+\kappa_{12}U^5+\kappa_{18}U^4
+\kappa_{24}U^3+\kappa_{15}^2U^2.
\end{split}
\end{equation}
Generically their singular fibers are $\mathrm I_2+10\mathrm I_1$ and
$\mathrm I_3+9\mathrm I_1$, respectively, with Mordell--Weil lattices
$E_7^\vee$ and $E_6^\vee$.  In both cases the unrescaled complementary
component is
\begin{equation}
\label{eq:exceptional-unrescaled-curve}
R^\infty_{\mathrm{II}^*}\colon \quad
Y_\infty^2=X_\infty^3-3E_4X_\infty+(x+2E_6),
\end{equation}
with singular fibers $\mathrm{II}^*+2\mathrm I_1$ for general $\tau$.
The corresponding normal-crossing surfaces of Kulikov Type~II are
\[
R^{\mathrm{loc}}_{E_7}\cup_{E_\tau}R^\infty_{\mathrm{II}^*}, \qquad
R^{\mathrm{loc}}_{E_6}\cup_{E_\tau}R^\infty_{\mathrm{II}^*}.
\]
\end{proposition}
\begin{proof}
We use the lift normalizations
\eqref{eq:A7-corrected-lift-normalizations}. The \(A_7\) coefficients
\(g_k,h_k\) are \(O(p)\), while \(f_4=-3E_4+O(p)\) and
\(f_6=2E_6+O(p)\)
\cite[Section~4.7]{ClingherMalmendierWilliams2026}. Their normalized
first coefficients \(\Delta^{-1}[p]g_k\) and
\(\Delta^{-1}[p]h_k\) are weak \(A_7\)-Jacobi forms. The dictionary in
Appendix~\ref{app:E7-coefficient-dictionary} gives the lower bounds in
\eqref{eq:E7-leading-orders}; restriction to \(A_5\oplus A_1\),
together with
\[
 P_{14}=Q_7^2,\qquad P_{22}=2Q_7Q_{15},\qquad P_{30}=Q_{15}^2,
\]
gives the lower bounds in \eqref{eq:E6-leading-orders}.
\par Since \(\epsilon=\Delta p\) and \(\Delta\) is independent of \(p\),
the \(p\)- and \(\epsilon\)-orders agree. The Weyl-vector formula for
the Borcherds products in \eqref{eq:E7-E6-modular-Jacobians} gives
\(\operatorname{ord}_p(J_7)=18\) and
\(\operatorname{ord}_p(J_6)=12\), equal to the sums of the respective
positive lower bounds; see
\cite[Theorems~10.4--10.5]{BorcherdsGrassmannians}. If any individual
bound were strict, the corresponding modular Jacobian would have
higher order. Hence every bound is attained. After \(p^{18}\) and
\(p^{12}\) are removed, the first coefficients of \(J_7\) and \(J_6\)
are nonzero elliptic modular factors times the augmented Jacobians of
the positive-index leading Jacobi forms. Their algebraic independence
over \(\mathbb C[E_4,E_6]\) follows from
\cite[Proposition~2.2]{WangWeylInvariantJacobi}; their Wirthm\"uller
bidegrees and \cite[Proposition~2.3]{WangWeylInvariantJacobi} give
generation. Substitution of the expansions and division by
\(\epsilon^6\) give \eqref{eq:E7-E-string-limit} and
\eqref{eq:E6-E-string-limit}; the monic \(x^7\)-term becomes
\(\epsilon U^7\) and vanishes.
\par Their leading terms at $U=\infty$ give the smooth fiber
$Y^2=X^3-3E_4X+2E_6$. The discriminant calculations
\eqref{eq:E7-local-discriminant} and \eqref{eq:E6-local-discriminant}
apply unchanged at $U=0$, giving $\mathrm I_2$ and $\mathrm I_3$. The
remaining discriminant factors are generically squarefree of degrees
ten and nine. The trivial lattices have root parts $A_1$ and $A_2$,
and the generic Mordell--Weil groups have no torsion. The height
pairing identifies their free lattices with the duals of the
orthogonal complements $E_7,E_6\subset E_8$; see
\cite{ShiodaMordellWeil}.  Finally, in the unrescaled equations the
limit is $\eta^2=T^3-3E_4x^4T+x^6(x+2E_6)$. The substitution
$T=x^2X_\infty$, $\eta=x^3Y_\infty$ gives
\eqref{eq:exceptional-unrescaled-curve}. 
Blowing up the \((x,\epsilon)\)-base at \((0,0)\) gives the rescaled
and unrescaled charts. Their intersection is the smooth anticanonical
fiber \(E_\tau\). The argument of
Proposition~\ref{prop:E-string-stable-pair} gives
\(d\)-semistability and Kulikov Type~II.
\end{proof}
The exceptional coefficients can also be recovered directly from the
coefficients in the \(U\oplus A_7(-1)\)-polarized Weierstrass equation
\eqref{eq:A7-family-2}. For each polynomial \(P_k\) in
Appendix~\ref{app:E7-coefficient-dictionary}, replace \(f_4,f_6\) by
their constant terms \(-3E_4,2E_6\), respectively, and substitute
\(g_j=e\mathfrak g_j\) and \(h_j=e\mathfrak h_j\), where
\[
 \mathfrak g_k=\Delta^{-1}[p]g_k,\qquad
 \mathfrak h_k=\Delta^{-1}[p]h_k.
\]
Here \(e\) is a bookkeeping variable. The coefficient of \(e^{m_k}\)
is \(\pi_k\). For example,
\begin{equation}
\label{eq:E7-leading-from-A7}
\begin{aligned}
 \pi_{10}&=-3E_4\mathfrak g_6-8E_6\mathfrak g_4+\mathfrak g_{10},\\
 \pi_{12}&=9E_4^2\mathfrak g_4+E_4\mathfrak g_8
 +4E_6\mathfrak g_6+\mathfrak g_{12},\\
 \pi_{14}&=\mathfrak h_7^2-4\mathfrak h_5\mathfrak h_9.
\end{aligned}
\end{equation}
The $E_6$
coefficients are given by the restriction identities
\begin{equation}
\label{eq:exceptional-limit-rank-one}
\pi_{14}|_{E_6}=\kappa_7^2,\qquad
\pi_{22}|_{E_6}=2\kappa_7\kappa_{15},\qquad
\pi_{30}|_{E_6}=\kappa_{15}^2.
\end{equation}
\par In the standard realization $E_8\subset \mathbb R^8$, we set
\begin{equation}
\label{eq:exceptional-Wilson-subspaces}
\mathbf z^{(7)}=(z_1,z_2,z_3,z_4,z_5,z_6,z,-z),\qquad \mathbf
z^{(6)}=(z_1,z_2,z_3,z_4,z_5,z,z,-z),
\end{equation}
which are the $E_7$ and $E_6$ root subspaces used in
\cite[Equations~(2.12), (2.31)]{SakaiEnJacobi}. 
\begin{corollary}
\label{cor:exceptional-restrictions-commute-with-limit}
Restriction of the Hashimoto--Ueda family
\eqref{eq:UE8-general} to the \(E_7\)- and \(E_6\)-Wilson-line
subspaces \eqref{eq:exceptional-Wilson-subspaces} commutes with the
rescaled cusp limit: the restrictions give the Weierstrass
models \eqref{eq:E7-universal-model} and
\eqref{eq:E6-universal-model}, respectively, whose cusp limits are the rational elliptic surfaces \eqref{eq:E7-E-string-limit} and
\eqref{eq:E6-E-string-limit}; equivalently, these rational surfaces are obtained
by restricting the Wilson-line parameters in the E-string
surface \eqref{eq:E-string-SW-curve} for \(E_8\).
\end{corollary}
\begin{proof}
For \(r=7,6\), restriction to the corresponding Wilson-line subspace
commutes with the extraction of the Fourier--Jacobi coefficient, i.e.,
\begin{equation}
\label{eq:restriction-FJ-commutes}
 [p^m]\bigl(
 F(\tau,\rho,\mathbf z)|_{\mathbf z=\mathbf z^{(r)}}\bigr)
 =
 \bigl([p^m]F(\tau,\rho,\mathbf z)\bigr)
 |_{\mathbf z=\mathbf z^{(r)}}.
\end{equation}
The two Wilson-line subspaces are the type-IV subdomains for $T_7$ and
$T_6$. On the generic locus, global Torelli and uniqueness of the
normal forms in Theorems~\ref{thm:E7-universal-model} and
\ref{thm:E6-universal-model} identify the restricted K3 families with
\eqref{eq:E7-universal-model} and \eqref{eq:E6-universal-model}; the
identity extends by holomorphy. Since $\Delta=\Delta(\tau)$ is
unchanged by restriction, Equation~\eqref{eq:restriction-FJ-commutes}
also holds for $\Delta^{-m}[p^m]F$. It gives
\eqref{eq:E7-E-string-limit} and \eqref{eq:E6-E-string-limit} after
the rescaling. 
Using \eqref{eq:E7-E6-Jacobi-normalization} the relations
\eqref{eq:exceptional-limit-rank-one} become Sakai's quadratic
identities \cite[Equation~(3.21)]{SakaiEnJacobi}.

\end{proof}
\begin{remark}
To compare our Weierstrass normalization with the one used by Sakai, we define
$\alpha^{(r)}_{w,m}$ from our leading coefficients by
\begin{equation}
\label{eq:E7-E6-Jacobi-normalization}
\begin{aligned}
 \pi_k&=\tfrac14\lambda^k \alpha^{(7)}_{k-12m_k,m_k},
 &&k=10,12,14,16,18,22,24,30,\\ \kappa_k&=\tfrac14\lambda^k
 \alpha^{(6)}_{k-12n_k,n_k}, &&k=7,10,12,15,16,18,24, \qquad
 \lambda^2=-12.
\end{aligned}
\end{equation}
The coefficients \(\alpha^{(r)}_{w,m}\) are theta-normalized weak Jacobi 
generators for which the resulting Weierstrass equations agree with Sakai's normal forms
\cite[Sections~3.3--3.4]{SakaiEnJacobi}. We first set
$U=\lambda^6u$, $X=\lambda^{14}\mathsf x$, $Y=\lambda^{21}\mathsf
y/2$. For $E_7$, the further substitution $\mathsf x=ux_0$, $\mathsf
y=uy_0$ gives \cite[Equation~(3.10)]{SakaiEnJacobi}.  For $E_6$, the
substitution
\begin{equation}
\label{eq:E6-Sakai-coordinate-change}
\mathsf x=ux_0,\qquad \mathsf y=u^2y_0-\frac{u}{2}
\left(\alpha^{(6)}_{-5,1}x_0+
\alpha^{(6)}_{-9,2}\right)
\end{equation}
followed by division by $u^3$ gives
\cite[Equation~(3.19)]{SakaiEnJacobi}. 
\end{remark}
\begin{remark}
\label{rem:breaking}
The Wilson-line subspaces in \eqref{eq:exceptional-Wilson-subspaces}
have orthogonal root systems $A_1$ and $A_2$.  Thus the affected
heterotic $E_8$ factor has generic unbroken algebra
\[
\begin{aligned}
\mathfrak e_8&\longrightarrow \mathfrak{su}_2\oplus\mathfrak
u(1)^{\oplus7} &&(E_7\text{ subspace}),\\ \mathfrak
e_8&\longrightarrow \mathfrak{su}_3\oplus\mathfrak u(1)^{\oplus6}
&&(E_6\text{ subspace}).
\end{aligned}
\]
The other $E_8$ factor, represented by the $\mathrm{II}^*$ fiber on
the complementary component, remains unbroken.  The nonabelian parts
are therefore $\mathfrak e_8\oplus\mathfrak{su}_2$ and $\mathfrak
e_8\oplus\mathfrak{su}_3$, in agreement with the
$\mathrm{II}^*+\mathrm I_2$ and $\mathrm{II}^*+\mathrm I_3$ fiber
configurations.
\end{remark}
\section{Restrictions}
\label{sec:restrictions}
We construct two arithmetic restrictions of the coefficient models and
then describe their F-theory interpretations.
\subsection{Witt restriction and the Inose family}
\label{sec:Witt-Inose}
For $L_E=U^{\oplus2}\oplus E_8(-1)$, we write a Fourier--Jacobi
expansion at the standard one-dimensional cusp in tube-domain
coordinates as
\begin{equation*}
F(\tau_1,\mathbf z,\tau_2) =\sum_{m\geq0}\phi_m(\tau_1,\mathbf
z)q_2^m, \qquad q_2=e^{2\pi i\tau_2}.
\end{equation*}
The Witt restriction is evaluation on the subdomain $\mathbf z=0$,
i.e.,
\begin{equation}
\label{eq:Witt-definition}
\operatorname W(F)(\tau_1,\tau_2) =F(\tau_1,0,\tau_2)
=\sum_{m\geq0}\phi_m(\tau_1,0)q_2^m.
\end{equation}
The orthogonal group interchanges the two copies of \(U\), so the
image is symmetric in \(\tau_1,\tau_2\). Let \(E_4,E_6\) be the
normalized elliptic Eisenstein series and put
\begin{equation}
\label{eq:Witt-e4-e6-d}
e_4=E_4(\tau_1)E_4(\tau_2),\qquad e_6=E_6(\tau_1)E_6(\tau_2),\qquad
d=\Delta(\tau_1)\Delta(\tau_2),
\end{equation}
where \(\Delta=(E_4^3-E_6^2)/1728\).  We also introduce
\begin{equation}
\label{eq:Witt-abD}
a=-3e_4,\qquad b=\frac{2}{7}e_6,\qquad D=1728^2d.
\end{equation}
These forms have weights \(4,6,12\), respectively. The geometric
interpretation of this diagonal restriction uses the Shioda--Inose
construction and its formulation in terms of \(U\oplus
E_8(-1)^{\oplus2}\)-polarized K3 surfaces
\cite{ShiodaInose,Inose,MorrisonLargePicard,ClingherDoran}.
\begin{proposition}
\label{prop:Witt-target-ring}
The graded ring of symmetric parallel-weight bimodular forms is
\begin{equation}
\label{eq:Witt-target-ring}
\bigoplus_{k\geq0} \bigl(M_k(\operatorname{SL}_2(\mathbb Z))\otimes
M_k(\operatorname{SL}_2(\mathbb Z))\bigr)^{\mathfrak S_2} =\mathbb
C[e_4,e_6,d].
\end{equation}
In particular, \(\operatorname W\) maps the Hashimoto--Ueda ring to
the polynomial ring $\mathbb C[e_4,e_6,d]$.
\end{proposition}
\begin{proof}
Write \(x_i=E_4(\tau_i)\), \(y_i=E_6(\tau_i)\), and \(D=1728^2d\). The
two mixed monomials \(A=x_1^3y_2^2\) and \(B=x_2^3y_1^2\) satisfy
\begin{equation*}
A+B=e_4^3+e_6^2-D,\qquad AB=e_4^3e_6^2.
\end{equation*}
Every symmetric pair of monomials of equal weight is a monomial in
\(e_4,e_6\) times \(A^r+B^r\) for some \(r\geq0\). The power sums
\(A^r+B^r\) are generated recursively by \(A+B\) and \(AB\). This
proves generation. The diagonal subring of \(\mathbb
C[x_1,y_1,x_2,y_2]\) has Krull dimension three, and taking the finite
\(\mathfrak S_2\)-invariants preserves dimension. Since it is a domain
generated by the three elements \(e_4,e_6,d\), they are algebraically
independent.
\end{proof}
\par Dieckmann, Krieg, and Woitalla proved that the normalized orthogonal
Eisenstein series \(\mathcal E_k\), for
\begin{equation*}
k=4,10,12,16,18,22,24,28,30,36,42,
\end{equation*}
form another system of polynomial generators for \(M_*(\mathrm
O^+(L_E))\) \cite[Theorem~4.3]{DieckmannKriegWoitalla}. Thus the
Hashimoto--Ueda coefficients and these Eisenstein series are two
coordinate systems on the same affine cone. Substitution in that
change of generators and then in \eqref{eq:Witt-definition} gives the
following result:
\begin{proposition}
\label{prop:Witt-Inose}
The Witt restrictions of the Hashimoto--Ueda generators are
\begin{equation}
\label{eq:Witt-HU-generators}
\begin{alignedat}{2}
\operatorname W(C_4)&=a,&\qquad \operatorname
W(C_{10})&=-4ab,\\ \operatorname W(C_{12})&=-21b^2+D,& \operatorname
W(C_{16})&=6ab^2,\\ \operatorname W(C_{18})&=70b^3-5bD,& \operatorname
W(C_{22})&=-4ab^3,\\ \operatorname W(C_{24})&=-105b^4+10b^2D,&
\operatorname W(C_{28})&=ab^4,\\ \operatorname
W(C_{30})&=84b^5-10b^3D,& \operatorname
W(C_{36})&=-35b^6+5b^4D,\\ \operatorname W(C_{42})&=6b^7-b^5D.&&
\end{alignedat}
\end{equation}
The Witt restriction of the family \eqref{eq:UE8-general}, after
change of variables, is
\begin{equation}
\label{eq:modular-Inose-family}
y^2=x^3-3e_4T^4x +T^5\bigl(T^2+2e_6T+1728^2d\bigr).
\end{equation}
For \(d\ne0\), Equation~\eqref{eq:modular-Inose-family} is the Inose
family associated with the unordered pair of elliptic curves \(\lbrace
E_{\tau_1},E_{\tau_2} \rbrace \).  The divisor \(d=0\) in the
Baily--Borel compactification of the Witt surface maps exactly onto
the Hashimoto--Ueda non-K3 locus. In particular, on \(d=0\) the
coefficient point is
\begin{equation}
\label{eq:HU-non-K3-parametrization}
[\mathbf C]={}[a:-4ab:-21b^2:6ab^2:70b^3:-4ab^3:
  -105b^4:ab^4:84b^5:-35b^6:6b^7].
\end{equation}
\end{proposition}
\begin{proof}
The identities \eqref{eq:Witt-HU-generators} follow by restricting the
Eisenstein polynomials and using $E_4^3-E_6^2=1728\Delta$.
Equivalently, the resulting coefficient substitution
in~\eqref{eq:UE8-general} is
\begin{equation}
\label{eq:Witt-HU-factorized}
y^2=x^3+a(t-b)^4x +(t-b)^5\bigl((t-b)^2+7b(t-b)+D\bigr).
\end{equation}
After the translation \(T=t-b\), this
becomes~\eqref{eq:modular-Inose-family}. Expanding
\eqref{eq:Witt-HU-factorized} gives exactly the eleven
coefficients. For \(d\ne0\), the orders at \(T=0\) are
\((\geq4,5,10)\) for \((f,g,4f^3+27g^2)\), and the same orders occur
at \(T=\infty\). Hence the two distinguished fibers are of type
\(\mathrm{II}^*\). Moreover,
\begin{equation}
\label{eq:j-pair-from-Witt}
j(\tau_1)j(\tau_2)=\frac{e_4^3}{d}, \qquad
\bigl(j(\tau_1)-1728\bigr)\bigl(j(\tau_2)-1728\bigr)=\frac{e_6^2}{d}.
\end{equation}
Thus the two classical \(j\)-invariants, up to permutation, are
recovered from \((e_4,e_6,d)\). 
\par On the compactified product of modular curves, $d=0$ is the image of
the union of the two boundary divisors $\Delta(\tau_1)=0$ and
$\Delta(\tau_2)=0$. Before quotienting by $\mathfrak S_2$, these are
two components interchanged by permutation; after quotienting, their
image is the single one-dimensional Baily--Borel boundary component
$\operatorname{Proj}\mathbb C[e_4,e_6]$.  Setting \(D=0\) gives
\eqref{eq:HU-non-K3-parametrization} and in the Weierstrass model
\begin{equation}
\label{eq:HU-non-K3-factorization}
f(t)=a(t-b)^4,\qquad g(t)=(t-b)^6(t+6b).
\end{equation}
At \(t=b\) one has \(\operatorname{ord}(f,g)\geq(4,6)\), so the
Weierstrass model is non-minimal and is not a K3 surface. Hashimoto
and Ueda prove that \eqref{eq:HU-non-K3-parametrization} is the entire
non-K3 locus \cite[Corollary~3.2]{HashimotoUeda}; hence the image is
exact.
\end{proof}
\begin{remark}
The Witt map \eqref{eq:Witt-definition} restricts to a two-dimensional
totally geodesic modular surface. On the ordered product, the preimage
of $d=0$ has two boundary components, which meet at the
zero-dimensional cusp; see \cite{Scattone,DolgachevMirror}.
\end{remark}
\subsection{Five-dimensional Gaussian ball subfamilies}
\label{sec:Gaussian-ball-subfamilies}
Let \(L\) be one of the rank-ten lattices in \eqref{eq:AN-chain}, and
set \(T_L=L^\perp\subset\Lambda_{\mathrm{K3}}\). We fix an isometry
\(\rho_T\in\mathrm O(T_L)\) with \(\rho_T^2=-1\), and consider the
locus whose periods lie in its \(+i\)-eigenspace. This eigenspace has
dimension six and Hermitian signature \((1,5)\); projectivizing its
positive cone gives a complex ball \(\mathbb B^5\).
We restrict to the geometric locus on
which \(\rho_T\) extends to an order-four Hodge isometry of the K3
lattice preserving an ample chamber. By the global Torelli theorem,
the extension is induced by a purely non-symplectic automorphism
\(\rho\) of order four. This restricts the ten-dimensional type-IV
moduli space to a five-dimensional unitary modular subvariety, without
increasing the Picard rank of a very general member
\cite{KondoEightPoints,ClingherMalmendierWilliams2025,
  ArtebaniSartiOrderFour}.
\par We first restrict the five Jacobian fibrations associated with the
polarizing lattices in~\eqref{eq:AN-chain} to this complex ball and
determine the corresponding equations and automorphisms.
\begin{enumerate}[label=\textup{(\roman*)}]
\item For \(U\oplus E_8(-2)\), the restriction yields a quadratic base
  change of a rational elliptic surface, namely
\begin{equation}
\label{eq:UE8minus2-restricted}
\begin{split}
& y^2=x^3+ (f_0+f_2\xi^4+f_4\xi^8)x
  +g_0+g_2\xi^4+g_4\xi^8+g_6\xi^{12},\\ & \rho(\xi,x,y)=(i\xi,x,y).
\end{split}
\end{equation}
\item For \(U\oplus N\), one has
\begin{equation}
\label{eq:UN-restricted}
\begin{split}
& Y^2=X^3+b(t)X,\qquad \deg b=8,\\ & \rho(t,X,Y)=(t,-X,iY).
\end{split}
\end{equation}
If \(b\) has eight simple roots, there are eight singular fibers of
type \(\mathrm{III}\) whose unordered positions give five moduli.
\item For \(U\oplus D_4(-1)^{\oplus2}\), one has
\begin{equation}
\label{eq:U2D4-restricted}
\begin{split}
&y^2=x^3 +u^2(a_0+a_2u^2+a_4u^4)x
  +u^3(b_0+b_2u^2+b_4u^4+b_6u^6),\\ &\rho(u,x,y)=(-u,-x,iy).
\end{split}
\end{equation}
\item For \(U\oplus D_8(-1)\), the standard fibration restricts to
\begin{equation}
\label{eq:UD8-restricted}
\begin{split}
& y^2= x^3+(s^3+sF_4)x^2+(s^2G_8+G_{12})x + H_8
  s^5+H_{12}s^3+H_{16}s,\\ & \rho(s,x,y)=(-s,-x,iy).
\end{split}
\end{equation}
\item For \(U\oplus E_8(-1)\), the Hashimoto--Ueda model restricts to
\begin{equation}
\label{eq:UE8-restricted}
\begin{split}
& y^2=x^3+(C_4t^4+C_{16}t^2+C_{28})x
  +t^7+C_{12}t^5+C_{24}t^3+C_{36}t,\\ & \rho(t,x,y)=(-t,-x,iy).
\end{split}
\end{equation}
\end{enumerate}
In (i), \(\rho^2\) is the deck involution of the quadratic base
change; in (ii)--(v), it is the fiberwise elliptic involution.
\par Let \(\operatorname{Conf}_8(\mathbb P^1)\subset(\mathbb P^1)^8\)
be the locus of ordered eight-tuples of pairwise distinct points, and
set
\begin{equation*} \mathcal M_{\mathrm K}^{\mathrm{ord}}:=\operatorname{Conf}_8(\mathbb P^1)/\operatorname{PGL}_2,\qquad \overline{\mathcal M}_{\mathrm K}^{\mathrm{ord}}:=(\mathbb P^1)^8\mathbin{/\mkern-6mu/}_{\mathcal O(1,\ldots,1)}\operatorname{SL}_2. \end{equation*}
Thus \(\mathcal M_{\mathrm K}^{\mathrm{ord}}=M_{0,8}\), while
\(\overline{\mathcal M}_{\mathrm K}^{\mathrm{ord}}\) is the GIT
compactification by semistable ordered configurations. Kond\=o
identifies the latter with the Satake--Baily--Borel compactification
of \(\mathbb B^5/\Gamma(1-i)\), where \(\Gamma\) is the ambient
arithmetic unitary group and \(\Gamma(1-i)\) is its principal
congruence subgroup of level \(1-i\). The identification is
\(S_8\)-equivariant, and \(\Gamma/\Gamma(1-i)\cong S_8\)
\cite[Theorems~3.7 and~4.11]{KondoEightPoints}. We write
\begin{equation*} \mathcal M_{\mathrm K}:=\mathcal M_{\mathrm K}^{\mathrm{ord}}/S_8 \end{equation*}
for the unlabelled moduli space. This convention distinguishes
Kond\=o's ordered ball quotient from the moduli space of the
underlying unlabelled K3 surfaces; see
\cite{MatsumotoYoshida,MatsumotoTerasoma}.
\par Write a point of \(\mathbb P^1\) as a nonzero column vector
\(v_i\in\mathbb C^2\), and put \([ij]:=\det(v_i,v_j)\). A tableau is a
perfect matching of \(\{1,\ldots,8\}\), written as four rows
\((\tau_{r1},\tau_{r2})\), and its bracket monomial is
\begin{equation}
\label{eq:cross-ratio-monomial}
\mu_\tau:=\prod_{r=1}^4[\tau_{r1}\tau_{r2}].
\end{equation}
Reversing the orientation of one row changes the sign of \(\mu_\tau\);
the conventional orientations of standard tableaux fix these
signs. There are \(8!/(2^4 4!)=105\) matching classes. Straightening
by the Pl\"ucker relations gives fourteen standard tableaux
\(\tau_1,\ldots,\tau_{14}\), which form a basis of the degree-one
invariants and define the classical map \(\overline{\mathcal
  M}_{\mathrm K}^{\mathrm{ord}}\to\mathbb P^{13}\); see
\cite{DolgachevOrtland,HowardMillsonSnowdenVakil}. If \(F_V\) is
Kond\=o's weight-four orthogonal modular form attached to the
corresponding maximal totally singular subspace \(V\), then
\(F_V|_{\mathbb B^5}\) has an even divisor. Since \(\mathbb B^5\) is
simply connected, Kond\=o chooses a square root \(G_V\) of ball weight
two. For the subspaces \(V_1,\ldots,V_{14}\) corresponding to the
chosen standard tableaux, his comparison theorem states
\begin{equation}
\label{eq:Kondo-projective-equality}
[G_{V_1}:\cdots:G_{V_{14}}]=[\mu_{\tau_1}:\cdots:\mu_{\tau_{14}}].
\end{equation}
This is an equality of projective maps, not a canonical affine
equality between individual modular forms and bracket polynomials
\cite[Sections~6--7, especially Theorem~7.5]{KondoEightPoints}.
\par The finite data required by the neighbor construction consist of an
ordered decomposition \(S_0\sqcup S_\infty=\{1,\ldots,8\}\), with
\(|S_0|=|S_\infty|=4\), and a perfect matching of each four-element
set. Let \(\mathfrak d_0\) denote the marking
\begin{equation} 
\label{eqn:marking}
S_0=\{1,2,3,4\},\qquad S_\infty=\{5,6,7,8\},\qquad
\pi_0=(12)(34),\qquad \pi_\infty=(56)(78),
\end{equation}
and let \(H_{\mathfrak d_0}\subset S_8\) be its stabilizer. We define
\begin{equation}
\label{eq:Kondo-marked-cover}
\widetilde{\mathcal M}_{\mathrm K}:=\mathcal M_{\mathrm
  K}^{\mathrm{ord}}/H_{\mathfrak d_0}\longrightarrow\mathcal
M_{\mathrm K}=\mathcal M_{\mathrm K}^{\mathrm{ord}}/S_8.
\end{equation}
It is finite of generic degree \(\binom84\cdot3^2=630\), and it is
\'etale over the locus of configurations with trivial projective
automorphism group. Let \(\varepsilon_{\tau_0}:H_{\mathfrak
  d_0}\to\{\pm1\}\) be the character by which the stabilizer acts on
\(\mu_{\tau_0}\), and put \(H_{\mathfrak
  d_0}^+:=\ker(\varepsilon_{\tau_0})\). The additional character cover
is
\begin{equation*} \widetilde{\mathcal M}_{\mathrm K}^{H_4}:=\mathcal M_{\mathrm K}^{\mathrm{ord}}/H_{\mathfrak d_0}^+\longrightarrow\widetilde{\mathcal M}_{\mathrm K}. \end{equation*}
On this twofold cover the sign of \(H_4\) is defined. An affine
coordinate on \(\mathbb P^1\), monic equations for the two quartics,
and a Weierstrass scale are coordinate gauges.
\par On the affine chart on which all branch points are finite, write
\begin{equation}
\label{eq:Kondo-A-B}
b_1(t)=\prod_{j=1}^4(t-\lambda_j),\qquad
b_2(t)=\prod_{j=5}^8(t-\lambda_j).
\end{equation}
We choose the affine lifts \(v_i=(1,\lambda_i)^t\), so that
\([ij]=\lambda_j-\lambda_i\).  The restricted double-cover model is
\begin{equation}
\label{eq:VGSrestricted_dc}
Z^2=b_2(t)\xi^4+b_1(t).
\end{equation}
The substitution $u=\xi^2$, $v=\xi Z$ has degree two and yields $v^2=u(b_2(t)u^2+b_1(t))$. Setting $X=b_2(t)u$ and $Y=b_2(t)v$
gives $Y^2=X^3+b_1(t)b_2(t)X$. This Weierstrass equation has
$j=1728$ and discriminant $-64(b_1b_2)^3$; see \cite[Section~2]{KondoEightPoints}.
\begin{proposition}
\label{prop:Kondo-degree-dictionary}
Fix the marking \(\mathfrak d_0\)
in~\eqref{eqn:marking}.
The two-neighbor and relative-Jacobian construction defines a
polynomial map
\begin{equation*} (\lambda_1,\ldots,\lambda_8)\longmapsto(F_4,G_8,H_8,G_{12},H_{12},H_{16}). \end{equation*}
With \(\deg(\lambda_i)=1\), their ordinary total degrees are
\begin{equation}
\label{eq:Kondo-D8-polynomial-degrees}
\begin{array}{c|cccccc}\text{coefficient}&F_4&G_8&H_8&G_{12}&H_{12}&H_{16}\\ \hline \deg_\lambda&4&8&8&12&12&16\end{array}
\end{equation}
In particular, for \(\tau_0=(12)(34)(56)(78)\) one has
\begin{equation}
\label{eq:Kondo-H8-factorization}
H_8=-\frac14[12]^2[34]^2[56]^2[78]^2=-\frac14\mu_{\tau_0}^2,
\end{equation}
and one can choose \(H_4=\mu_{\tau_0}/2\) on \(\widetilde{\mathcal
  M}_{\mathrm K}^{H_4}\).
\end{proposition}
\begin{proof}
We set $Q(t,u)=b_1(t)+u^2b_2(t)$, and write
$\mathcal I(Q)=I_0+I_2u^2+I_4u^4$ and
$\mathcal J(Q)=J_0+J_2u^2+J_4u^4+J_6u^6$. Set
$a_{2j}=-I_{2j}/3$ and $b_{2j}=-J_{2j}/27$. Since
$\mathcal I(uQ)=u^2\mathcal I(Q)$ and
$\mathcal J(uQ)=u^3\mathcal J(Q)$, the relative Jacobian is
\eqref{eq:U2D4-restricted}.
\par For a four-element set \(S\), let \(e_2(S):=\sum_{i<j,\,i,j\in
  S}\lambda_i\lambda_j\), and let \(\operatorname{Pair}(S)\) be its
set of three perfect matchings. If
\(\pi=\{\{i,j\},\{k,l\}\}\in\operatorname{Pair}(S)\), define
\begin{equation*} \gamma_\pi(S):=\frac13e_2(S)-(\lambda_i\lambda_j+\lambda_k\lambda_l),\qquad \alpha_\pi:=\gamma_\pi(\{1,2,3,4\}),\qquad \beta_\pi:=\gamma_\pi(\{5,6,7,8\}). \end{equation*}
A computation gives
\begin{equation*} z^3+a_0z+b_0=\prod_{\pi\in\operatorname{Pair}(\{1,2,3,4\})}\!\!\!(z-\alpha_\pi),\qquad z^3+a_4z+b_6=\prod_{\pi\in\operatorname{Pair}(\{5,6,7,8\})}\!\!\!(z-\beta_\pi). \end{equation*}
For a
generic configuration there are therefore \(3\cdot3=9\) possible pairs
\((\alpha,\beta)\). The reference marking $\mathfrak d_0$ gives
\begin{equation}
\label{eq:Kondo-alpha-beta}
\alpha=\frac13e_2(\{1,2,3,4\})-(\lambda_1\lambda_2+\lambda_3\lambda_4),\qquad
\beta=\frac13e_2(\{5,6,7,8\})-(\lambda_5\lambda_6+\lambda_7\lambda_8).
\end{equation}
The construction in
Appendix~\ref{App:2NS}, with the odd-index coefficients set to zero,
gives
\begin{equation}
\label{eq:Kondo-neighbor-quartic}
w^2=ps+(qs^2+r)u+(s^3+ms)u^2+(ns^2+o)u^3+hsu^4,
\end{equation}
where are given by
\begin{equation}
\label{eq:Kondo-auxiliary-coefficients}
\begin{aligned} p&=3\alpha^2+a_0,& q&=3\alpha,& r&=\beta p+a_2\alpha+b_2,\\ m&=6\alpha\beta+a_2,& n&=3\beta,& o&=3\alpha\beta^2+a_2\beta+a_4\alpha+b_4,\\  h&=3\beta^2+a_4. \end{aligned}
\end{equation}
Every coefficient of \(b_1\) and \(b_2\) is an elementary symmetric
polynomial in the corresponding four branch points. Hence
\(I_{2j},J_{2j},\alpha,\beta\), and the seven coefficients in
\eqref{eq:Kondo-auxiliary-coefficients} are polynomials in the
\(\lambda_i\). Assigning degree two to \(s\), every coefficient of the
quartic in \eqref{eq:Kondo-neighbor-quartic} is homogeneous of degree
six. Its binary-quartic invariants therefore have degrees twelve and
eighteen. Comparing the resulting relative Jacobian with
\eqref{eq:UD8-restricted} gives polynomial coefficients of the degrees
listed in \eqref{eq:Kondo-D8-polynomial-degrees}. The same comparison
gives
\[
 H_8=-\frac14(3\alpha^2+4a_0)(3\beta^2+4a_4).
\]
The endpoint bracket identities
\[
 3\alpha^2+4a_0=-[12]^2[34]^2,\qquad
 3\beta^2+4a_4=-[56]^2[78]^2
\]
give~\eqref{eq:Kondo-H8-factorization}.
\end{proof}
\subsection{F-theory interpretations of the restrictions}
\label{subsec:restriction-physics}
The restrictions in Sections~\ref{sec:Witt-Inose}
and~\ref{sec:Gaussian-ball-subfamilies} have different interpretations
in eight-dimensional F-theory/heterotic duality. We distinguish the
heterotic variables $(\tau,\rho,\mathbf z)$ and the type-IIB
axiodilaton $\tau_{\mathrm{IIB}}$, which is the modulus of the
elliptic fiber of the F-theory background.
\subsubsection{The Witt surface and vanishing Wilson lines}
For the Hashimoto--Ueda family, one heterotic $E_8$ factor is 
unbroken. The Witt restriction $\mathbf z=0$ turns off the Wilson
lines in the other factor and restores the nonabelian algebra
$\mathfrak e_8\oplus\mathfrak e_8$. Geometrically, this gives the two
$\mathrm{II}^*$ fibers in Proposition~\ref{prop:Witt-Inose}. The two
remaining complex moduli are interpreted as the torus complex
structure $\tau$ and complexified K\"ahler modulus $\rho$
\cite{VafaFTheory,MorrisonVafaII,ClingherMorganFTheory}.
Geometrically, the Shioda--Inose construction relates the
corresponding family of K3 surfaces to the unordered pair
$(E_\tau,E_\rho)$; $E_\rho$ is an auxiliary elliptic curve encoding
the K\"ahler modulus. The interchange $\tau\leftrightarrow\rho$ is
part of the torus duality group, and the two modular actions explain
the symmetric diagonal ring in \eqref{eq:Witt-target-ring}.
\par The Witt surface lies in the interior when
$\Delta(\tau)\Delta(\rho)\ne0$. The additional limit
$\operatorname{Im}\rho\to\infty$ reaches its Type~II boundary.
Evaluation at $\mathbf z=0$ commutes with the E-string limit.  In
Sakai's gauge, the specialized Jacobi coefficients are
\begin{equation}
\label{eq:zero-Wilson-E-string-coefficients}
a_2=a_3=a_4=0,\qquad b_1=-4,\qquad b_2=\cdots=b_6=0
\end{equation}
\cite{EguchiSakaiEString,SakaiEnJacobi,SakaiFTheoryK3}.  The rescaled
component becomes
\begin{equation}
\label{eq:zero-Wilson-E-string-curve}
y^2=4x^3-\frac{E_4}{12}U^4x-\frac{E_6}{216}U^6+4U^5.
\end{equation}
It has a $\mathrm{II}^*$ fiber at $U=0$ and two nodal fibers for
general $\tau$. Inverting $U$ and rescaling the fiber coordinates
identifies it with the extremal unrescaled curve
\eqref{eq:Sakai-unrescaled-RES}. The zero-Wilson-line limit therefore
has two extremal $\mathrm{II}^*+2\mathrm I_1$ components, as expected
from the two unbroken $E_8$ factors. At finite $\rho$, the
Shioda--Inose K3 generally has a varying elliptic-fiber modulus. In
particular, vanishing heterotic Wilson lines does not imply constant
$\tau_{\mathrm{IIB}}$.
\par The other extremal component in
\eqref{eq:CE-unrescaled-extremal-curve} admits a complementary
Wilson-line interpretation. In the standard Euclidean realization of
$E_8$, the class $\mathbf z_0=e_1$ has order two modulo the $E_8$
lattice. The roots $\pm e_i\pm e_j$ have integral pairing with
$\mathbf z_0$, whereas the half-integral spinor roots have
half-integral pairing. Hence this discrete Wilson line preserves the
$D_8$ root system. Replacing \(e_1\) by \(\tau e_1\) or \((1+\tau)e_1\) 
gives the other two nonzero two-torsion classes in the elliptic subgroup
generated by \(e_1\). This root calculation explains the $\mathrm I_4^*$
enhancement and the marked two-torsion point in \eqref{eq:CE-gluing-data}; 
the surface was given in 
Proposition~\ref{prop:CE-standard-E-string-limit}.
\subsubsection{The Gaussian ball and constant-coupling models}
The ball condition of Section~\ref{sec:Gaussian-ball-subfamilies} selects periods in an eigenspace of an order-four
lattice isometry. It is a symmetry condition on the complex-structure
moduli, rather than a generic Picard-rank enhancement. Its most direct
F-theory realization is the Kond\=o elliptic fibration~\eqref{eq:UN-restricted}, i.e., 
\begin{equation}
\label{eq:ball-F-theory-constant-coupling}
Y^2=X^3+b(t)X,\qquad b(t)=\prod_{i=1}^8(t-\lambda_i),\qquad
j=1728,\qquad \Delta=-64b(t)^3.
\end{equation}
Thus $\tau_{\mathrm{IIB}}=i$ up to $\operatorname{SL}_2(\mathbb Z)$,
and the eight simple roots of $b$ give eight fibers of type
$\mathrm{III}$. Each contributes three units of discriminant
multiplicity and an $A_1$ root, for a total of twenty-four units and
nonabelian algebra $\mathfrak{su}_2^{\oplus8}$. This is the
five-dimensional constant-coupling branch studied in
\cite{DasguptaMukhiConstantCoupling}; its parameters are the eight
fiber positions modulo $\operatorname{PGL}_2$. The fiberwise
automorphism $(X,Y)\mapsto(-X,iY)$ realizes the Gaussian symmetry.
Its finite-order monodromy fixes $i$, so this branch is not a
weak-coupling limit of the type-IIB axiodilaton.  The section
$(X,Y)=(0,0)$ has order two. For a very general member it meets the
nonidentity component of each of the eight $\mathrm{III}$ fibers,
giving the diagonal central quotient $\operatorname{SU}(2)^8/\mathbb
Z_2$ for the semisimple gauge factor. This is the relation between
Mordell--Weil torsion and the global gauge group described in
\cite{MayrhoferMorrisonTillWeigand}.
For the marking \(\tau_0=(12)(34)(56)(78)\), each factor
\([ij]=\lambda_j-\lambda_i\) in
\eqref{eq:Kondo-H8-factorization} is a local parameter for the
collision of the corresponding pair of type-\(\mathrm{III}\) fibers;
when \([ij]=0\), the pair coalesces to a fiber of type
\(\mathrm I_0^*\). Thus
\(H_4=\frac12[12][34][56][78]\) is the product of the four marked
collision parameters, and \(H_8=-H_4^2\) vanishes to order two along
each associated \(D_4\)-enhancement divisor.
\par The other restricted Weierstrass equations need not have constant
$\tau_{\mathrm{IIB}}$.
In particular, \eqref{eq:UD8-restricted} and
\eqref{eq:UE8-restricted} generally have nonconstant $j$, with
$j(-s)=j(s)$ in \eqref{eq:UD8-restricted} and $j(-t)=j(t)$ in
\eqref{eq:UE8-restricted}.
Their order-four
automorphisms combine a base involution with an isomorphism of the
paired elliptic fibers.  In terms of their periods, they define a
unitary symmetry locus in the orthogonal moduli space. For genus-one
fibrations without a section, its Jacobian has the same fiber modulus and discriminant divisor
but forgets the torsor class; see
\cite[Sections~2 and~4.3--4.4]{BraunMorrisonGenusOne}.  The chosen
ruling and torsor class are part of the discrete data of the F-theory
background.
\par There is also a constraint on the E-string limit of the
Hashimoto--Ueda family.  After restriction to the Gaussian ball, the
equation has even $f(t)$ and odd $g(t)$ (and the cubic-removing
translation vanishes). At the common cusp its sextic coefficient
forces $E_6(\tau)=0$, hence $\tau=i$ modulo
$\operatorname{SL}_2(\mathbb Z)$. The surviving rescaled curve is
\begin{equation}
\label{eq:Gaussian-E-string-limit}
y^2=4x^3-(a_0U^4+a_2U^2+a_4)x -(b_1U^5+b_3U^3+b_5U),\qquad
a_0=E_4(i)/12,
\end{equation}
which has the symmetry $(U,x,y)\mapsto(-U,-x,iy)$ and generally has
varying $j$. Its gluing fiber has complex multiplication by $\mathbb
Z[i]$, while its remaining Jacobi coefficients contain Wilson-line
deformation data. Thus this restriction fixes a duality-symmetric
torus modulus and imposes symmetry on the Wilson lines.
\section*{Conclusion}
The constructions in this article relate the modular period maps for
certain families of lattice-polarized K3 surfaces to the rational
elliptic surfaces at their boundaries. Vinberg’s scroll, the
alternate genus-one pencil, and its relative Jacobian define the
coefficient map to the Hashimoto--Ueda family; after normalization,
this map agrees with Sakai’s reconstruction from the E-string
Seiberg--Witten curve. The exceptional index-two inclusions extend the
construction to the \(E_7\) and \(E_6\) modular families and identify
their distinguished Heegner divisors by Borcherds equations.
\par Our boundary calculations for these families of K3 surfaces use
one \emph{unified} procedure.
For each family, retaining the leading Fourier--Jacobi terms and
applying the Weierstrass rescaling gives the rational components,
gluing curve, and Wilson-line dependence. Taking the relative
Jacobian in type \(D_8\), and restricting to the \(E_7\)- and
\(E_6\)-subfamilies, commute with this limit. 
The broader point of this article is to demonstrate how a
\emph{single} distinguished orthogonal modular form can organize
geometric and physical structures across families of K3 surfaces.
The form \(H_4\) links
the ramification of the modular coefficient space, the alternate
genus-one fibration, and the Borcherds--Enriques invariant. Its
expansions at the two cusp orbits, together with the coefficient map,
connect the global automorphic geometry to the boundary geometries
and E-string curves.
\par A natural next step is to identify the heterotic counterpart of
\(H_4\) and interpret its character, reflective divisor, and cusp
expansions. Our degeneration and Seiberg--Witten results concern the
algebraic family of curves, its Wilson-line dependence, and its marked
smooth fiber; they do not require a choice of Seiberg--Witten
differential. Extracting periods or a prepotential would also require
a physical normalization of that differential.
\appendix
\section{Weierstrass equations associated with root systems}
\label{app:root-system-Weierstrass-models}
Table~\ref{tab:root-system-Weierstrass-models} gives the Weierstrass
models for families whose very general member has transcendental
lattice $T_R=U^{\oplus2}\oplus L_R(-1)$, where $U$ is the hyperbolic
plane and $L_R$ is the root lattice in its standard even
normalization. For simply laced $R$ one has $L_R=R$; in the other
cases the conventions are
\begin{equation*}
L_{B_r}=A_1^{\oplus r},\qquad L_{C_r}=D_r,\qquad L_{G_2}=A_2,\qquad
L_{F_4}=D_4,
\end{equation*}
with $D_3\cong A_3$. The modular group is
\begin{equation*}
\Gamma_R=\big\langle\widetilde{\mathrm O}^{+}(T_R),W(R)\big\rangle,
\end{equation*}
where $\widetilde{\mathrm O}^{+}(T_R)$ is the discriminant kernel and
$W(R)$ acts trivially on $U^{\oplus2}$; see \cite[Theorems~1.2
  and~5.9]{WangWilliamsJacobian} and \cite{HashimotoUeda} for
$E_8$. The column $\rho=18-\operatorname{rank}(R)$ gives the Picard
rank of the very general member.
\par The subscripts indicate modular weights, and distinct letters with the
same subscript denote distinct generators. The generators of
$M_*(\Gamma_R)$ occur in the combinations shown in
Table~\ref{tab:root-system-Weierstrass-models}, with row-dependent
normalizations; the $D_4$ model reflects triality.  In the $A_r$ and
$E_6$ models, the relevant square terms have been moved to the
left-hand side so that the odd-weight generators occur linearly.  The
base coordinate is $t$ throughout. In the $E_6$ row, $c_7=2Q_7$,
$c_{15}=2Q_{15}$, and $c_k=Q_k$ for the even weights; this agrees with
\eqref{eq:E6-universal-model_alt} after renaming coordinates. The
$E_7$ and $E_8$ rows use $c_k=P_k$ and $c_k=C_k$, respectively.
\begingroup \small \setlength{\tabcolsep}{4pt}
\setlength{\extrarowheight}{0pt} \renewcommand{\arraystretch}{1.15}
\newcommand{\RootEquation}[1]{\raisebox{0pt}[\dimexpr\height+2pt\relax][\dimexpr\depth+2pt\relax]{$#1$}}
\begin{longtable}{|c|c|l|}
\caption{Weierstrass models associated with irreducible root systems.}
\label{tab:root-system-Weierstrass-models}\\
\hline $R$ & $\rho$ & Weierstrass equation \\ \hline \endfirsthead
\multicolumn{3}{c}{\tablename~\thetable\ (continued)}\\[3pt] \hline
$R$ & $\rho$ & Weierstrass equation \\ \hline \endhead
\multicolumn{3}{r}{\textit{Continued on next page}}\\ \endfoot
\endlastfoot $A_1$ & $17$ & \RootEquation{y^2 = x^3 + (t^3 + f_4 t +
  f_6) x^2 + (g_{10} t + g_{12}) x} \\[1pt] \hline $A_2$ & $16$ &
\RootEquation{y^2 + g_9 y = x^3 + (t^3 + f_4 t + f_6) x^2 + (g_{10} t
  + g_{12}) x} \\[1pt] \hline $B_2$ & $16$ & \RootEquation{y^2 = x^3 +
  (t^3 + f_4 t + f_6) x^2 + (g_8 t^2 + g_{10} t + g_{12}) x} \\[1pt]
\hline $G_2$ & $16$ & \RootEquation{y^2 = x^3 + (t^3 + f_4 t + f_6)
  x^2 + (g_{10} t + g_{12})x + h_{18}} \\[1pt] \hline $A_3$ & $15$ &
\RootEquation{y^2 + g_9 y = x^3 + (t^3 + f_4 t + f_6) x^2 + (g_8 t^2 +
  g_{10} t + g_{12})x} \\[1pt] \hline $B_3$ & $15$ & \RootEquation{y^2
  = x^3 + (t^3 + f_4 t + f_6) x^2 + (g_6 t^3 + g_8 t^2 + g_{10} t +
  g_{12}) x} \\[1pt] \hline $C_3$ & $15$ & \RootEquation{y^2 = x^3 +
  (t^3 + f_4 t + f_6) x^2 + (g_8 t^2 + g_{10} t + g_{12}) x + h_{18}}
\\[1pt] \hline $A_4$ & $14$ & \RootEquation{\begin{aligned} y^2 + (g_7
    t + g_9) y = x^3 &+ (t^3 + f_4 t + f_6) x^2 \\ &+ (g_8 t^2 +
    g_{10} t + g_{12})x \end{aligned} } \\[1pt] \hline $B_4$ & $14$ &
\RootEquation{\begin{aligned} y^2 = x^3 &+ (t^3 + f_4 t + f_6) x^2
    \\ &+ (g_4 t^4 + g_6 t^3 + g_8 t^2 + g_{10} t + g_{12})
    x \end{aligned}} \\[1pt] \hline $C_4$ & $14$ &
\RootEquation{\begin{aligned} y^2 = x^3 &+ (t^3 + f_4 t + f_6) x^2
    \\ &+ (g_8 t^2 + g_{10} t + g_{12}) x + h_{16} t +
    h_{18} \end{aligned}} \\[1pt] \hline $D_4$ & $14$ &
\RootEquation{\begin{aligned} y^2 = x^3 &+ (t^3 + f_4 t + f_6) x^2
    \\ &+ \left((a_8 - b_8) t^2 + g_{10} t + g_{12}\right) x - a_8 b_8
    t + h_{18} \end{aligned}} \\[1pt] \hline $F_4$ & $14$ &
\RootEquation{\begin{aligned} y^2 = x^3 &+ (f_4 t^4 + g_{10} t^3 +
    \chi_{16} t^2) x \\ &+(t^7 + f_6 t^6 + g_{12} t^5 + h_{18} t^4 +
    \chi_{24} t^3) \end{aligned}} \\[1pt] \hline $A_5$ & $13$ &
\RootEquation{\begin{aligned} y^2 + (g_7 t + g_9) y = x^3 &+ (t^3 +
    f_4 t + f_6) x^2 \\&+ (g_6 t^3 + g_8 t^2 + g_{10} t +
    g_{12})x \end{aligned} } \\[1pt] \hline $C_5$ & $13$ &
\RootEquation{\begin{aligned} y^2 = x^3 + (t^3 + f_4 t + f_6) x^2 &+
    (g_8 t^2 + g_{10} t + g_{12}) x \\ &+ h_{14} t^2 + h_{16} t +
    h_{18} \end{aligned}} \\[1pt] \hline $D_5$ & $13$ &
\RootEquation{\begin{aligned} y^2 = x^3 + (t^3 + f_4 t + f_6) x^2 &+
    (g_8 t^2 + g_{10} t + g_{12}) x \\ &+g_7^2 t^2 + h_{16} t +
    h_{18} \end{aligned}} \\[1pt] \hline $A_6$ & $12$ &
\RootEquation{\begin{aligned} &\quad\quad y^2 + (g_5 t^2 + g_7 t +
    g_9) y \\ &= x^3 + (t^3 + f_4 t + f_6) x^2 + (g_6 t^3 + g_8 t^2 +
    g_{10} t + g_{12})x \end{aligned} } \\[1pt] \hline $C_6$ & $12$ &
\RootEquation{\begin{aligned} y^2 = x^3 &+ (t^3 + f_4 t + f_6) x^2 +
    (g_8 t^2 + g_{10} t + g_{12}) x \\ &+ h_{12} t^3 + h_{14} t^2 +
    h_{16} t + h_{18} \end{aligned}} \\[1pt] \hline $D_6$ & $12$ &
\RootEquation{\begin{aligned} y^2 = x^3 &+ (t^3 + f_4 t + f_6) x^2 +
    (g_8 t^2 + g_{10} t + g_{12}) x \\ &-g_6^2 t^3 + h_{14} t^2 +
    h_{16} t + h_{18} \end{aligned}} \\[1pt] \hline $E_6$ & $12$ &
\RootEquation{\begin{aligned} &\quad\quad y^2 + c_7 xy + (c_{15} t) y
    \\ &= x^3 + (c_4 t^4 + c_{10} t^3 + c_{16} t^2) x \\ &\quad\quad\;
    + t^7 + c_6 t^6 + c_{12} t^{5} + c_{18} t^4 + c_{24}
    t^3 \end{aligned}} \\[1pt] \hline $A_7$ & $11$ &
\RootEquation{\begin{aligned} &\quad\quad y^2 + (g_5 t^2 + g_7 t +
    g_9) y \\ &= x^3 + (t^3 + f_4 t + f_6) x^2 + (g_4 t^4 + g_6 t^3 +
    g_8 t^2 + g_{10} t + g_{12})x \end{aligned} } \\[1pt] \hline $C_7$
& $11$ & \RootEquation{\begin{aligned} y^2 = x^3 &+ (t^3 + f_4 t +
    f_6) x^2 + (g_8 t^2 + g_{10} t + g_{12}) x \\ &+ h_{10} t^4 +
    h_{12} t^3 + h_{14} t^2 + h_{16} t + h_{18} \end{aligned}} \\[1pt]
\hline $D_7$ & $11$ & \RootEquation{\begin{aligned} y^2 = x^3 &+ (t^3
    + f_4 t + f_6) x^2 + (g_8 t^2 + g_{10} t + g_{12}) x \\ &+ g_5^2
    t^4 + h_{12} t^3 + h_{14} t^2 + h_{16} t + h_{18} \end{aligned}}
\\[1pt] \hline $E_7$ & $11$ & \RootEquation{\begin{aligned} y^2 = x^3
    &+ c_{14} x^2 + (c_4 t^4 + c_{10} t^3 + c_{16} t^2 + c_{22} t) x
    \\ &+ t^7 + c_6 t^6 + c_{12} t^5 + c_{18} t^4 + c_{24} t^3 +
    c_{30} t^2 \end{aligned}} \\[1pt] \hline $C_8$ & $10$ &
\RootEquation{\begin{aligned} y^2 = x^3 &+ (t^3 + f_4 t + f_6) x^2 +
    (g_8 t^2 + g_{10} t + g_{12}) x \\ &+ h_8 t^5 + h_{10} t^4 +
    h_{12} t^3 + h_{14} t^2 + h_{16} t + h_{18} \end{aligned}} \\[1pt]
\hline $D_8$ & $10$ & \RootEquation{\begin{aligned} y^2 = x^3 &+ (t^3
    + f_4 t + f_6) x^2 + (g_8 t^2 + g_{10} t + g_{12}) x \\ &-g_4^2
    t^5 + h_{10} t^4 + h_{12} t^3 + h_{14} t^2 + h_{16} t +
    h_{18} \end{aligned}} \\[1pt] \hline $E_8$ & $10$ &
\RootEquation{\begin{aligned} y^2 = x^3 &+ (c_4 t^4 + c_{10} t^3 +
    c_{16} t^2 + c_{22} t + c_{28}) x \\ &+ t^7 + c_{12} t^5 + c_{18}
    t^{4} + c_{24} t^3 + c_{30} t^2 + c_{36} t + c_{42} \end{aligned}}
\\[1pt] \hline
\end{longtable}
\endgroup
\section{A two-neighbor construction for \texorpdfstring{\(U\oplus D_4(-1)^{\oplus 2} \cong U(2)\oplus D_8(-1)\)}{U plus two copies of D4(-1)}}
\label{App:2NS}
We describe the two-neighbor construction which transforms the
Jacobian elliptic fibration realizing the lattice polarization
$U\oplus D_4(-1)^{\oplus 2}$ into a genus-one fibration realizing the
polarization $U(2)\oplus D_8(-1)$. Let $X$ be a K3 surface given by
the Weierstrass equation
\begin{equation}
\label{eq:U2D4-Weierstrass}
\begin{split}
y^2&=x^3+u^2 \big(a_0+a_1u+a_2u^2+a_3u^3+a_4u^4\big) x\\ &+u^3
\big(b_0+b_1 u+b_2 u^2+b_3 u^3+b_4u ^4+b_5 u^5+b_6 u^6\big).
\end{split}
\end{equation}
For a general choice of the coefficients, the Jacobian elliptic
fibration $\pi:X\rightarrow\mathbb P^1_{(u)}$ has singular fibers
$2\mathrm I_0^*+12\mathrm I_1$ and trivial Mordell--Weil group. We
construct a second fibration on $X$ by choosing suitable components of
the two $\mathrm I_0^*$ fibers. At $u=0$, the leading cubic associated
with the $\mathrm I_0^*$ fiber is
\begin{equation}
\label{eq:alpha-root}  
z^3+a_0z+b_0,
\end{equation}
and we choose a root $\alpha$ of this cubic. Similarly, at $u=\infty$,
the leading cubic is
\begin{equation}
\label{eq:beta-root}  
z^3+a_4z+b_6,
\end{equation}
and we choose a root $\beta$. Geometrically, the choices of $\alpha$
and $\beta$ determine distinguished simple components in the two
fibers. We introduce a new parameter $s$ by considering the pencil
\begin{equation}
\label{eq:neighbor-pencil}
x=u\left(\alpha+s u+\beta u^2\right) \qquad \Leftarrow \quad
s=\frac{x-\alpha u-\beta u^3}{u^2}.
\end{equation}
The choice of the linear terms $\alpha u$ and $\beta u^3$ forces the
pencil to have the prescribed incidence with the chosen components of
the fibers over $u=0$ and $u=\infty$. Substituting
\begin{equation*} x=u\left(\alpha+s u+\beta u^2\right), \qquad y=u^2w \end{equation*}
into \eqref{eq:U2D4-Weierstrass}, we obtain
\begin{equation*} u^4w^2 = u^3\left[ \left(\alpha+su+\beta u^2\right)^3 + \big(a_0+\dots +a_4u^4\big) \left(\alpha+s u+\beta u^2\right) + \big(b_0+\dots +b_6 u^6\big) \right], \end{equation*}
whence
\begin{equation}
\label{eq:quartic-form-pre}
w^2 = \frac{ \left(\alpha+su+\beta u^2\right)^3 + \big(a_0+\dots
  +a_4u^4\big) \left(\alpha+s u+\beta u^2\right) + \big(b_0+\dots +b_6
  u^6\big)} {u}.
\end{equation}
Because $\alpha$ and $\beta$ are roots of
Equations~\eqref{eq:alpha-root} and \eqref{eq:beta-root},
respectively, the numerator in \eqref{eq:quartic-form-pre} has
vanishing constant term and vanishing coefficient of $u^6$. Therefore,
the right-hand side is a polynomial of degree four in
$u$. Equation~\eqref{eq:quartic-form-pre} defines a genus-one
fibration $\varphi\colon X\rightarrow \mathbb P^1_{(s)}$ whose generic
fiber is a double cover of the $u$-line branched over a quartic. The
new fibration is given by
\begin{equation}
\label{eq:quartic-model}
w^2 = A_4(s)+A_3(s)u+A_2(s)u^2+A_1(s)u^3+A_0(s)u^4,
\end{equation}
with
\begin{equation}
\begin{split}
A_4(s) &= (3\alpha^2+a_0)s+a_1\alpha+b_1, \\ A_3(s) &= 3\alpha s^2+a_1
s +3\alpha^2\beta+a_0\beta+a_2\alpha+b_2, \\ A_2(s) &=
s^3+(6\alpha\beta+a_2)s+a_1\beta+a_3\alpha+b_3, \\ A_1(s) &= 3\beta
s^2+a_3s +3\alpha\beta^2+a_2\beta+a_4\alpha+b_4, \\ A_0(s) &=
(3\beta^2+a_4)s+a_3\beta+b_5.
\end{split}
\label{eq:A4}
\end{equation}
The old elliptic fiber class intersects the fibers of $\varphi$ in
degree two and therefore supplies a bisection. We compute the
binary-quartic invariants \(\mathcal I(s)\) and \(\mathcal J(s)\) for
\eqref{eq:quartic-model} using Equation~\eqref{eq:I-J-of-s}.  The
relative Jacobian fibration has a fiber of type \(\mathrm I_4^*\) at
infinity, and fourteen simple finite roots of the discriminant give
fourteen fibers of type \(\mathrm I_1\)
\cite{MirandaEllipticSurfaces}.
\par For the very general member, the original Jacobian fibration has
two fibers of type \(\mathrm I_0^*\), twelve fibers of type \(\mathrm
I_1\), and trivial Mordell--Weil group. The Shioda--Tate formula then
gives \cite{ShiodaMordellWeil}
\begin{equation*} \operatorname{NS}(X)\cong U\oplus D_4(-1)^{\oplus2}\cong U(2)\oplus D_8(-1). \end{equation*}
The bisection shows that the divisibility of the new primitive fiber
class is either one or two. If it were one, then \(\varphi\) would
have a section, and its fiber, zero section, and \(\mathrm
I_4^*\)-components would generate the full-rank trivial lattice
\(U\oplus D_8(-1)\), of discriminant \(4\), inside
\(\operatorname{NS}(X)\), of discriminant \(16\). This is impossible
because a finite-index inclusion of lattices satisfies
\begin{equation*} \lvert\operatorname{disc}(U\oplus D_8(-1))\rvert=[\operatorname{NS}(X):U\oplus D_8(-1)]^2\lvert\operatorname{disc}\operatorname{NS}(X)\rvert. \end{equation*}
Thus the new fiber class has divisibility two. Hence \(\varphi\) has
no section for the very general member, and the old fiber class is its
natural bisection.
\section{The Hashimoto--Ueda family}
\label{sec:Hashimoto-Ueda}
\subsection{Orthogonal modular forms and the modular Jacobian}
Let $L$ be an even lattice of signature $(2,n)$, let
\(\Gamma\subset\mathrm O^+(L)\) be an arithmetic subgroup, and let
$\mathcal{D}_L$ be a connected component of
\begin{equation*} \bigl\{[Z]\in\mathbb P(L\otimes\mathbb C):(Z,Z)=0, \ (Z,\overline Z)>0\bigr\}, \end{equation*}
and let
\begin{equation*} \mathcal A_L=\{Z\in L\otimes\mathbb C:[Z]\in\mathcal{D}_L\} \end{equation*}
be its affine cone. Fix a primitive isotropic vector $c\in L$ and the
corresponding affine Siegel-domain slice
\begin{equation*} \mathcal S_c=\{Z\in\mathcal A_L:(Z,c)=1\}, \end{equation*}
and fix tube-domain coordinates \(z_1,\ldots,z_n\) arising from a
chosen linear basis of \(c^\perp/c\) at a cusp on $\mathcal
S_c$. Suppose that $f_0,\ldots,f_n$ are modular forms of weights
$k_0,\ldots,k_n$ and characters $\chi_0,\ldots,\chi_n$. Their modular
Jacobian is defined to be
\begin{equation}
\label{eq:general-modular-jacobian}
J(f_0,\ldots,f_n) = \det
\begin{pmatrix}
 k_0f_0&k_1f_1&\cdots&k_nf_n\\ \partial_{z_1}f_0&\partial_{z_1}f_1&\cdots&\partial_{z_1}f_n\\ \vdots&\vdots&&\vdots\\ \partial_{z_n}f_0&\partial_{z_n}f_1&\cdots&\partial_{z_n}f_n
\end{pmatrix}.
\end{equation}
With the convention $f(\lambda Z)=\lambda^{-k}f(Z)$, the first row is
the negative of the radial Euler derivative on $\mathcal A_L$. The
sign is chosen to match the standard modular-Jacobian
convention. Replacing the chosen linear tube-domain basis changes the
determinant only by a nonzero constant.
\begin{proposition}
\label{prop:Jacobian-general}
The determinant \eqref{eq:general-modular-jacobian} is a modular form
of weight
\begin{equation*} k_J=k_0+\cdots+k_n+n \end{equation*}
and character
\begin{equation*} \chi_J=\chi_0\cdots\chi_n\det. \end{equation*}
If all $\chi_j$ are trivial and a reflection $\sigma_r\in\Gamma$ fixes
the hyperplane $r^\perp$, then $J(f_0,\ldots,f_n)$ vanishes on
$r^\perp$.
\end{proposition}
\begin{proof}
Let \(j(\gamma,z)\) be the orthogonal automorphy factor. The
transformation law follows by differentiating
\begin{equation*} f_j(\gamma z)=\chi_j(\gamma)j(\gamma,z)^{k_j}f_j(z). \end{equation*}
In each derivative row, the terms involving derivatives of
\(j(\gamma,z)\) are multiples of the first row and are removed by
elementary row operations.Taking determinants and using the standard tangent-Jacobian identity
for a type-IV tube domain gives the factor
\(j(\gamma,z)^{k_0+\cdots+k_n+n}\), the product
\(\chi_0\cdots\chi_n\), and the determinant character. This proves the
stated weight and character. A reflection reverses the normal
direction to its fixed hyperplane, so the Jacobian is odd in a local
normal coordinate and vanishes there.
\end{proof}
\subsection{The Weierstrass equation}
\label{ssec:WeqUE8}
Consider the family of \(U\oplus E_8(-1)\)-polarized K3 surfaces. The
very general member admits, up to automorphism, a unique Jacobian
elliptic fibration with singular fibers \(\mathrm{II}^*+14\mathrm
I_1\) \cite{HashimotoUeda,ClingherMalmendierWilliams2026}. A
Weierstrass model over \(\mathbb{P}^1_{(t)}\) can be normalized as
\begin{equation}
\label{eq:Hashimoto-Ueda-from-Vinberg}
y^2=x^3+f(t)x+g(t)
\end{equation}
with
\begin{equation}
\label{eq:HU-polynomial-shapes}
\begin{aligned}
f(t) &=C_4t^4+C_{10}t^3+C_{16}t^2+C_{22}t+C_{28}, \\ g(t)
&=t^7+C_{12}t^5+C_{18}t^4+C_{24}t^3 +C_{30}t^2+C_{36}t+C_{42}.
\end{aligned}
\end{equation}
This family was investigated by Hashimoto and Ueda in
\cite{HashimotoUeda}. The coefficients \(C_n\) are assigned weight
\(n\), and \(t\) is assigned weight \(6\). With these conventions
every monomial in \(f\) has weight \(28\), while every monomial in
\(g\) has weight \(42\). The discriminant polynomial
\begin{equation*} \Delta(t):=4f(t)^3+27g(t)^2 \end{equation*}
is weighted homogeneous of total weight \(84\). Since \(\Delta(t)\)
has degree \(14\) in \(t\), we may write
\begin{equation*} \Delta(t)=27\prod_{i=1}^{14}(t-r_i). \end{equation*}
Under the scaling $C_n\longmapsto \lambda^n C_n$, one has
$t\longmapsto \lambda^6 t$, $r_i\longmapsto \lambda^6 r_i$. Therefore,
\begin{equation*} \operatorname{Disc}_t(\Delta) = 27^{26}\prod_{i<j}(r_i-r_j)^2, \end{equation*}
has weight \(1092\). In turn, the resultant
\(\operatorname{Res}_t(f,g)\) has weight \(196\), as seen from the
Sylvester determinant. Geometrically, the divisor defined by
\(\operatorname{Disc}_t(\Delta)\) has two irreducible components: one
corresponds to the collision of two nodal fibers, producing an
additional \(A_1\)-singularity, and the other to the simultaneous
vanishing of \(f\) and \(g\), producing a \(\mathrm{II}\)-fiber. Thus
one writes
\begin{equation*} \operatorname{Disc}_t(\Delta)=\operatorname{Res}_t(f,g)^3\, \Psi_{504}, \end{equation*}
where \(\Psi_{504}\) is the equation of the divisor parameterizing
surfaces acquiring an additional \(A_1\)-singularity; its weight is
\(1092-3\cdot196=504\) \cite[Section~3]{HashimotoUeda}.
\par Hashimoto and Ueda proved in \cite{HashimotoUeda} that the
scalar-valued ring of orthogonal modular forms for $\Gamma:=\mathrm
O^+(U^{\oplus2}\oplus E_8(-1))$ is
\begin{equation}
\label{eq:HU-scalar-ring}
M_*(\Gamma) = \mathbb
C[C_4,C_{10},C_{12},C_{16},C_{18},C_{22},C_{24},C_{28},
  C_{30},C_{36},C_{42}],
\end{equation}
and that the full ring with characters is obtained by adjoining a
determinant-character generator $\Phi_{252}$ satisfying one relation
of weight $504$, namely $\Phi_{252}^2=\Psi_{504}$; see also
\cite{FreitagSalvatiManni,HashimotoUeda,YoshikawaII}. 

Shiga proved that the period map identifies the weighted projective
coefficient space associated with the generators
in~\eqref{eq:HU-scalar-ring} with the Satake--Baily--Borel
compactification of the modular variety for~\(\Gamma\); see
\cite[Theorem~1.1]{HashimotoUeda}; see also results by
Brieskorn  \cite[Theorem~5]{BrieskornUnfolding}.  Applying
Proposition~\ref{prop:Jacobian-general} to the eleven scalar
generators, we obtain the following.
\begin{theorem}
\label{thm:HU-Jacobian}
The modular Jacobian $J=J(C_4,\ldots,C_{42})$ of the Hashimoto--Ueda
generators has weight $252$ and determinant character. There is a
constant $\kappa\in\mathbb C^\times$ such that
$J=\kappa\Phi_{252}$. The scalar equation of the reflective branch
divisor is $\Psi_{504}=\Phi_{252}^2=0$.
\end{theorem}
\begin{proof}
The weight and character follow from
Proposition~\ref{prop:Jacobian-general}. The forms in
\eqref{eq:HU-scalar-ring} are algebraically independent, so their
modular Jacobian is nonzero. Hashimoto--Ueda's description of the full
graded ring shows that the determinant-character module in weight
$252$ is one-dimensional and generated by $\Phi_{252}$. Hence
$J=\kappa\Phi_{252}$ for some $\kappa\neq0$, and squaring gives the
stated branch equation.
\end{proof}
\section{Relation with coefficients in Vinberg's scroll model}
\label{App:VinbergCoefficients}
Expanding \eqref{eq:cubic-translation-identity} gives the following
polynomial change of parameters from Vinberg's coefficients
to those of \eqref{eq:H+D8-weierstrass_intro}:
\begin{equation}
\label{eq:vinberg-to-FGH}
\begin{aligned}
F_4&=-\frac32a_2, & F_6&=-\frac32f_2, \\[1mm] G_8&=f_1+\frac34a_2^2, &
G_{10}&=a_1+\frac32a_2f_2, & G_{12}&=a+\frac34f_2^2, \\[1mm]
H_8&=b_2-\frac14a_2^2, & H_{10}&=g_2-\frac12a_2f_2, \\[1mm]
H_{12}&=h-\frac12a_2f_1-\frac14f_2^2-\frac18a_2^3,
\\ H_{14}&=g_1-\frac12(a_2a_1+f_2f_1)-\frac38a_2^2f_2,
\\ H_{16}&=b_1-\frac12(a_2a+f_2a_1)-\frac38a_2f_2^2,
\\ H_{18}&=b-\frac12f_2a-\frac18f_2^3.
\end{aligned}
\end{equation}
The change of parameters is triangular.  Its inverse is
\begin{equation}
\label{eq:FGH-to-vinberg}
\begin{aligned}
a_2&=-\frac23F_4, & f_2&=-\frac23F_6, \\[1mm] f_1&=G_8-\frac13F_4^2, &
a_1&=G_{10}-\frac23F_4F_6, & a&=G_{12}-\frac13F_6^2, \\[1mm]
b_2&=H_8+\frac19F_4^2, & g_2&=H_{10}+\frac29F_4F_6, \\[1mm]
h&=H_{12}-\frac13F_4G_8+\frac19F_6^2 +\frac{2}{27}F_4^3,
\\ g_1&=H_{14}-\frac13(F_6G_8+F_4G_{10}) +\frac29F_4^2F_6,
\\ b_1&=H_{16}-\frac13(F_6G_{10}+F_4G_{12}) +\frac29F_4F_6^2,
\\ b&=H_{18}-\frac13F_6G_{12}+\frac{2}{27}F_6^3.
\end{aligned}
\end{equation}
\section{Coefficient dictionary for the exceptional
\texorpdfstring{$E_7$}{E7} model}
\label{app:E7-coefficient-dictionary}
We expand Equations~\eqref{eq:E7-BC-from-invariants} and
\eqref{eq:E7-polynomial-shapes} in the coefficients of the $A_7$
model.  We obtain:
\begin{align}
P_4={}&f_4-4g_4, & P_6={}&f_6+2g_6,
\label{eq:E7-dictionary-low}\\
P_{10}={}&f_4g_6-4f_6g_4-4h_5^2+g_{10},
\nonumber\\ P_{12}={}&f_4^2g_4-\frac13f_4g_8+2f_6g_6
-\frac83g_4g_8+g_6^2+4h_5h_7+g_{12},
\nonumber\\ P_{14}={}&h_7^2-4h_5h_9,
\nonumber\\ P_{16}={}&2f_4h_5h_7-4f_6h_5^2-4g_4g_{12}
+g_6g_{10}-\frac13g_8^2+2h_7h_9,
\nonumber\\ P_{18}={}&f_4^2h_5^2+2f_4g_4g_{10}
-\frac13f_4g_6g_8-2f_4h_5h_9 -\frac83f_6g_4g_8+f_6g_6^2
\nonumber\\ &+4f_6h_5h_7-4g_4h_7^2+4g_6h_5h_7
-\frac83g_8h_5^2+2g_6g_{12} -\frac13g_8g_{10}+h_9^2,
\nonumber\\ P_{22}={}&-4g_4h_9^2+2g_6h_7h_9
-\frac43g_8h_5h_9-\frac23g_8h_7^2 +2g_{10}h_5h_7-4g_{12}h_5^2.
\label{eq:E7-dictionary-middle}
\end{align}
The two highest-weight coefficients are
\begin{align}
P_{24}={}&4f_4g_4h_7h_9-2f_4g_6h_5h_9
-\frac23f_4g_8h_5h_7+2f_4g_{10}h_5^2 -4f_6g_4h_7^2
\nonumber\\ &+4f_6g_6h_5h_7-\frac83f_6g_8h_5^2
-\frac83g_4g_8g_{12}+g_4g_{10}^2+g_6^2g_{12}
\nonumber\\ &-\frac13g_6g_8g_{10}+2g_6h_9^2
+\frac2{27}g_8^3-\frac23g_8h_7h_9 -2g_{10}h_5h_9+4g_{12}h_5h_7,
\label{eq:E7-dictionary-P24}\\
P_{30}={}&-\frac83g_4g_8h_9^2+4g_4g_{10}h_7h_9
-4g_4g_{12}h_7^2+g_6^2h_9^2 -\frac23g_6g_8h_7h_9
\nonumber\\ &-2g_6g_{10}h_5h_9+4g_6g_{12}h_5h_7
+\frac89g_8^2h_5h_9+\frac19g_8^2h_7^2
\nonumber\\ &-\frac23g_8g_{10}h_5h_7-\frac83g_8g_{12}h_5^2
+g_{10}^2h_5^2.
\label{eq:E7-dictionary-P30}
\end{align}
These identities provide a coefficient-level check of the weight
decomposition in Equation~\eqref{eq:E7-universal-model}.  Under the
specialization \eqref{eq:A5A1-to-A7-specialization}, coefficients
factor as in Equation~\eqref{eq:source-proof-rank-one-relation}.
\section{Explicit \texorpdfstring{$D_8$}{D8} Jacobi coefficients}
\label{app:D8-Jacobi-normalizations}
In this appendix we fix the normalization of the $D_8$ Jacobi forms
used in Section~\ref{subsec:two-E-string-limits}.  We use the notation
\[
 D_8=\{(n_1,\ldots,n_8)\in\mathbb Z^8: \textstyle\sum_i n_i\in2\mathbb
 Z\}
\]
with its Euclidean form.  The discriminant group has four elements:
the zero class, two nonzero isotropic spinor classes, denoted by \(1\)
and \(2\), and the vector class \(v\). The vector class contains a vector of norm one; thus for the discriminant form used we have
\(q(v)=1\bmod 2\mathbb Z\). Equivalently, in the convention
\(Q(x)=(x,x)/2\bmod\mathbb Z\), one has \(Q(v)=1/2\).
\par For $a\in D_8^\vee/D_8$, set
\[
 \Theta_a(\tau,\mathbf z)=\sum_{\ell\in a+D_8}
 q^{(\ell,\ell)/2}e^{2\pi i(\ell,\mathbf z)}, \qquad q=e^{2\pi i\tau}.
\]
We choose the spinor labels so that
$\Theta_1-\Theta_2=\prod_{i=1}^8\vartheta(\tau,z_i)$ with
\[
 \vartheta(\tau,z)=q^{1/8}\sum_{n\in\mathbb Z}
 (-1)^nq^{n(n+1)/2}e^{2\pi i(n+1/2)z}.
\]
For the index-two generators, we use the classical weak Jacobi forms
\begin{equation}
\label{eq:D8-classical-Jacobi-conventions}
 \chi_{-2}(\tau,z)=\frac{\vartheta(\tau,z)^2}{\eta(\tau)^6}, \qquad
 \chi_0(\tau,z)=-\frac{3}{\pi^2}\chi_{-2}(\tau,z)\wp(\tau,z),
\end{equation}
where $\wp$ has periods $1,\tau$ and principal part $z^{-2}$.  Their
constant $q$-terms are, respectively, $\zeta+\zeta^{-1}-2$ and
$\zeta+\zeta^{-1}+10$.  Define
\begin{equation}
\label{eq:D8-symmetric-Jacobi-polynomials}
 s_j=\sum_{\substack{J\subset\{1,\ldots,8\}\\|J|=j}} \prod_{i\notin
   J}\chi_{-2}(\tau,z_i) \prod_{i\in J}\chi_0(\tau,z_i), \qquad 0\leq
 j\leq5.
\end{equation}
Thus \(s_j\) is a \(D_8\)-Jacobi form of weight \(-16+2j\) and index
two.  The character generator is
\begin{equation}
\label{eq:D8-h-theta}
 h=72\frac{\Theta_1-\Theta_2}{\Delta}, \qquad h^2=5184s_0.
\end{equation}
\par Next we specify the index-one generators.  Define
\[
 \mathbf c_r=(c_{r,0},c_{r,s},c_{r,v}), \qquad r\in\{-4,-2,0\}.
\]
The second subscript specifies the theta component.  In particular,
$c_{r,s}$ is the coefficient shared by the two spinor theta series
$\Theta_1$ and $\Theta_2$; the letter $s$ denotes ``spinor''.  The
initial vector is
\begin{equation}
\label{eq:D8-eta-coefficients}
\begin{aligned}
 c_{-4,0}&=256\frac{\eta(2\tau)^{16}}{\eta(\tau)^{32}},\\ c_{-4,s}&=\frac{1}{\eta(\tau)^8\eta(2\tau)^8},\\ c_{-4,v}&=-\frac{16}{\eta(\tau)^8}
 \left(\frac{1}{\eta(\tau/2)^8} +\frac{e^{\pi
     i/3}}{\eta((\tau+1)/2)^8}\right).
\end{aligned}
\end{equation}
Let $D_w=q\frac{d}{dq}-\frac{w}{12}E_2$ be the Serre derivative,
acting componentwise.  The remaining coefficient vectors and the
Jacobi forms are defined by
\begin{equation}
\label{eq:D8-Jacobi-generator-definition}
\begin{aligned}
 \mathbf c_{-2}&=-3D_{-8}\mathbf c_{-4},& \mathbf c_0&=-6D_{-6}\mathbf
 c_{-2}-2E_4\mathbf
 c_{-4},\\ \gamma_r&=c_{r,0}\Theta_0+c_{r,s}(\Theta_1+\Theta_2)
 +c_{r,v}\Theta_v, &&r=-4,-2,0.
\end{aligned}
\end{equation}
Here the entries of $\mathbf c_r$ have modular weight $r-4$, while
$\Theta_a$ has weight four; hence $\gamma_r$ has weight $r$ and index
one.  The three entries of $\mathbf c_{-4}$ begin with
\[
256+8192q+\cdots,\qquad q^{-1}+8+52q+\cdots, \qquad
-32q^{-1/2}-1664q^{1/2}+\cdots.
\]
The spinor entries of $\mathbf c_{-2}$ and $\mathbf c_0$ begin with
$q^{-1}+32+268q+\cdots$ and $q^{-1}-520-8156q+\cdots$, respectively.
With these conventions,
\begin{equation}
\label{eq:D8-index-one-leading-explicit}
 g_8=1296\gamma_{-4},\qquad g_{10}=-2592\gamma_{-2},\qquad
 g_{12}=1296\gamma_0.
\end{equation}
Finally, an expansion of the index-two generators gives
\begin{equation}
\label{eq:D8-index-two-leading-explicit}
\begin{aligned}
 h_8&=-5184s_0=-h^2,\\ h_{10}&=5184s_1,\\ h_{12}&=-5184s_2-15552E_4s_0,\\ h_{14}&=5184s_3+15552E_4s_1+10368E_6s_0,\\ h_{16}&=-5184s_4-15552E_4s_2-10368E_6s_1
 -46656E_4^2s_0+\tfrac14g_8^2,\\ h_{18}&=5184s_5+15552E_4s_3+10368E_6s_2
 +46656E_4^2s_1+62208E_4E_6s_0 +\tfrac12g_8g_{10}.
\end{aligned}
\end{equation}
Here $h_8=\Delta^{-2}[p^2]H_8$.  For the structure of the ring of weak \(W(D_8)\)-invariant Jacobi
forms, see \cite[Corollary~4.1]{AdlerGritsenkoD8}.

\section{Selected formulas for the coefficient map 
\texorpdfstring{$\Theta$}{Theta }}
\label{app:C-formulas}
We provide some of the relations for the geometric coefficient map
$\Theta$ in Proposition~\ref{prop:polynomial_map}. Computing the
coefficients $C_k$ from the binary-quartic invariants yields
\begin{align}
C_{16} ={}& \frac{6}{49}F_{6}^{2}F_{4} +\frac{96}{49}F_{6}^{2}H_{4}
-\frac{3}{7}F_{6}G_{10} -\frac{16}{7}F_{6}H_{10} -7G_{12}H_{4}+H_{16}
\nonumber\\ & -\frac{1}{3}F_{4}^{2}H_{4}^{2}
-\frac{1}{3}F_{4}H_{4}G_{8} +H_{4}^{2}G_{8} +F_{4}H_{12} +3H_{4}H_{12}
-\frac{1}{3}G_{8}^{2},
\label{eq:appendix-C16}
\end{align}
\begin{align}
C_{22} ={}& -\frac{4}{3}F_{4}H_{4}H_{14} +\frac{8}{7}F_{6}H_{4}H_{12}
-\frac{1}{3}F_{4}G_{10}H_{4}^{2} \nonumber\\ &
+\frac{1}{3}G_{8}G_{10}H_{4} +\frac{2}{3}F_{4}F_{6}H_{4}^{3}
+\frac{22}{21}F_{6}G_{8}H_{4}^{2} \nonumber\\ & -G_{10}H_{4}^{3}
+G_{10}H_{12} -2H_{4}^{2}H_{14} -\frac{2}{3}G_{8}H_{14} \nonumber\\ &
+\frac{3}{49}F_{6}^{2}G_{10} -\frac{4}{343}F_{6}^{3}F_{4}
+2F_{6}G_{12}H_{4} \nonumber\\ & +\frac{2}{21}F_{6}F_{4}^{2}H_{4}^{2}
-\frac{2}{7}F_{6}F_{4}H_{12} -\frac{106}{343}F_{6}^{3}H_{4}
\nonumber\\ & +\frac{44}{49}F_{6}^{2}H_{10} -\frac{2}{7}F_{6}H_{16}
+\frac{2}{21}F_{6}G_{8}^{2} \nonumber\\ & -4G_{12}H_{10} -8H_{18}H_{4}
+\frac{2}{21}F_{6}F_{4}H_{4}G_{8}.
\label{eq:appendix-C22}
\end{align}
\begin{align}
C_{28} ={}& -\frac{1}{7}F_{6}G_{10}H_{12}
+\frac{2}{21}F_{6}G_{8}H_{14} +F_{4}G_{12}H_{4}^{3} \nonumber\\ &
+F_{4}H_{4}^{2}H_{16} +G_{8}G_{12}H_{4}^{2} +G_{8}H_{4}H_{16}
\nonumber\\ & +G_{12}H_{4}H_{12} +G_{12}H_{4}^{4} +H_{4}^{3}H_{16}
+H_{12}H_{16} \nonumber\\ & -\frac{1}{3}H_{14}^{2}
+\frac{1}{49}F_{6}^{2}F_{4}H_{12} -\frac{11}{49}F_{6}^{2}H_{4}H_{12}
\nonumber\\ & +\frac{4}{7}F_{6}G_{12}H_{10}
-\frac{11}{21}F_{6}G_{10}H_{4}^{3} -\frac{8}{21}F_{6}H_{4}^{2}H_{14}
\nonumber\\ & -\frac{2}{3}G_{10}H_{4}H_{14}
-\frac{1}{7}F_{6}^{2}G_{12}H_{4}
-\frac{1}{147}F_{6}^{2}F_{4}^{2}H_{4}^{2} \nonumber\\ &
-\frac{25}{147}F_{6}^{2}H_{4}^{2}G_{8} -\frac{1}{3}G_{10}^{2}H_{4}^{2}
-\frac{1}{147}F_{6}^{2}F_{4}H_{4}G_{8} \nonumber\\ &
+\frac{1}{21}F_{6}F_{4}G_{10}H_{4}^{2}
-\frac{1}{21}F_{6}G_{8}G_{10}H_{4} +\frac{4}{21}F_{6}F_{4}H_{4}H_{14}
\nonumber\\ & +\frac{8}{7}F_{6}H_{18}H_{4}
-\frac{2}{21}F_{4}F_{6}^{2}H_{4}^{3} -4H_{18}H_{10} \nonumber\\ &
-\frac{1}{3}F_{6}^{2}H_{4}^{4} +\frac{1}{2401}F_{6}^{4}F_{4}
-\frac{1}{343}F_{6}^{3}G_{10} \nonumber\\ &
+\frac{37}{2401}F_{6}^{4}H_{4} +\frac{1}{49}F_{6}^{2}H_{16}
-\frac{24}{343}F_{6}^{3}H_{10} \nonumber\\ &
-\frac{1}{147}F_{6}^{2}G_{8}^{2}.
\label{eq:appendix-C28}
\end{align}
\section{Eisenstein expressions for the Hashimoto--Ueda generators}
\label{app:Eisenstein-generator-identities}
We give the relations between the generators used in
Theorem~\ref{thm:Sakai-comparison} and the Eisenstein series.
Throughout this appendix, $\mathcal E_k$ denotes the cusp-normalized
orthogonal Eisenstein series of
\eqref{eq:Sakai-versus-normalized-Eisenstein}.  These are not Sakai's
series $\mathcal E_k^{\mathrm S}$ in \cite{SakaiFTheoryK3}; the two
normalizations differ by the Bernoulli factor in that equation.
Direct computation gives
\begin{align}
C_4={}&-3\mathcal E_4,\nonumber\\ C_{10}={}&\frac{24}{7}\mathcal
E_{10}, \nonumber\\ C_{12}={}&-\frac{4146}{875}\mathcal E_{12}
+\frac{378}{125}\mathcal E_4^3,
\label{eq:Eisenstein-C4-C22}\\
C_{16}={}&-\frac{22494123}{1990625}\mathcal E_{16}
+\frac{3756276}{284375}\mathcal E_4\mathcal E_{12}
-\frac{10557}{3125}\mathcal E_4^4,
\nonumber\\ C_{18}={}&\frac{175468}{13475}\mathcal E_{18}
-\frac{3132}{275}\mathcal E_4^2\mathcal E_{10},
\nonumber\\ C_{22}={}&\frac{245355851592}{6407025625}\mathcal E_{22}
-\frac{528963264}{48173125}\mathcal E_{10}\mathcal E_{12}
-\frac{469552368}{10058125}\mathcal E_4\mathcal E_{18}
+\frac{10401912}{529375}\mathcal E_4^3\mathcal E_{10}.  \nonumber
\end{align}
In weights $24$, $28$, and $30$, one has
\begin{align}
C_{24}={}& -\frac{247630621761606}{5887382125625}\mathcal E_{24}
+\frac{88284326976}{10221859375}\mathcal E_{12}^2
+\frac{6061824}{468391}\mathcal E_4\mathcal E_{10}^2
\nonumber\\ &+\frac{119646240237}{2673409375}\mathcal E_4^2\mathcal
E_{16} -\frac{3193630632}{112328125}\mathcal E_4^3\mathcal E_{12}
+\frac{4270941}{1234375}\mathcal E_4^6,
\label{eq:Eisenstein-C24}\\[3pt]
C_{28}={}& -\frac{28869530234706049572}{212555734025009375}\mathcal
E_{28} +\frac{47840342617727064}{1650995262109375}\mathcal
E_{12}\mathcal E_{16} +\frac{3191204229888}{84553024325}\mathcal
E_{10}\mathcal E_{18}
\nonumber\\ &+\frac{4554670026061219158}{25250981936805625}\mathcal
E_4\mathcal E_{24} -\frac{21493441243002816}{1096038871484375}\mathcal
E_4\mathcal E_{12}^2 -\frac{258000273792}{7174746425}\mathcal
E_4^2\mathcal E_{10}^2
\nonumber\\ &-\frac{134448206960657241}{1433281601171875}\mathcal
E_4^3\mathcal E_{16} +\frac{70065352514484}{1720626171875}\mathcal
E_4^4\mathcal E_{12} -\frac{2446356636099}{926491015625}\mathcal
E_4^7,
\label{eq:Eisenstein-C28}\\[3pt]
C_{30}={}& \frac{27209419519503579144}{186144949274659375}\mathcal
E_{30} -\frac{25369042718016}{562070884375}\mathcal E_{12}\mathcal
E_{18} -\frac{392398848}{88110869}\mathcal E_{10}^3
\nonumber\\ &-\frac{1899223793136}{20116221125}\mathcal E_4\mathcal
E_{10}\mathcal E_{16} -\frac{396004344469488}{2187752826875}\mathcal
E_4^2\mathcal E_{22} +\frac{4032113824512}{46487065625}\mathcal
E_4^2\mathcal E_{10}\mathcal E_{12}
\nonumber\\ &+\frac{8765776266336}{67942634375}\mathcal E_4^3\mathcal
E_{18} -\frac{18776847816}{510846875}\mathcal E_4^5\mathcal E_{10}.
\label{eq:Eisenstein-C30}
\end{align}
The two longest identities are
\begin{align}
C_{36}={}& -\frac{10966101442664558639832373507186}
{20664063689028667667525390625}\mathcal E_{36}
+\frac{711906653618754304}{11152800682340625}\mathcal E_{18}^2
\nonumber\\ &+\frac{9952374401195974480016}{74473435858305859375}
\mathcal E_{12}\mathcal E_{24}
-\frac{16025286253876891776}{2198153360634765625}\mathcal E_{12}^3
\nonumber\\ &+\frac{173837469469737984}{3793119345640625} \mathcal
E_{10}^2\mathcal E_{16}
+\frac{17805779546323923071901}{108248886284423828125} \mathcal
E_4\mathcal E_{16}^2
\nonumber\\ &+\frac{386071447810308586944}{1072560276146734375}
\mathcal E_4\mathcal E_{10}\mathcal E_{22}
-\frac{442884732896131584}{5539866741784375} \mathcal E_4\mathcal
E_{10}^2\mathcal E_{12}
\nonumber\\ &+\frac{316115129874677291208708972}{416187922859075945703125}
\mathcal E_4^2\mathcal E_{28}
-\frac{90605903434364369268733632}{308219146123005126953125} \mathcal
E_4^2\mathcal E_{12}\mathcal E_{16}
\nonumber\\ &-\frac{830577945310736931456}{2197903021688171875}
\mathcal E_4^2\mathcal E_{10}\mathcal E_{18}
-\frac{36320037840362329936000044}{59671226689413792578125} \mathcal
E_4^3\mathcal E_{24}
\nonumber\\ &+\frac{271723776565975828624176}{2590076858176513671875}
\mathcal E_4^3\mathcal E_{12}^2
+\frac{169076009521112340096}{1305521343709515625} \mathcal
E_4^4\mathcal E_{10}^2
\nonumber\\ &+\frac{751395218162920164090549}{3387023583769287109375}
\mathcal E_4^5\mathcal E_{16}
-\frac{135095208176303563806}{1498020160888671875} \mathcal
E_4^6\mathcal E_{12}
\nonumber\\ &+\frac{1582982591047959024}{312773440185546875}\mathcal
E_4^9,
\label{eq:Eisenstein-C36}
\end{align}
and
\begin{align}
C_{42}={}& \frac{937658283184177005411726687689439004}
{472563301052781894820381313671875}\mathcal E_{42} \nonumber\\ &
-\frac{174377132156027221342821058336}
{433495455925792705523203125}\mathcal E_{18}\mathcal E_{24}
\nonumber\\ &-\frac{524243665969361758958944518286208}
            {1258499497774101443340558203125}\mathcal E_{12}\mathcal
            E_{30} \nonumber\\ & +\frac{509119589590898191681843968}
            {9684077944495913161328125}\mathcal E_{12}^2\mathcal
            E_{18} \nonumber\\ &-\frac{255976422600718744575572808}
            {1690659567345019755859375}\mathcal E_{10}\mathcal
            E_{16}^2 \nonumber\\ & -\frac{253979387411191905980928}
            {1522866217812128564375}\mathcal E_{10}^2\mathcal E_{22}
            \nonumber\\ &+\frac{2150945991433083703296}{96081458005145738125}
            \mathcal E_{10}^3\mathcal E_{12} \nonumber\\ &
            -\frac{579685483277498303643935181348}
            {478058855307500586259765625}\mathcal E_4\mathcal
            E_{16}\mathcal E_{22}
            \nonumber\\ &-\frac{9342416050138099539846158310096}
                        {6500132405485958107978484375}\mathcal
                        E_4\mathcal E_{10}\mathcal E_{28}
                        \nonumber\\ &
                        +\frac{9295217211099331001609272699968}
                        {18480648441128142380439453125} \mathcal
                        E_4\mathcal E_{10}\mathcal E_{12}\mathcal
                        E_{16}
                        \nonumber\\ &+\frac{84336513837366609025425408}
                                    {246533341350932314188125}\mathcal
                                    E_4\mathcal E_{10}^2\mathcal
                                    E_{18} \nonumber\\ &
                                    +\frac{10766383697512235671478541672}
                                    {18044266397177518544921875}\mathcal
                                    E_4^2\mathcal E_{16}\mathcal
                                    E_{18}
                                    \nonumber\\ &+\frac{411133600292986721437818995664}
                                                {775339151885274060068359375}\mathcal
                                                E_4^2\mathcal
                                                E_{12}\mathcal E_{22}
                                                \nonumber\\ &
                                                +\frac{725349075991254266821519329528}
                                                {931960907501616498240109375}\mathcal
                                                E_4^2\mathcal
                                                E_{10}\mathcal E_{24}
                                                \nonumber\\ &-\frac{633175334098004178493948224768}
                                                            {3435226416115584113069921875}\mathcal
                                                            E_4^2\mathcal
                                                            E_{10}\mathcal
                                                            E_{12}^2
                                                            \nonumber\\ &
                                                            -\frac{814796153444743328014789939632}
                                                            {446118219700142305331640625}\mathcal
                                                            E_4^3\mathcal
                                                            E_{30}
                                                            \nonumber\\ &-\frac{1545660921032085339748204896}
                                                                        {29265070711388748564453125}\mathcal
                                                                        E_4^3\mathcal
                                                                        E_{12}\mathcal
                                                                        E_{18}
                                                                        \nonumber\\ &
                                                                        -\frac{123243972544099392334848}
                                                                        {2956623058517985250625}\mathcal
                                                                        E_4^3\mathcal
                                                                        E_{10}^3
                                                                        \nonumber\\ &+\frac{73492865380951694535471199362576}
                                                                                    {112305478988394096004208984375}\mathcal
                                                                                    E_4^4\mathcal
                                                                                    E_{10}\mathcal
                                                                                    E_{16}
                                                                                    \nonumber\\ &
                                                                                    +\frac{1999656919995936407243390016}
                                                                                    {1217172922896819560546875}\mathcal
                                                                                    E_4^5\mathcal
                                                                                    E_{22}
                                                                                    \nonumber\\ &-\frac{5620222544686734151056072672}
                                                                                                {12256409362478892939453125}\mathcal
                                                                                                E_4^5\mathcal
                                                                                                E_{10}\mathcal
                                                                                                E_{12}
                                                                                                \nonumber\\ &
                                                                                                -\frac{303046910652213425086168068}
                                                                                                {321594183641634599609375}\mathcal
                                                                                                E_4^6\mathcal
                                                                                                E_{18}
                                                                                                \nonumber\\ &+\frac{275930723848426493011104564}
                                                                                                            {1481543988871074970703125}\mathcal
                                                                                                            E_4^8\mathcal
                                                                                                            E_{10}.
\label{eq:Eisenstein-C42}
\end{align}
Each monomial in these equations has the indicated total weight.
After the normalization change
\eqref{eq:Sakai-versus-normalized-Eisenstein}, we find that the
identities agree term by term with~\cite[Appendix~A]{SakaiFTheoryK3}.
\bibliographystyle{amsplain} \bibliography{vinberg_to_hashimoto_ueda}
\end{document}